\documentclass[11pt]{amsart}

\usepackage[margin=1.15in]{geometry}
\usepackage{amssymb,mathtools}
\newcommand{\setZ}{\mathbb Z}
\providecommand\given{}
\newcommand\SetSymbol[1][]{%
  \nonscript\:#1\vert\allowbreak\nonscript\:\mathopen{}%
}
\DeclarePairedDelimiterX\Set[1]\{\}{%
  \renewcommand\given{\SetSymbol[\delimsize]}#1%
}
\DeclareMathOperator{\rk}{rk}
\DeclareMathOperator{\Mat}{M}
\DeclareMathOperator{\diag}{diag}
\DeclareMathOperator{\id}{id}
\DeclareMathOperator{\Ad}{Ad}
\usepackage{enumitem}
\usepackage{needspace}
\usepackage{xcolor}
\usepackage{tikz}
\usetikzlibrary{arrows.meta,calc}
\tikzset{
  bratteli edge/.style={-{Stealth[length=1.6mm,width=1.1mm]},
  line width=.45pt,shorten >=1.7pt,shorten <=1.7pt}
}

\newtheorem{theorem}{Theorem}[section]
\newtheorem{proposition}[theorem]{Proposition}
\newtheorem{lemma}[theorem]{Lemma}
\newtheorem{corollary}[theorem]{Corollary}
\theoremstyle{definition}
\newtheorem{definition}[theorem]{Definition}
\newtheorem{example}[theorem]{Example}
\theoremstyle{remark}
\newtheorem{remark}[theorem]{Remark}

\usepackage[colorlinks=true, hypertexnames=false,
  linkcolor=blue!55!black, citecolor=blue!55!black,
  urlcolor=blue!55!black]{hyperref}
\usepackage[nameinlink, noabbrev, capitalise]{cleveref}

\title{Uniqueness of Rank-Metric Completions of Bratteli Systems}

\author{Baojie Jiang}
\address{\hskip-\parindent
  Baojie Jiang,
  School of Mathematical Sciences, Chongqing Normal University, University Town, Shapingba District, Chongqing, 401331, China.}
\email{jiangbaojie@gmail.com}

\date{\today}
\subjclass[2010]{Primary 16E50; Secondary 16S50, 16W80, 46L10}
\keywords{Sylvester matrix rank function, Bratteli diagram, rank-metric
completion, ultramatricial algebra}

\hypersetup{
  pdftitle={Uniqueness of Rank-Metric Completions of Bratteli Systems},
  pdfauthor={Baojie Jiang},
  pdfsubject={Sylvester rank completions of Bratteli systems and their
    operator-algebraic interpretation},
  pdfkeywords={Sylvester matrix rank function, Bratteli diagram,
    rank-metric completion, ultramatricial algebra}
}

\begin{document}

\allowdisplaybreaks
\raggedbottom

\begin{abstract}
Let $R$ be a unital ring equipped with a Sylvester matrix rank function
$\rk$.  A harmonic function $\alpha$ on a Bratteli diagram $B$
defines a weighted matrix rank on the associated algebraic direct limit
$A(B,R)$.
We prove that, if $\alpha$ is extreme and the total weight of blocks
of any fixed bounded size tends to zero, then the rank completion of
$A(B,R)$ is isomorphic to $\mathcal M_{R,\rk}$, the rank completion
of the direct system $\Mat_{2^k}(R)$ with connecting maps
$x\mapsto\diag(x,x)$ and normalized ranks $2^{-k}\rk$.
The isomorphism preserves the unital $R$-algebra structure and the
ranks on all rectangular matrices.
The coefficient ring need not be regular, and the specified rank
need not be induced from a regular ring.
We recover factor-sequence uniqueness and construct corners of every
prescribed rank in $(0,1]$ that are isomorphic to $\mathcal M_{R,\rk}$
with their normalized ranks.
Examples show that the coefficient rank can affect the isomorphism
type and that the completion can be non-regular and non-simple.
For complex coefficients, the trace determined by $\alpha$ gives a
rank completion of the associated AF $C^*$-algebra canonically
isomorphic to the affiliated-operator ring of its GNS closure.
The rank completion of the algebraic direct limit can be a proper
subring of this ring.
\end{abstract}

\maketitle
\enlargethispage{2pt}

\section{Introduction}

The hyperfinite $\mathrm{II}_1$ factor $\mathcal R$ originates in
the work of Murray and von Neumann
\cite{MurrayVonNeumann1936,MurrayVonNeumann1943} on rings of operators.
It can be realized as the weak operator closure of
$\bigcup_{n\geq1}\Mat_{2^n}(\mathbb C)$, with inclusions
$x\mapsto
\diag(x,x)$, in the GNS representation of the trace
induced by the normalized matrix traces.
Murray and von Neumann \cite{MurrayVonNeumann1943} proved that every
hyperfinite $\mathrm{II}_1$ factor with separable predual is isomorphic
to $\mathcal R$.
The dimension theory of projections in finite factors also led von
Neumann to continuous geometry \cite{vonNeumann1936Geometry}.

An algebraic counterpart is obtained by completing matrix algebras in
the rank metric.
For a unital ring $R$ with a specified Sylvester matrix rank function $\rk$, put
\[
 \mathcal M_{R,\rk}
 :=\overline{\varinjlim_k
 \left(\Mat_{2^k}(R),\,x\mapsto
 \begin{pmatrix*}[c]
  x & 0\\
  0 & x
 \end{pmatrix*}\right)}^{\,\rk},
\]
where the normalized rank at level $k$ is $2^{-k}\rk$.
For a division ring $D$ with its unique Sylvester matrix rank function, this is
von Neumann's continuous factor $\mathcal M_D$, a complete regular ring whose rank
takes every value in $[0,1]$ \cite{Goodearl1991,Halperin1968}.
Rank completions of ultramatricial algebras also occur in Elek's work
\cite{Elek2013} on regular closures of amenable group algebras.
More recent work studies the rank topology and unitary representations of unit
groups of continuous regular rings
\cite{Schneider2024,Schneider2026,SchneiderThom2026}.

For a division ring $D$, von Neumann proved that the rank completion
is unchanged, up to isomorphism, when the sequence $(2^k)$ is replaced
by a factor sequence $(n_i)$ with $n_i\mid n_{i+1}$ and $n_i\to\infty$.
Halperin \cite{Halperin1968} presented this proof and extended the result
to unital regular rank rings whose rank completions have no nontrivial
central idempotents.
For finite fields, Anderson \cite[Section~4.2]{Anderson2017} gave another
proof of factor-sequence uniqueness using an approximate extension
property and a back-and-forth argument.
Ara and Claramunt \cite[Theorems~2.2 and 3.2]{AraClaramunt2018}
proved that the completion of an
ultramatricial $D$-ring with respect to a non-discrete extremal
pseudo-rank is isomorphic to $\mathcal M_D$ as a $D$-ring.
They also characterized the continuous factor over a field by local matricial
approximation.

We study rank completions of algebraic direct limits of finite
products of matrix algebras over a unital ring $R$.
We allow arbitrary coefficient rings with a specified Sylvester matrix
rank; neither regularity of $R$ nor regularity or simplicity of its
rank completion is assumed.
When $R$ contains a field in its centre with the same identity,
uniqueness follows from
the field case of Ara and Claramunt by coefficient extension and
the correspondence between Sylvester matrix ranks on $R$ and
$\Mat_q(R)$ for $q\geq1$ \cite[Proposition~1.4]{JaikinLopez2020};
see \Cref{prop:central-field-transfer}.
The theorem also applies to ranks not induced from regular rings, such as
the normalized length rank on $\mathbb Z/4\mathbb Z$, for which
$\rk_\ell(2)=1/2$ \cite[Example~2.1.13]{LopezAlvarez2021}; see also
\cite[Remark~4.6]{HungLi2023}.

A Bratteli diagram $B$ records the block multiplicities of the connecting
maps in such a direct system.  Write $A(B,R)$ for the associated
algebraic direct limit over $R$.
A harmonic function $\alpha$ assigns compatible
probability weights to its matrix blocks.  These weights and the
specified rank on $R$ define a matrix rank on $A(B,R)$.
We denote its rank completion by $\overline A_\alpha(B,R)$.

The main theorem identifies this completion under two conditions on
the weights.  Extremality of $\alpha$ is equivalent to ergodicity of
the associated central measure for the tail equivalence relation.
The second condition, called
\emph{$\alpha$-aperiodicity}, requires the total weight of blocks of
any fixed bounded size to tend to zero along the system.
See \Cref{def:harmonic,def:aperiodic} for the definitions.

\begin{theorem}[Main theorem]\label{thm:main}
Let $R$ be a unital ring equipped with a Sylvester matrix rank function
$\rk$, and let $\alpha$ be a harmonic function on a Bratteli diagram $B$.
If $\alpha$ is extreme and $(B,\alpha)$ is $\alpha$-aperiodic, then
there is a unital $R$-algebra isomorphism
\[
 \overline A_\alpha(B,R)\cong\mathcal M_{R,\rk}
\]
preserving the specified rank on every rectangular matrix space.
\end{theorem}

The completion therefore depends only on $(R,\rk)$ under the stated
hypotheses.
One-vertex diagrams give factor-sequence uniqueness.
For the normalized length rank on $\mathbb Z/4\mathbb Z$,
\Cref{ex:nonregular-completion} shows that the completion contains
a nonzero central square-zero element and is therefore
non-regular and non-simple.

For each fixed level, extremality implies that the path distributions
from later levels converge, in weighted average, to the vertex weights
at that level.
The $\alpha$-aperiodicity condition makes the weighted rank of the
remainders from fitting fixed-size matrix blocks tend to zero.
Together, these estimates give maps between finite stages whose
compositions approximate the connecting maps in rank.
Choosing summable errors yields mutually inverse rank-preserving maps
between the completions.
The maps are constructed by copying matrix blocks, inserting zero
blocks, and conjugating by permutation matrices.

For $R=\mathbb C$, let $\mathcal A_B$ be the $C^*$-completion
of $A(B,\mathbb C)$.
This unital AF \(C^{*}\)-algebra is classified by its ordered $K_0$-group with
order unit \cite{Elliott1976}.
The harmonic function $\alpha$ determines a tracial state $\tau$
on $\mathcal A_B$.
For the associated GNS representation $\pi_\tau$, put
$M_\tau=\pi_\tau(\mathcal A_B)''$.
The trace $\tau$ induces a faithful normal tracial state on $M_\tau$,
and the hypotheses of the main theorem give $M_\tau\cong\mathcal R$.

The ring $\mathcal U(M_\tau)$ of closed densely defined operators
affiliated with $M_\tau$ contains $M_\tau$ and is von Neumann regular.
The normal trace on $M_\tau$, extended to matrices using the ordinary
matrix trace, defines the rank of a matrix over $\mathcal U(M_\tau)$
as the trace of its range projection.
Its pullback along $\pi_\tau$ defines a matrix rank on $\mathcal A_B$
that restricts to $\rk_\alpha$ on $A(B,\mathbb C)$.
By \Cref{prop:cstar-rank-affiliated}, the rank completion of
$\mathcal A_B$ is canonically isomorphic to $\mathcal U(M_\tau)$.
The rank completion of $A(B,\mathbb C)$ identifies canonically with
the rank closure of $\pi_\tau(A(B,\mathbb C))$ in $\mathcal U(M_\tau)$
and, under the hypotheses of the main theorem, is isomorphic to
$\mathcal M_{\mathbb C}$.
\Cref{ex:proper-core-rank-closure} shows that this rank closure can
be a proper subring of $\mathcal U(M_\tau)$.

The paper is organized as follows.
\Cref{sec:rank-completions} collects the preliminaries.
\Cref{sec:bratteli-ranks} constructs the weighted Bratteli algebras
and introduces the completion $\mathcal M_{R,\rk}$.
\Cref{sec:uniqueness} proves the uniqueness theorem.
\Cref{sec:consequences} develops its consequences, examines the role of
coefficient rings and ranks, and studies corner isomorphisms.
\Cref{sec:af-completions} compares $C^*$-completions, weak operator
closures, and rank completions through affiliated-operator rings.

\section*{Acknowledgments}
The author thanks Prof. Hanfeng Li for posing the factor-sequence question that
initiated this work.

\paragraph{\textbf{AI Use Statement.}}
This work was completed with the assistance of ChatGPT
(GPT-5.6 Sol and GPT-6 Astra) in developing and checking
the mathematical arguments.
The author takes full responsibility for all mathematical
statements, proofs, references, and editorial decisions
in the final version.

\section{Preliminaries}\label{sec:rank-completions}

Throughout this paper, the identity element of a unital ring \(R\) is denoted by \(1_R\).
For positive integers \(m,n\), we write \(\Mat_{m\times n}(R)\) for the set of
all \(m\times n\) matrices over \(R\), and put \(\Mat_n(R)=\Mat_{n\times
  n}(R)\).
The identity matrix in \(\Mat_n(R)\) is denoted by \(1_n\); by convention,
\(1_0\) denotes the \(0\times0\) empty matrix.
For $t\in\mathbb R$, $\lfloor t\rfloor$ denotes the greatest integer
not exceeding $t$.

For an abelian group $V$ and an integer $n\geq1$, we write $V^n$ for its
$n$-fold direct product.
When the distinction between rows and columns matters, we use \(V^{1\times n}\)
and \(V^{n\times1}\) for row and column vectors, respectively, with entries in
\(V\).

For matrices $A=(a_{ij})\in\Mat_{r\times s}(R)$ and
$B=(b_{kl})\in\Mat_{m\times n}(R)$, their \emph{Kronecker product} is
\[
A\otimes B=(a_{ij}B)_{i,j}\in\Mat_{rm\times sn}(R),
\qquad
(A\otimes B)_{(i,k),(j,l)}=a_{ij}b_{kl},
\]
where the row and column index pairs are ordered lexicographically.
Interchanging the indices in each row pair and each column pair gives the
alternative block form
\[
 P(A\otimes B)Q=(A b_{kl})_{k,l},
\]
where $P,Q$ are the corresponding row and column permutation matrices.
The resulting matrix equals $B\otimes A$ when every entry of $A$
commutes with every entry of $B$; in general the order of multiplication must be retained.
Identifying $a\in R$ with a $1\times1$ matrix, we have
\[
 1_d\otimes a=\diag(\underbrace{a,\ldots,a}_{d\text{ copies}}),\qquad
 1_d\otimes B=\underbrace{B\oplus\cdots\oplus B}_{d\text{ copies}}.
\]
We interpret $1_0\otimes B$ as an empty block.

\subsection{Sylvester matrix rank functions}
\label{subsec:sylvester-ranks}

Sylvester matrix rank functions were introduced by Malcolmson
\cite{Malcolmson1980} to study homomorphisms to division rings.
For further background, see
\cite{JaikinZapirain2019,JiangLi2021,Li2021} and
\cite[Part~I, Chapter~7]{Schofield1985}.

\Needspace{12\baselineskip}
\begin{definition}\label{def.MRank}
Let $R$ be a unital ring.
A \emph{Sylvester matrix rank function} on $R$ is an
$\mathbb R_{\geq0}$-valued function $\rk$ on the set of all finite rectangular
matrices over $R$ satisfying the following conditions:
\begin{enumerate}[label=(SM\arabic*)]
\item\label{item:sm1} $\rk(0)=0$ and $\rk(1_R)=1$;
\item\label{item:sm2} $\rk(AB)\leq\min\Set*{\rk(A),\rk(B)}$ for any matrices \(A\) and \(B\) of compatible sizes;
\item\label{item:sm3} \( \rk\begin{pmatrix}A&0\\0&B\end{pmatrix}
 =\rk(A)+\rk(B)\) for any matrices \(A\) and \(B\);
\item\label{item:sm4} \(\rk\begin{pmatrix}A&C\\0&B\end{pmatrix}
\geq\rk(A)+\rk(B)\) for matrices $A,B,C$ of appropriate sizes.
\end{enumerate}
\end{definition}

A Sylvester matrix rank function $\rk$ is \emph{faithful} if
$\rk(A)>0$ for every nonzero rectangular matrix $A$ over $R$.

For a division ring $D$, the ordinary matrix rank is denoted by
\[
 \operatorname{rank}_D(X):=\dim_D\bigl(XD^{s\times1}\bigr),
 \qquad X\in\Mat_{r\times s}(D),
\]
where dimension is taken over $D$ on the right.

Normalization \ref{item:sm1} and block additivity \ref{item:sm3} give
$\rk(1_n)=n$.
For $A\in\Mat_{m\times n}(R)$, the product axiom \ref{item:sm2} applied to $A=1_mA=A1_n$
therefore gives
\[
 0\leq\rk(A)\leq\min\{m,n\}.
\]
Applying \ref{item:sm2} to $PAQ$ and $A=P^{-1}(PAQ)Q^{-1}$ shows that
$\rk(PAQ)=\rk(A)$ whenever $P$ and $Q$ are invertible matrices of the
appropriate sizes.  In particular, $\rk(-A)=\rk(A)$.

For $A,B\in\Mat_{m\times n}(R)$, the factorization
\[
 A+B
 =\begin{pmatrix}1_m&1_m\end{pmatrix}
  \begin{pmatrix}A&0\\0&B\end{pmatrix}
  \begin{pmatrix}1_n\\1_n\end{pmatrix}
\]
gives $\rk(A+B)\leq\rk(A)+\rk(B)$.
Applying this to $A=(A-B)+B$ and $B=(B-A)+A$ yields
\[
 \lvert\rk(A)-\rk(B)\rvert\leq\rk(A-B).
\]

For $A=(a_{ij})\in\Mat_{m\times n}(R)$, multiplication by coordinate rows
and columns, together with subadditivity, gives
\[
 \rk(a_{ij})\leq\rk(A)\leq\sum_{i,j}\rk(a_{ij}).
\]
In particular, faithfulness is equivalent to $\rk(a)>0$ for every nonzero $a\in R$.

A ring carrying a Sylvester matrix rank has \emph{invariant basis number}
(IBN): $R^m\cong R^n$ as right $R$-modules implies $m=n$
\cite[Section~3]{JiangLi2021}.  Indeed, inverse module isomorphisms give
$A\in\Mat_{m\times n}(R)$ and $B\in\Mat_{n\times m}(R)$ with
$AB=1_m$ and $BA=1_n$, so
\[
 m=\rk(AB)\leq\rk(A)\leq n,\qquad
 n=\rk(BA)\leq\rk(B)\leq m.
\]

Rank is also additive on orthogonal idempotents.  For orthogonal
idempotents $e,f\in\Mat_n(R)$, set
$U=\begin{pmatrix}e&f\end{pmatrix}$ and
$V=\begin{pmatrix}e\\f\end{pmatrix}$.
Then $UV=e+f$, $VU=\diag(e,f)$, and
$UV=U(VU)V$, $VU=V(UV)U$.
The product axiom gives equality of the ranks of $UV$ and $VU$;
block additivity therefore gives $\rk(e+f)=\rk(e)+\rk(f)$.
Induction yields
\[
 \rk(e_1+\cdots+e_k)=\sum_{i=1}^k\rk(e_i)
\]
for pairwise orthogonal idempotents.  In particular, for $e^2=e\in R$,
\begin{equation}\label{eq:idempotent-complement-rank}
 \rk(1_R-e)=1-\rk(e).
\end{equation}

\begin{remark}
\label{rem:pseudo-rank-dimension}
A \emph{pseudo-rank function} on a unital ring $R$ is a map
$N\colon R\to[0,1]$ satisfying:
\begin{enumerate}
    \item \(N(0)=0\) and \(N(1_R)=1\);
    \item \(N(a+b)\leq N(a)+N(b)\) for all \(a,b\in R\);
    \item \(N(ab)\leq\min\{N(a),N(b)\}\) for all \(a,b\in R\);
    \item \(N(e+f)=N(e)+N(f)\) for all \(e,f\in R\) with \(e^2=e\), \(f^2=f\)
  and \(ef=fe=0\).
\end{enumerate}
This is the formulation for general unital rings in
\cite[Definition~2.1]{AraClaramunt2018}.  For regular rings, it agrees
with Goodearl's definition \cite[Chapter~16, p.~226]{Goodearl1991}:
subadditivity follows from the other axioms by
\cite[Proposition~16.1(d)]{Goodearl1991}.
In that terminology, a pseudo-rank $N$ is a \emph{rank function} if
$N(a)>0$ for every nonzero $a\in R$.
The preceding properties show that the scalar restriction
$N(a)=\rk(a)$ is a pseudo-rank; the associated rank pseudometric is
$d_N(a,b)=N(a-b)$.

Suppose that $R$ is von Neumann regular, that is, for each $a\in R$
there is $b\in R$ with $aba=a$.
Then every pseudo-rank $N$ extends uniquely to a Sylvester
matrix rank $\rk_N$ \cite[Proposition~1.3.9]{LopezAlvarez2021}.
Equivalently, $N$ determines a unique normalized
dimension $\dim_N$ on finitely generated projective right $R$-modules
\cite[Proposition~16.8]{Goodearl1991}.
It is nonnegative, invariant under isomorphisms, and satisfies
\[
 \dim_N(R)=1,\qquad
 \dim_N(P\oplus Q)=\dim_N(P)+\dim_N(Q).
\]
The correspondences are given by
\begin{equation}\label{eq:pseudo-rank-projective-dimension}
 N(a)=\dim_N(aR),\qquad
 \rk_N(X)=\dim_N(XR^{s\times1})
 \quad\bigl(X\in\Mat_{r\times s}(R)\bigr).
\end{equation}
Regularity ensures that these image modules are finitely generated
projective.
Thus the scalar and matrix formulations give the same rank completion
in the regular case.
\end{remark}

A \emph{rank ring} $(R,\rk)$ is a unital ring with a specified Sylvester
matrix rank function.
We assume neither regularity of $R$ nor
faithfulness of $\rk$.
For positive integers $r,s,p$, we use the
canonical identification
\[
 \Mat_{r\times s}(\Mat_p(R))
 \cong\Mat_{rp\times sp}(R).
\]
For $X\in\Mat_{r\times s}(\Mat_p(R))$, define
\begin{equation}\label{eq:normalized-matrix-rank}
 \rk_p(X):=\frac1p\rk(X).
\end{equation}
Then $\rk_p$ is a Sylvester matrix rank function on $\Mat_p(R)$.

Let $f\colon S\to T$ be a possibly nonunital ring homomorphism.
For $A=(a_{ij})\in\Mat_{r\times s}(S)$, we set
\[
 f(A):=(f(a_{ij}))\in\Mat_{r\times s}(T).
\]
If $f$ is unital and $\rk_T$ is a Sylvester matrix rank function on $T$, then
\[
 (f^*\rk_T)(A):=\rk_T(f(A))
\]
defines a Sylvester matrix rank function on $S$, called the \emph{pullback} of
$\rk_T$.
For rank rings $(S,\rk_S)$ and $(T,\rk_T)$, we call $f$
\emph{rank-preserving} if, for all $r,s\geq1$,
\[
 \rk_T(f(A))=\rk_S(A),
 \qquad \forall A\in\Mat_{r\times s}(S).
\]

\medskip
For a rank ring $(R,\rk)$, define
\[
 d_{\rk}(x,y):=\rk(x-y)\qquad(x,y\in R),
\]
which is a pseudometric on \(R\).
The set
\[
 \ker(\rk)=\Set*{a\in R\given\rk(a)=0}
\]
is a two-sided ideal.
The quotient $R/\ker(\rk)$ inherits a faithful Sylvester matrix rank, and $d_{\rk}$ induces its metric.
Subadditivity and the product axiom also give
\[
 \begin{aligned}
 d_{\rk}(x+y,x'+y')&\leq d_{\rk}(x,x')+d_{\rk}(y,y'),\\
 d_{\rk}(xy,x'y')&\leq d_{\rk}(x,x')+d_{\rk}(y,y').
 \end{aligned}
\]
Thus sums and products of Cauchy sequences define ring operations on the
metric completion, independently of the representatives.
The completion of $R/\ker(\rk)$ in this metric is called the
\emph{rank completion} of $R$.  We denote it by $\overline{R_{\rk}}$,
or by $\overline R$ when the rank is clear.

For a ring $S$ with a faithful Sylvester matrix rank \(\rk\), the bounds
\begin{equation}\label{eq:rank-continuity-bounds}
 \begin{aligned}
 |\rk(X)-\rk(Y)|&\leq\rk(X-Y),\\
 \max_{i,j}\rk(x_{ij}-y_{ij})
 &\leq\rk(X-Y)\leq\sum_{i,j}\rk(x_{ij}-y_{ij})
 \end{aligned}
\end{equation}
for $X,Y\in\Mat_{r\times s}(S)$ give a unique continuous extension
to matrices over $\overline S$:
\[
 \overline{\rk}(X):=\lim_k\rk(X_k),
 \qquad X_k\in\Mat_{r\times s}(S),\quad X_k\longrightarrow X
 \text{ entrywise}.
\]
The rank axioms pass to limits by continuity of matrix operations.
The extension is faithful by the entrywise lower bound and the identity
$\overline{\rk}(x)=d_{\rk}(x,0)$ on $\overline S$.
We henceforth omit the bar on rank.

Likewise, a rank-preserving homomorphism $f\colon S\to T$ between rings
with faithful ranks extends uniquely to a rank-preserving homomorphism
$\overline f\colon\overline S\to\overline T$, given by
\[
 \overline f\bigl(\lim_k x_k\bigr):=\lim_k f(x_k).
\]
Indeed, $f$ is an isometry, so this extension is well defined and unique
by density; continuity of ring operations and matrix ranks preserves
the homomorphism and rank-preservation identities.
For nonfaithful ranks, rank preservation gives
$f(\ker\rk_S)\subseteq\ker\rk_T$, and the same argument applies to
the induced map $S/\ker\rk_S\to T/\ker\rk_T$.

\subsection{Bratteli diagrams}
\label{subsec:bratteli-diagrams}

Bratteli \cite{Bratteli1972} introduced these diagrams to describe
inductive limits of finite-dimensional $C^*$-algebras.  Their connections
with dimension groups and topological dynamics are developed in
\cite{HermanPutnamSkau1992}.
The definition and terminology for Bratteli diagrams follow
\cite[Section~2.2]{BezuglyiKarpel2016}, with $E_n$ denoting the edges
from $V_{n-1}$ to $V_n$.
Unlike the standing assumption in that reference, we allow the infinite
path space $X_B$ to have isolated points.  This includes one-vertex diagrams
whose path spaces are finite.

\begin{definition}[Bratteli diagram]\label{def:bratteli}
A \emph{Bratteli diagram} is an infinite graph $B=(V,E)$ whose \emph{vertex} and \emph{edge}
sets are partitioned as
\[
 V=\bigsqcup_{n\geq0}V_n,
 \qquad
 E=\bigsqcup_{n\geq1}E_n,
\]
where $V_0=\Set*{v_0}$, every $V_n$ and $E_n$ is finite and non-empty, and there
are \emph{source} and \emph{range} maps $s,r\colon E\to V$ satisfying
\[
 s(E_n)=V_{n-1},
 \qquad
 r(E_n)=V_n
 \qquad(n\geq1).
\]
\end{definition}

Given a Bratteli diagram \(B=(V,E)\), every vertex emits an edge, and every
vertex other than $v_0$ receives an edge.
The set $V_n$ is the $n$th level of $B$.

A \emph{finite path} from level $m$ to level $n>m$ is a sequence
$(e_{m+1},\ldots,e_n)$ with $e_i\in E_i$ and
$r(e_i)=s(e_{i+1})$ for $m<i<n$.
An \emph{infinite path} starting at level $m$ is a sequence
$(e_i)_{i\geq m+1}$ with $e_i\in E_i$ and
$r(e_i)=s(e_{i+1})$ for $m<i$.
For $0 \leq m < n$, let $E_{m,n}$ denote the set of all paths from
$V_m$ to $V_n$.
Explicitly,
\[
E_{m,n} = \Set*{(e_{m+1}, e_{m+2}, \dots, e_n) \given
e_i \in E_i, \; m < i \leq n, \;
r(e_i) = s(e_{i+1}), \; m < i < n}.
\]
Thus $E_{n-1,n}=E_n$ for every $n\geq1$.
For $p = (e_{m+1}, e_{m+2}, \dots, e_n)$ in $E_{m,n}$, we let $s(p) =
s(e_{m+1})$ and $r(p) = r(e_n)$.
If $p\in E_{l,m}$ and $q\in E_{m,n}$ satisfy $r(p)=s(q)$, their \emph{concatenation}
is denoted by $pq\in E_{l,n}$.

For $v\in V_n$, let $1_v$ denote the \emph{empty path} at $v$.
We put $E_{n,n}=\Set*{1_v \given v\in V_n}$ and set
$s(1_v)=r(1_v)=v$.  Concatenation extends to empty paths by
$1_{s(p)}p=p=p1_{r(p)}$.

Let
\[
X_B = \Set*{(e_1, e_2, \dots) \given
e_n \in E_n, \; r(e_n) = s(e_{n+1}), \; n \geq 1},
\]
denote the set of infinite paths starting at the vertex \(v_{0}\).
We endow \(X_B\) with the topology generated by cylinder sets
\[
C(p) = \Set*{x \in X_B \given (x_1, x_2, \dots, x_n) = p}
\]
where $n \geq 1$ and $p\in E_{0,n}$.
The path space is a closed subspace of the product of the finite
discrete sets $E_n$, and every finite path extends to an infinite path.
Thus $X_B$ is non-empty, compact and metrizable.
The cylinders form
a countable clopen basis and generate the Borel $\sigma$-algebra
of $X_B$.

For $0\leq m\leq n$, $v\in V_m$, and $w\in V_n$, set
\[
 a_{w,v}^{(n,m)}
 =\bigl|\Set*{ p\in E_{m,n}  \given s(p)=v,\ r(p)=w }\bigr|.
\]
Here $|A|$ denotes the cardinality of a set $A$.
Thus $a_{w,v}^{(n,n)}=\delta_{w,v}$.
Define the \emph{incidence matrix} from \(V_{m}\) to $V_{n}$ by
\[
 F_{n,m}
 =\bigl(a_{w,v}^{(n,m)}\bigr)_{w\in V_n,\,v\in V_m}
 \in \Mat_{|V_n|\times |V_m|}(\setZ).
\]
Fix an ordering of each finite vertex set $V_n$, used for incidence
matrices and block sums indexed by vertices throughout the paper.  The rows of $F_{n,m}$ are indexed by $V_n$, and its columns by
$V_m$.
In particular, $F_{n,n} = 1_{|V_n|}$ is the identity matrix.
We write $F_n=F_{n+1,n}$ for the one-step incidence matrix; this agrees with
the convention of \cite[Section~2.2]{BezuglyiKarpel2016}.
For $v\in V_n$ and $w\in V_{n+1}$, we abbreviate $a_{w,v}=a_{w,v}^{(n+1,n)}$.

Every path from $u\in V_\ell$ to $w\in V_n$ has a unique vertex
$v\in V_m$ at level $m$ and a unique decomposition into a path from $u$
to $v$ and one from $v$ to $w$.  Hence
\[
 a_{w,u}^{(n,\ell)}
 =\sum_{v\in V_m}a_{w,v}^{(n,m)}a_{v,u}^{(m,\ell)}.
\]
This is the entrywise formula for
$F_{n,\ell}=F_{n,m}F_{m,\ell}$, including the endpoint cases in view of
the empty-path convention.

For $v\in V_n$, let
\[
 p_n(v)=a_{v,v_0}^{(n,0)}
\]
be the number of paths from the root to $v$.
Every nonroot vertex receives an edge, so induction on $n$ gives a
path from $v_0$ to every $v\in V_n$.  Thus $p_n(v)>0$, and the
empty-path convention gives $p_0(v_0)=1$.  Write
\[
 p^{(n)}=(p_n(v))_{v\in V_n}\in\setZ^{|V_n|\times1},
\]
where the entries are arranged as a column vector in the same order as the
rows of $F_{n,m}$.
For every $0\leq m\leq n$ we then have
\begin{equation}\label{eq:path-count}
 p^{(n)}=F_{n,m}p^{(m)},
 \qquad\text{equivalently}\qquad
 p_n(w)=\sum_{v\in V_m}a_{w,v}^{(n,m)}p_m(v)
 \quad(w\in V_n).
\end{equation}
In particular, $p^{(n+1)}=F_np^{(n)}$.

The \emph{tail equivalence relation} $\mathcal E_B$ on $X_B$ is given by
\[
  x\mathrel{\mathcal E_B}y
  \quad\Longleftrightarrow\quad
  \text{there is }N\text{ such that }x_j=y_j\text{ for every }j\geq N.
\]
This is the tail-equivalence convention of
\cite[Section~2.2]{BezuglyiKarpel2016}; its equivalence class at $x$ is denoted
by $[x]_{\mathcal E_B}$.

\subsection{Direct limits}

We use \emph{direct systems} $(R_n,\phi_{m,n})_{m\geq n\geq0}$ of unital
rings and unital ring homomorphisms $\phi_{m,n}\colon R_n\to R_m$, with
\[
 \phi_{n,n}=\id_{R_n},\qquad
 \phi_{m,n}=\phi_{m,k}\circ\phi_{k,n}\quad(m\geq k\geq n).
\]
Writing $\phi_n=\phi_{n,n-1}$ for $n\geq1$, we also denote this system
by $(R_n,\phi_n)$ and display it as
\[
 R_0\xrightarrow{\phi_1}R_1\xrightarrow{\phi_2}R_2\longrightarrow\cdots,
 \qquad \phi_{m,n}=\phi_m\circ\cdots\circ\phi_{n+1}\quad(m>n).
\]
Unitality here concerns the connecting maps; the auxiliary maps used in
\Cref{sec:uniqueness} need not be unital.
The \emph{algebraic direct limit} is
\[
 R_\infty=\varinjlim(R_n,\phi_{m,n})
 =\left(\bigsqcup_{n\geq0}(R_n\times\{n\})\right)\big/\sim,
\]
where
\[
 (x,n)\sim(y,m)
 \iff \phi_{k,n}(x)=\phi_{k,m}(y)\text{ for some }k\geq n,m.
\]
Write $[x,n]$ for the class of $(x,n)$.  The canonical maps satisfy
\[
 \phi_{\infty,n}\colon R_n\to R_\infty,\quad x\mapsto[x,n],
 \qquad
 \phi_{\infty,m}\phi_{m,n}=\phi_{\infty,n}\quad(m\geq n).
\]

Every finite family in $R_\infty$ has representatives at a common stage,
where its sums and products are computed.

\begin{proposition}\label{prop:direct-limit-rank}
Suppose each $R_n$ carries a Sylvester matrix rank $\rk_{R_n}$
preserved by the connecting maps.
Then $R_\infty$ carries the unique Sylvester matrix rank satisfying
\[
 \rk_{R_\infty}(\phi_{\infty,n}(X))=\rk_{R_n}(X),
 \qquad \forall X\in\Mat_{r\times s}(R_n).
\]
It is faithful if every $\rk_{R_n}$ is faithful.
\end{proposition}

\begin{proof}
For \(X\in\Mat_{r\times s}(R_n)\) and \(Y\in\Mat_{r\times s}(R_m)\), we have
\[
 \begin{aligned}
 \phi_{\infty,n}(X)=\phi_{\infty,m}(Y)
 &\implies \phi_{k,n}(X)=\phi_{k,m}(Y)\quad\text{for some }k\geq n,m,\\
 &\implies \rk_{R_n}(X)=\rk_{R_k}(\phi_{k,n}(X))
   =\rk_{R_m}(Y).
 \end{aligned}
\]
Thus the formula is well defined.  The rank axioms and, when applicable,
faithfulness are checked on representatives at a common stage;
these representatives also give uniqueness.
\end{proof}

\begin{proposition}\label{prop:cofinal-direct-limit}
For any infinite subsequence $n_0<n_1<\cdots$, the canonical map
\[
 \eta\colon\varinjlim_i(R_{n_i},\phi_{n_{i+1},n_i})
 \xrightarrow{\ \cong\ }R_\infty,
 \qquad [x,i]\longmapsto[x,n_i],
\]
is a ring isomorphism.  Under the hypotheses of
\Cref{prop:direct-limit-rank}, it preserves rank and induces an
isomorphism of rank completions.
\end{proposition}

\begin{proof}
Cofinality gives an inverse
\[
 \eta^{-1}([x,n])=[\phi_{n_i,n}(x),i]\qquad(n_i\geq n),
\]
independent of the representative and of $i$ by the common-stage
criterion.  Both maps preserve operations and the specified ranks on
stage representatives, so they extend to mutually inverse
rank-preserving maps on the completions.
\end{proof}

\subsection{\texorpdfstring{$C^*$}{C*}-algebras}
\label{subsec:cstar-preliminaries}

A \emph{$C^*$-algebra} is a complex Banach $*$-algebra $\mathcal A$
satisfying $\|a^*a\|=\|a\|^2$ for every $a\in\mathcal A$.
An element is \emph{positive} if it has the form $b^*b$ for some \(b\in\mathcal A\).
A \emph{state} on a unital $C^*$-algebra is a linear functional
$\tau\colon\mathcal A\to\mathbb C$ with $\tau(1)=1$ and
$\tau(a^*a)\geq0$ for all $a$.
It is \emph{tracial} if $\tau(ab)=\tau(ba)$ for all $a,b$, and
\emph{faithful} if $\tau(a^*a)=0$ implies $a=0$.
A \emph{$*$-representation} of $\mathcal A$ on a Hilbert space $H$ is a
$*$-homomorphism $\pi\colon\mathcal A\to\mathcal B(H)$, where
$\mathcal B(H)$ is the \(C^{*}\)-algebra of bounded linear operators on $H$.

Injective $*$-homomorphisms between $C^*$-algebras are isometric.
For a direct system $(\mathcal A_n,\phi_n)$ with injective
$*$-homomorphisms, the algebraic direct limit has a well-defined norm
\[
 \|\phi_{\infty,n}(a)\|=\|a\|_{\mathcal A_n}.
\]
The norm completion of this algebraic direct limit is the
\emph{$C^*$-algebraic inductive limit}.
A separable $C^*$-algebra is \emph{approximately finite-dimensional}, or
\emph{AF}, if it is the norm closure of an increasing union of
finite-dimensional $C^*$-subalgebras \cite{Bratteli1972}.
An infinite-dimensional unital $C^*$-algebra is \emph{uniformly
hyperfinite} (UHF) if it is the norm closure of an increasing union
of unital full matrix subalgebras \cite{Glimm1960}.

For a unital $C^*$-algebra $\mathcal A$ with tracial state $\tau$,
let $(\pi_\tau,H_\tau,\xi_\tau)$ be its GNS representation and put
\begin{equation}\label{eq:gns-normal-trace}
 M_\tau=\pi_\tau(\mathcal A)'',\qquad
 \widetilde\tau(T)=\langle T\xi_\tau,\xi_\tau\rangle.
\end{equation}
Then $\widetilde\tau$ is a faithful normal tracial state and
\[
 \widetilde\tau(\pi_\tau(a))=\tau(a)\qquad(a\in\mathcal A).
\]
The vector state $\widetilde\tau$ is normal, and the trace identity
extends to $M_\tau$ by separate ultraweak continuity.
The bounded right multiplication operators lie in $M_\tau'$
and have a dense orbit of $\xi_\tau$.
Thus $\xi_\tau$ is separating for $M_\tau$, so
$\widetilde\tau$ is faithful.
Consequently, $M_\tau$ is finite, even when $\tau$ is not faithful on
$\mathcal A$; see \cite[Section~2.4]{JiangLi2021}.

Since $\pi_\tau(\mathcal A)$ is ultraweakly dense in $M_\tau$,
$\widetilde\tau$ is the unique normal tracial state on $M_\tau$
whose pullback along $\pi_\tau$ is $\tau$.

Let $\mathcal A$ be a unital $C^*$-algebra and $\tau$ a tracial state.
For $A = (a_{ij})\in\Mat_k(\mathcal A)$, put
\[
 \tau_k(A):=(\operatorname{Tr}_k\otimes\tau)(A)
 =\sum_{j=1}^k\tau(a_{jj}),
\]
where $\operatorname{Tr}_k$ is the ordinary (unnormalized) matrix trace.
On $\Mat_k(\mathbb C)$, write
$\operatorname{tr}_k=k^{-1}\operatorname{Tr}_k$ for the normalized trace.
For
$X\in\Mat_{r\times s}(\mathcal A)$, let
$|X|=(X^*X)^{1/2}\in\Mat_s(\mathcal A)$ and define
\begin{equation}\label{eq:trace-root-rank}
 \rk_\tau(X):=\lim_{k\to\infty}\tau_s(|X|^{1/k}).
\end{equation}

By \cite[Proposition~2.4]{JiangLi2021}, the limit in
\eqref{eq:trace-root-rank} exists and defines a Sylvester matrix
rank function on $\mathcal A$.
For a projection $e\in\Mat_k(\mathcal A)$, that is,
$e=e^*=e^2$, functional calculus gives
$|e|^{1/j}=e$ for every $j\geq1$, and hence
\[
  \rk_\tau(e)=\tau_k(e).
\]

\section{\texorpdfstring{$R$}{R}-algebras associated with Bratteli diagrams and their ranks}
\label{sec:bratteli-ranks}

Throughout this section and \Cref{sec:uniqueness}, $(R,\rk)$ is a fixed
rank ring.

\subsection{The associated ultramatricial \texorpdfstring{$R$}{R}-algebra}
\label{subsec:associated-algebra}

An \emph{$R$-algebra} is a ring $S$ with a specified unital
homomorphism $\iota_S\colon R\to S$; the image is not required to be central.
An $R$-algebra homomorphism is a ring homomorphism $f\colon S\to T$
satisfying $f\iota_S=\iota_T$; in particular, it is unital.
The diagonal homomorphism
\[
 \iota_{\Mat_p(R)}\colon R\longrightarrow\Mat_p(R),\qquad
 r\longmapsto1_p\otimes r,
\]
defines the standard $R$-algebra structure on $\Mat_p(R)$.

An $R$-algebra is \emph{matricial} if it is isomorphic as an $R$-algebra to
$\bigoplus_{j=1}^t\Mat_{d_j}(R)$ for some positive integers
$t,d_1,\ldots,d_t$, with structure map
\[
 \iota\colon R\longrightarrow\bigoplus_{j=1}^t\Mat_{d_j}(R),
 \qquad r\longmapsto\bigl(1_{d_j}\otimes r\bigr)_{j=1}^t.
\]
It is \emph{ultramatricial} if it is isomorphic as an \(R\)-algebra to an
algebraic direct limit of a sequence of matricial \(R\)-algebras.
These definitions extend the notions of matricial and ultramatricial
$D$-rings in \cite[Section~3]{AraClaramunt2018} from division rings
to arbitrary unital coefficient rings.

Let $B=(V,E)$ be a Bratteli diagram, with notation as in
\Cref{subsec:bratteli-diagrams}.  For each $n\geq0$, define the finite
product of matrix rings
\begin{equation}\label{eq:level-algebra}
 A_n(B,R)=\bigoplus_{v\in V_n}\Mat_{p_n(v)}(R).
\end{equation}
Equip $A_n(B,R)$ with the componentwise $R$-algebra structure above,
and denote its structure map by $\iota_n$.
In particular, $A_0(B,R)=R$.  We abbreviate $A_n(B,R)$ to $A_n$ when the
diagram \(B\) and the ring \(R\) are fixed.

Fix a total order on each edge set $E_n$.
For $x=(x_v)_{v\in V_n}\in A_n$ and $w\in V_{n+1}$, the blocks
$x_{s(e)}$, with
$e\in E_{n+1}$ and $r(e)=w$, inherit this order.  Since
\[
 \sum_{\substack{e\in E_{n+1}\\r(e)=w}}p_n(s(e))
 =\sum_{v\in V_n}a_{w,v}p_n(v)=p_{n+1}(w),
\]
the ordered block sum defines a map
\begin{equation}\label{eq:bratteli-connecting-map}
 \phi_{n+1}\colon A_n \to A_{n+1},\qquad (x_v)_{v\in V_n} \mapsto
 \Bigg(\bigoplus_{\substack{e\in E_{n+1}\\r(e)=w}}x_{s(e)}\Bigg)_{w\in V_{n+1}}.
\end{equation}
Regrouping the copies by their source vertices, in the fixed order on
$V_n$, gives
\[
 C_w\bigl(\phi_{n+1}(x)\bigr)_w C_w^{-1}
 =\bigoplus_{v\in V_n}(1_{a_{w,v}}\otimes x_v)
\]
for a permutation matrix $C_w$ independent of $x$; terms with
$a_{w,v}=0$ are omitted.
Block diagonal addition and multiplication yield
\[
 \phi_{n+1}(x+y)=\phi_{n+1}(x)+\phi_{n+1}(y)\text{ and }
 \phi_{n+1}(xy)=\phi_{n+1}(x)\phi_{n+1}(y).
\]
Moreover, the path-count identity implies, for $r\in R$,
\[
 \bigl(\phi_{n+1}(\iota_n(r))\bigr)_w
 =\bigoplus_{\substack{e\in E_{n+1}\\r(e)=w}}
   (1_{p_n(s(e))}\otimes r)
 =1_{p_{n+1}(w)}\otimes r.
\]
Hence $\phi_{n+1}\iota_n=\iota_{n+1}$; taking $r=1_R$ also gives
$\phi_{n+1}(1_{A_n})=1_{A_{n+1}}$.
Thus $\phi_{n+1}$ is a unital $R$-algebra homomorphism.
It is injective because every vertex emits an edge, so every $x_v$
occurs as a diagonal block of $\phi_{n+1}(x)$.
These maps define a direct system of unital $R$-algebras
\[
 R=A_0\xrightarrow{\phi_1}A_1\xrightarrow{\phi_2}A_2
 \xrightarrow{\phi_3}\cdots.
\]
For $m<n$, put $\phi_{n,m}=\phi_n\circ\cdots\circ\phi_{m+1}$.
Order $E_{m,n}$ lexicographically by the reversed edge sequences
$(e_n,\ldots,e_{m+1})$, using the fixed orders on the edge sets.  Then
\begin{equation}\label{eq:bratteli-composite-map}
 \bigl(\phi_{n,m}(x)\bigr)_w
 =\bigoplus_{\substack{p\in E_{m,n}\\r(p)=w}}x_{s(p)}.
\end{equation}
Indeed, composition with $\phi_{n+1}$ orders the blocks first by the
last edge and then by the preceding path, which proves the formula
by induction.  The $w$-component therefore contains
$a_{w,v}^{(n,m)}$ copies of $x_v$.  Grouping them by source vertex
requires a permutation conjugation independent of $x$.
Define
\[
 A(B,R)=\varinjlim(A_n(B,R),\phi_{n,m}).
\]
The compatible maps $\iota_n$ give the limit its specified $R$-algebra
structure.  Injectivity of the connecting maps identifies this limit with the
increasing union of the canonical images of the $A_n$.
It is the ultramatricial $R$-algebra associated with $B$.

Changing the orders on the edge sets gives an isomorphic direct system.
If $\phi'_{n+1}$ denotes the new map, reordering the copies
gives a blockwise permutation matrix $U_{n+1}\in A_{n+1}$ such that
\[
 \phi'_{n+1}=\Ad(U_{n+1})\phi_{n+1},\qquad
 \Ad(U)(x):=UxU^{-1}.
\]
Set $P_0=1_R$ and $P_{n+1}=U_{n+1}\phi_{n+1}(P_n)$.  Each $P_n$ is
blockwise a permutation matrix, and
\[
 \Ad(P_{n+1})\phi_{n+1}=\phi'_{n+1}\Ad(P_n).
\]
These conjugacies fix the scalar diagonal matrices, hence are
$R$-algebra isomorphisms and induce an isomorphism of the direct limits.

\subsection{Harmonic functions and weighted ranks}

For $X=(x_{ij})\in\Mat_{r\times s}(A_n)$, write
$x_{ij}=(x_{ij,v})_{v\in V_n}$ and for \(v\in V_n\) set
\[
 X_v:=(x_{ij,v})_{i,j}
 \in\Mat_{r\times s}\bigl(\Mat_{p_n(v)}(R)\bigr).
\]

Fix $n\geq1$ and $\lambda\colon V_n\to[0,1]$ such that $\sum_{v\in V_n}\lambda(v)=1$.
For $X\in\Mat_{r\times s}(A_n)$, define its weighted rank by
\begin{equation}\label{eq:level-weighted-rank}
 \rk_{\lambda,n}(X)
 :=\sum_{v\in V_n}\lambda(v)\rk_{p_n(v)}(X_v)
 =\sum_{v\in V_n}\frac{\lambda(v)}{p_n(v)}\rk(X_v).
\end{equation}
As a convex combination of the normalized component ranks,
$\rk_{\lambda,n}$ is a Sylvester matrix rank function on $A_n$.
For $v\in V_n$, define $e_{n,v}\in A_n$ by
\[
 (e_{n,v})_w=
 \begin{cases}
  1_{p_n(v)},&w=v,\\
  0,&w\ne v.
 \end{cases}
\]
The element $e_{n,v}$ is the identity of the $v$-summand, regarded
as an element of $A_n$, and
\[
 \rk_{p_n(v)}(1_{p_n(v)})=1,
 \qquad \rk_{\lambda,n}(e_{n,v})=\lambda(v).
\]
In particular, distinct weights give distinct ranks.

For a family of probability weights $(\lambda_n)_{n\geq1}$,
compatibility of the associated ranks with $\phi_{n+1}$, evaluated at
$e_{n,v}$, requires
\[
 \lambda_n(v)=\sum_{w\in V_{n+1}}
 \frac{p_n(v)a_{w,v}}{p_{n+1}(w)}\lambda_{n+1}(w).
\]
These are the harmonicity equations for the vertex weights.

\begin{definition}\label{def:harmonic}
Let $B=(V,E)$ be a Bratteli diagram.
A \emph{harmonic function} on $B$ is a family
$\alpha=(\alpha_n)_{n\geq1}$, where
$\alpha_n\colon V_n\to[0,1]$, such that, for every $n\geq1$,
\[
 \sum_{v\in V_n}\alpha_n(v)=1,
\]
and, for every $v\in V_n$,
\begin{equation}\label{eq:harmonic}
  \alpha_n(v)=\sum_{w\in V_{n+1}}
  \frac{p_n(v)a_{w,v}}{p_{n+1}(w)}\alpha_{n+1}(w).
\end{equation}
\end{definition}

This normalization agrees with that of Elek
\cite[Section~4]{Elek2013}.
We set $\alpha_0(v_0)=1$; then \eqref{eq:harmonic} also holds for $n=0$.
We denote the set of harmonic functions by $\mathcal H(B)$ and give it the
product topology.
By \Cref{lem:harmonic-simplex-nonempty}, it is non-empty and compact.
The function $\alpha$ is \emph{extreme} if an equality
$\alpha=t\beta+(1-t)\gamma$, with $\beta,\gamma\in\mathcal H(B)$ and $0<t<1$, forces
$\beta=\gamma=\alpha$.
It is \emph{strictly positive} if $\alpha_n(v)>0$ for all $n\geq1$ and
$v\in V_n$.

For \(1\leq m<n\), iterating \eqref{eq:harmonic} gives
\begin{equation}\label{eq:harmonic-long}
 \alpha_m(v)
 =\sum_{w\in V_n}
 \frac{p_m(v)a_{w,v}^{(n,m)}}{p_n(w)}\alpha_n(w).
\end{equation}
The induction step follows by substitution and path concatenation:
\[
 \sum_{w\in V_n}
 \frac{p_m(v)a_{w,v}^{(n,m)}}{p_n(w)}
 \frac{p_n(w)a_{z,w}^{(n+1,n)}}{p_{n+1}(z)}
 =\frac{p_m(v)}{p_{n+1}(z)}a_{z,v}^{(n+1,m)}.
\]

Fix $\alpha\in\mathcal H(B)$.  At level $n$, denote the rank in
\eqref{eq:level-weighted-rank} with $\lambda=\alpha_n$ by
$\rk_{\alpha,n}$.  Thus, for
$X\in\Mat_{r\times s}(A_n)$,
\begin{equation}\label{eq:weighted-matrix-rank}
 \rk_{\alpha,n}(X)
 :=\sum_{v\in V_n}\alpha_n(v)\rk_{p_n(v)}(X_v)
 =\sum_{v\in V_n}\frac{\alpha_n(v)}{p_n(v)}\rk(X_v).
\end{equation}
\begin{proposition}\label{prop:rank-compatible}
For every $n\geq1$, the function $\rk_{\alpha,n}$ is a Sylvester matrix rank
function on $A_n(B,R)$.  These ranks are compatible with the connecting maps
and therefore induce a Sylvester matrix rank function on $A(B,R)$.
\end{proposition}

\begin{proof}
The first assertion follows from the construction in
\eqref{eq:level-weighted-rank}.

To prove compatibility, fix $1\leq m<n$ and $X\in\Mat_{r\times s}(A_m)$.
For each $w\in V_n$, the map $\phi_{n,m}$ repeats the $v$-block
$a_{w,v}^{(n,m)}$ times.
Applying this map entrywise, block
additivity and permutation invariance give
\[
 \rk\bigl((\phi_{n,m}(X))_w\bigr)
 =\sum_{v\in V_m}a_{w,v}^{(n,m)}\rk(X_v).
\]
Together with \eqref{eq:harmonic-long}, this yields
\begin{align*}
 \rk_{\alpha,n}(\phi_{n,m}(X))
 &=\sum_w\frac{\alpha_n(w)}{p_n(w)}
   \sum_v a_{w,v}^{(n,m)}\rk(X_v)\\
 &=\sum_v\frac{\rk(X_v)}{p_m(v)}
   \sum_w\frac{p_m(v)a_{w,v}^{(n,m)}}{p_n(w)}\alpha_n(w)\\
 &=\rk_{\alpha,m}(X).
\end{align*}
Equip $A_0=R$ with its fixed rank $\rk$.  For a rectangular
matrix $X$ over $R$, the $v$-block of $\phi_1(X)$ has $p_1(v)$ copies
of $X$, up to row and column permutations.  Hence
\[
 \rk_{\alpha,1}(\phi_1(X))
 =\sum_{v\in V_1}\alpha_1(v)\rk(X)=\rk(X).
\]
This proves compatibility at all levels.
\end{proof}

Let $\phi_{\infty,n}\colon A_n(B,R)\to A(B,R)$ be the canonical map.
Every $X\in\Mat_{r\times s}(A(B,R))$ has a representative
$X_n\in\Mat_{r\times s}(A_n(B,R))$ for some \(n\).
By \Cref{prop:direct-limit-rank},
\[
 \rk_\alpha(X):=\rk_{\alpha,n}(X_n),
 \qquad X=\phi_{\infty,n}(X_n),
\]
is independent of the representative and defines a Sylvester matrix
rank function on $A(B,R)$.

We write
\[
 \overline A_\alpha(B,R)
 =\overline{\bigl(A(B,R)/\ker(\rk_\alpha)\bigr)}
\]
for the completion in the induced rank metric.
The permutation conjugacies associated with a change of edge orders
preserve \eqref{eq:weighted-matrix-rank} by invariance of rank under
invertible row and column operations.  Thus the ranked $R$-algebras
$A(B,R)$ and $\overline A_\alpha(B,R)$ are independent of these
orders up to rank-preserving $R$-algebra isomorphism.

The existence of harmonic functions follows from compactness.
For a non-empty finite set $W$, write
\[
 \Delta(W)
 :=\Set*{x\colon W\to[0,1] \given \sum_{v\in W}x(v)=1}
\]
for the simplex on $W$.

\begin{lemma}\label{lem:harmonic-simplex-nonempty}
For every Bratteli diagram $B=(V,E)$, the set $\mathcal H(B)$ is non-empty, compact,
and convex.
Consequently, $\mathcal H(B)$ has extreme points.
\end{lemma}

\begin{proof}
For $n\geq1$, define
\[
 K_n\colon\Delta(V_{n+1})\longrightarrow\Delta(V_n)
\]
by
\[
 (K_nx)(v)
 =\sum_{w\in V_{n+1}}
   \frac{p_n(v)a_{w,v}}{p_{n+1}(w)}x(w).
\]
Moreover, \eqref{eq:path-count} gives, for every $w\in V_{n+1}$,
\[
 \sum_{v\in V_n}
 \frac{p_n(v)a_{w,v}}{p_{n+1}(w)}=1.
\]
Thus $K_n$ is a well-defined continuous affine map between simplices.

The product $X=\prod_{n\geq1}\Delta(V_n)$ is compact because each factor is a
non-empty compact simplex.
For $N\geq1$, let
\[
 F_N=
 \Set*{(x_n)_{n\geq1}\in X \given
 x_n=K_nx_{n+1}\text{ for }1\leq n\leq N}.
\]
Each $F_N$ is closed.  It is non-empty: for any
$x_{N+1}\in\Delta(V_{N+1})$, set
$x_n=K_n\cdots K_Nx_{N+1}$ for $1\leq n\leq N$ and choose the
coordinates $x_n$ with $n>N+1$ arbitrarily.
Since
\[
 F_1\supseteq F_2\supseteq\cdots,
\]
compactness gives \(\bigcap_{N\geq1}F_N\neq\emptyset\).
Since $\mathcal H(B)=\bigcap_{N\geq1}F_N$, it is non-empty and compact.
Its defining equations are affine, so it is convex.
The ambient product $\prod_{n\geq1}\mathbb R^{V_n}$ is a Hausdorff
locally convex space, so the Krein--Milman theorem gives extreme points.
\end{proof}

\subsection{Central measures, ergodicity, and aperiodicity}
\label{subsec:central-measures}

Following \cite[Section~2.3, Definition~1]{Vershik2014},
a Borel probability measure $\mu$ on $X_B$ is \emph{central} if
\[
 \mu(C(p))=\mu(C(q))
\]
whenever $p,q\in E_{0,n}$ and $r(p)=r(q)$.
A central measure is \emph{tail ergodic} if every Borel union of tail classes
has measure $0$ or $1$; see \cite[Section~2.3]{Vershik2014}.

\medskip
Let $\alpha\in\mathcal H(B)$.
For $p\in E_{0,n}$ with $r(p)=v\in V_n$, set
\begin{equation}\label{eq:cylinder-measure}
 \mu_\alpha(C(p))
 :=\frac{\alpha_n(v)}{p_n(v)}.
\end{equation}
Harmonicity makes these masses consistent under cylinder refinement,
and their sum at each level is one.  The measure extension theorem
therefore gives a unique Borel probability measure $\mu_\alpha$ on $X_B$.
Since the value in \eqref{eq:cylinder-measure} depends only on the terminal
vertex of $p$, the measure $\mu_\alpha$ is central.

We write
\[
 \operatorname{Prob}_{\mathrm{cent}}(X_B)
\]
for the convex set of central Borel probability measures on $X_B$, equipped
with the weak-* topology.

The correspondence between central measures and the inverse limit of
the vertex simplices takes the following form in this normalization;
see \cite[Proposition~1 and the discussion following it]{Vershik2014}.
For structural results on tail-invariant measures when the sequence
$(|V_n|)$ is bounded, see \cite{BezuglyiKwiatkowskiMedynetsSolomyak2013}.

\begin{proposition}
\label{prop:harmonic-central-measure}
Let $B=(V,E)$ be a Bratteli diagram.
The map
\[
 \mathcal H(B)\longrightarrow
 \operatorname{Prob}_{\mathrm{cent}}(X_B),
 \qquad
 \alpha\longmapsto\mu_\alpha,
\]
is an affine homeomorphism.
Under this correspondence, extreme harmonic functions correspond precisely to
tail-ergodic central measures, and strictly positive harmonic functions
correspond precisely to central measures with full support.
\end{proposition}

\begin{proof}
For $\mu\in\operatorname{Prob}_{\mathrm{cent}}(X_B)$ and $v\in V_n$, define
\begin{equation}\label{eq:harmonic-vector-from-measure}
 \alpha_n^\mu(v):=\mu\Bigg(\bigsqcup_{\substack{p\in E_{0,n}\\r(p)=v}}C(p)\Bigg).
\end{equation}
These sets, indexed by $v\in V_n$, partition $X_B$, so
$\sum_v\alpha_n^\mu(v)=1$.
Centrality gives, for $r(p)=v$,
\begin{equation}\label{eq:central-cylinder-from-tower}
 \mu(C(p))=\frac{\alpha_n^\mu(v)}{p_n(v)}.
\end{equation}
Refinement by one edge therefore yields
\[
 \frac{\alpha_n^\mu(v)}{p_n(v)}
 =\sum_{w\in V_{n+1}}a_{w,v}
   \frac{\alpha_{n+1}^\mu(w)}{p_{n+1}(w)},
\]
which is \eqref{eq:harmonic} after multiplication by $p_n(v)$.
Thus $\alpha^\mu\in\mathcal H(B)$.
The cylinder formulas and uniqueness of the measure give
\[
 \alpha^{\mu_\alpha}=\alpha,
 \qquad \mu_{\alpha^\mu}=\mu.
\]
Both constructions are affine.

For a sequence $\alpha^{(j)}\in\mathcal H(B)$, the product topology and
\eqref{eq:cylinder-measure} give
\[
 \begin{aligned}
 \alpha^{(j)}\longrightarrow\alpha
 &\iff \alpha_n^{(j)}(v)\longrightarrow\alpha_n(v)
       \quad\text{for every }n\geq1,\ v\in V_n\\
 &\iff \mu_{\alpha^{(j)}}(C(p))\longrightarrow\mu_\alpha(C(p))
       \quad\text{for every finite path }p\\
 &\iff \mu_{\alpha^{(j)}}\xrightarrow{\mathrm{w}^*}\mu_\alpha.
 \end{aligned}
\]
The last equivalence holds because cylinder indicators are continuous and
their linear span is uniformly dense in $C(X_B)$.  Both spaces are
metrizable, so these equivalences prove that the correspondence is a
homeomorphism.

Central measures are precisely the measures invariant under finite prefix
permutations: these permutations exchange paths ending at the same vertex
and leave all subsequent edges unchanged.
Their orbits are the tail classes.
The standard equivalence between extremality and ergodicity for
these invariant measures therefore gives
\[
 \alpha\in\operatorname{ext}\mathcal H(B)
 \iff \mu_\alpha\in\operatorname{ext}
       \operatorname{Prob}_{\mathrm{cent}}(X_B)
 \iff \mu_\alpha\text{ is tail ergodic};
\]
see \cite[Section~2.3]{Vershik2014}.
Finally, since the cylinders form a basis and $p_n(v)>0$,
\[
 \alpha_n(v)>0\ \text{for all }n,v
 \iff \mu_\alpha(C(p))>0\ \text{for every finite path }p
 \iff \operatorname{supp}\mu_\alpha=X_B.
\]
\end{proof}

Following \cite[Section~2.2]{BezuglyiKarpel2016}, the diagram is
\emph{aperiodic} if every $\mathcal E_B$-class is infinite.
It is \emph{simple} if for every $n\geq0$ there exists $m>n$ such that
\[
 a_{w,v}^{(m,n)}>0,
 \qquad \forall v\in V_n,\ w\in V_m.
\]

The corresponding aperiodicity condition for a harmonic function is as follows.

\begin{definition}\label{def:aperiodic}
For $\alpha\in\mathcal H(B)$, the pair $(B,\alpha)$ is
\emph{$\alpha$-aperiodic} if, for every $L\geq1$,
\begin{equation}\label{eq:alpha-aperiodic}
  \lim_{n\to\infty}
  \sum_{\substack{v\in V_n\\p_n(v)<L}}\alpha_n(v)=0.
\end{equation}

\end{definition}

The tail relation is also the AF-equivalence relation of \cite[Definition~3.8
and Proposition~3.9]{Putnam2018}.
Moreover, as noted in \cite[Section~2.2]{BezuglyiKarpel2016}, the diagram is
simple if and only if $\mathcal E_B$ is minimal, meaning that every tail class
is dense in $X_B$.

\begin{remark}
  Aperiodicity is equivalent to
\[
 \min_{v\in V_n}p_n(v)\longrightarrow\infty.
\]
For $x=(x_1,x_2,\ldots)\in X_B$, let $[x]_n$ denote the paths agreeing
with $x$ after level $n$.
Then \(|[x]_n|=p_n(r(x_n))\) and \([x]_n\subseteq[x]_{n+1}\).
Thus $[x]_{\mathcal E_B} =\bigcup_{n\geq1}[x]_n$ is infinite if and only if
$p_n(r(x_n))\to\infty$.

For $L\geq1$, the clopen sets
\[
  U_n(L)=\Set*{ x\in X_B  \given p_n(r(x_n))\geq L  }
\]
are increasing in $n$.  If every tail class is infinite, then
$\bigcup_n U_n(L)=X_B$; compactness gives $U_N(L)=X_B$ for some $N$.
Since every vertex lies on an infinite path,
\[
 \min_{v\in V_n}p_n(v)\geq L\qquad \forall n\geq N.
\]
Conversely, this uniform divergence implies $|[x]_n|\to\infty$ for every
$x\in X_B$.
\end{remark}

\begin{remark}
  Fix an integer $L\geq1$ and put
\[
 D_n(L):=\Set*{x\in X_B\given p_n(r(x_n))<L}.
\]

  The set of paths passing through $v\in V_n$ has
  $\mu_\alpha$-measure $\alpha_n(v)$.
  Hence
\[
  \mu_\alpha\Set*{x\in X_B \given p_n(r(x_n))<L}
  =\sum_{\substack{v\in V_n\\p_n(v)<L}}\alpha_n(v).
\]
The sets $D_n(L)$ decrease with $n$.
Their intersection consists exactly of paths whose full tail class has
fewer than $L$ elements, since $[x]_n$ increases to $[x]_{\mathcal E_B}$.
Continuity of a finite measure from above therefore gives
\[
 \lim_n\sum_{\substack{v\in V_n\\p_n(v)<L}}\alpha_n(v)
 =\mu_\alpha\Set*{x\in X_B\given |[x]_{\mathcal E_B}|<L}.
\]
Since the finite tail classes are exhausted by the integer bounds $L$,
\[
 (B,\alpha)\text{ is }\alpha\text{-aperiodic}
 \iff |[x]_{\mathcal E_B}|=\infty
       \quad\text{for }\mu_\alpha\text{-almost every }x.
\]
\end{remark}

By \Cref{prop:harmonic-central-measure}, the hypotheses of the main
theorem are equivalent to
\[
 \mu_\alpha\text{ is tail ergodic},\qquad
 |[x]_{\mathcal E_B}|=\infty
       \quad\text{for }\mu_\alpha\text{-almost every }x.
\]

\begin{proposition}\label{prop:simple-positive}
If $B$ is simple, then every harmonic function is strictly positive.
If $B$ is aperiodic, then $(B,\alpha)$ is $\alpha$-aperiodic for every
$\alpha\in\mathcal H(B)$.
\end{proposition}

\begin{proof}
Fix $m\geq1$ and $v\in V_m$, and choose $n>m$ such that
$a_{w,v}^{(n,m)}>0$ for all $w\in V_n$.
Since $\sum_w\alpha_n(w)=1$, there is $w_0$ with $\alpha_n(w_0)>0$.
Nonnegativity of the terms in \eqref{eq:harmonic-long} gives
\[
 \alpha_m(v)\geq
 \frac{p_m(v)a_{w_0,v}^{(n,m)}}{p_n(w_0)}\alpha_n(w_0)>0.
\]
The second assertion follows because, for fixed $L$, the sum in
\eqref{eq:alpha-aperiodic} is eventually empty.
\end{proof}

The Pascal diagram below admits strictly positive extreme harmonic functions
satisfying the corresponding aperiodicity condition, although the diagram
is neither simple nor aperiodic.  For background and recent developments,
see \cite[Section~5.2]{Vershik2014} and \cite{BezuglyiDudkoKarpel2026}.

\begin{example}\label{ex:pascal-harmonic}
For the Pascal diagram, write $V_n=\{0,\ldots,n\}$ in level coordinates,
and join $k\in V_n$ to $k$ and $k+1$ in $V_{n+1}$ by one edge each.
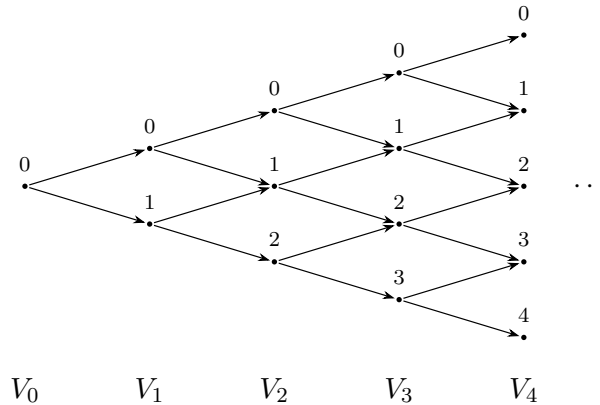
\begin{center}
  \begin{tikzpicture}[x=1.65cm,y=.5cm]
    \foreach \n in {0,1,2,3,4} {
      \foreach \k in {0,...,\n} {
        \fill (\n,\n-2*\k) circle [radius=1pt] coordinate (v\n-\k)
          node[above=2pt,font=\scriptsize] {$\k$};
      }
    }
    \foreach \n/\m in {0/1,1/2,2/3,3/4} {
      \foreach \k in {0,...,\n} {
        \pgfmathtruncatemacro{\l}{\k+1}
        \draw[bratteli edge] (v\n-\k) -- (v\m-\k);
        \draw[bratteli edge] (v\n-\k) -- (v\m-\l);
      }
    }
    \foreach \n in {0,1,2,3,4} {
      \node[below] at (\n,-4.8) {$V_{\n}$};
    }
    \node at (4.55,0) {$\cdots$};
  \end{tikzpicture}
\end{center}
Then
\[
 p_n(k)=\binom nk.
\]
For every $t\in[0,1]$, put
\[
 \alpha_n^t(k)=\binom nk t^k(1-t)^{n-k},
\]
with the convention $0^0=1$ in this formula.
The weights $\alpha_n^t$ have total mass one by the binomial theorem.
The only successors of $k$ are $k$ and $k+1$, and
\[
 \frac{\alpha_{n+1}^t(k)}{p_{n+1}(k)}
 +\frac{\alpha_{n+1}^t(k+1)}{p_{n+1}(k+1)}
 =t^k(1-t)^{n-k}
 =\frac{\alpha_n^t(k)}{p_n(k)}.
\]
Multiplying this equality by $p_n(k)$ gives the harmonic equation, so
$\alpha^t\in\mathcal H(B)$.
Identify a path with its sequence of increments in $\{0,1\}$.
The associated central measure is the Bernoulli measure with parameter $t$.
Tail equivalence coincides with equivalence under finite permutations
of the increments.  Indeed, paths with the same tail after level $n$
have the same number of ones among their first $n$ increments, which
therefore differ by a permutation.  Conversely, a permutation supported
on the first $n$ coordinates preserves their sum and all later
increments, hence preserves the path after level $n$.
The Hewitt--Savage zero--one law for the independent, identically
distributed increments implies tail ergodicity.  Extremality of
$\alpha^t$ follows from \Cref{prop:harmonic-central-measure}.
For $0<t<1$ and $n\geq\max\{2,L\}$, the inequality
$\binom nk\geq n$ for $1\leq k\leq n-1$ gives
\[
 \sum_{\substack{0\leq k\leq n\\p_n(k)<L}}\alpha_n^t(k)
 \leq t^n+(1-t)^n\longrightarrow0.
\]
Thus $(B,\alpha^t)$ is $\alpha^t$-aperiodic.
However, $p_n(0)=p_n(n)=1$ for every $n$, so $B$ itself is not aperiodic.
For $t\in\{0,1\}$, the measure is supported on a single path with
$p_n=1$ at every level, so the aperiodicity condition fails for $L>1$.

For $0<t<1$, all weights are positive, but $B$ is not simple:
for every $m>1$, the vertex $1\in V_1$ cannot reach $0\in V_m$.
Thus strict positivity and extremality of a harmonic function do not
imply simplicity of the diagram.
\end{example}

\subsection{Support reduction and faithfulness}

The \emph{support subdiagram} of $\alpha$ is denoted by
$B_\alpha=(V^\alpha,E^\alpha)$, where
\[
  V_n^\alpha=\Set*{v\in V_n \given \alpha_n(v)>0},
  \qquad \forall n\geq0.
\]
In particular, $V_0^\alpha=V_0$ by the convention $\alpha_0(v_0)=1$.
Its edges are
\[
 E_n^\alpha=\Set*{e\in E_n\given r(e)\in V_n^\alpha}.
\]
Equip $E_n^\alpha$ with the restriction of the fixed order on $E_n$.

For $n\geq1$, $v\in V_{n-1}$, and $w\in V_n$, harmonicity gives
\[
 a_{w,v}^{(n,n-1)}\alpha_n(w)>0
 \quad\Longrightarrow\quad\alpha_{n-1}(v)>0.
\]
Conversely, for $v\in V_n$, $\alpha_n(v)>0$ implies
$a_{w,v}^{(n+1,n)}\alpha_{n+1}(w)>0$ for some $w\in V_{n+1}$.
Every support vertex is therefore reachable from the root and emits an
edge in $B_\alpha$, so $B_\alpha$ is a Bratteli diagram.

\begin{proposition}\label{prop:support-reduction}
Let $B_\alpha$ be the support subdiagram of $\alpha$.
\begin{enumerate}[label=(\alph*), ref=\alph*]
\item\label{item:support-harmonic} The restriction $\alpha|_{B_\alpha}$ is a strictly positive harmonic
  function, and the path counts of vertices in $B_\alpha$ are the same
  as their path counts in $B$.
\item\label{item:support-rank} Deleting the zero-weight blocks induces a rank-preserving surjection
  \[
    A(B,R)\longrightarrow A(B_\alpha,R)
  \]
  and a rank-preserving isomorphism
  \begin{equation}\label{eq:support-completion}
    \overline A_\alpha(B,R)
    \cong
    \overline A_{\alpha|_{B_\alpha}}(B_\alpha,R).
  \end{equation}
\item\label{item:support-face} Extension by zero identifies $\mathcal H(B_\alpha)$ with the closed face
  \[
    F_\alpha=
    \Set*{\beta\in\mathcal H(B) \given
      \beta_n(v)=0\text{ whenever }\alpha_n(v)=0}.
  \]
  In particular,
  \[
    \alpha\in\operatorname{ext}\mathcal H(B)
    \quad\Longleftrightarrow\quad
    \alpha|_{B_\alpha}\in\operatorname{ext}\mathcal H(B_\alpha).
  \]
\item\label{item:support-aperiodic} The pair $(B,\alpha)$ is $\alpha$-aperiodic if and only if
  $(B_\alpha,\alpha|_{B_\alpha})$ is $\alpha|_{B_\alpha}$-aperiodic.
\end{enumerate}
\end{proposition}

\begin{proof}
  If \(a_{w,v}>0\) and \(\alpha_{n+1}(w)>0\), then by harmonicity,
\[
 \alpha_n(v)\geq
 \frac{p_n(v)a_{w,v}}{p_{n+1}(w)}\alpha_{n+1}(w)>0.
\]
Iteration shows that, for every $w\in V_m^\alpha$, all paths from
the root to $w$ lie in $B_\alpha$.
Thus the path counts are unchanged, and deleting zero-weight terms
from the harmonic equations proves (\ref{item:support-harmonic}).

\medskip
For $n\geq1$, let
\[
 I_n^\alpha=
 \bigoplus_{\substack{v\in V_n\\\alpha_n(v)=0}}\Mat_{p_n(v)}(R).
\]
Since no edge joins a zero-weight vertex to a positive-weight vertex,
$\phi_{n+1}(I_n^\alpha)\subseteq I_{n+1}^\alpha$.
Thus the quotient maps
\[
 A_n(B,R)\longrightarrow A_n(B,R)/I_n^\alpha
 \cong A_n(B_\alpha,R)
\]
commute with the connecting maps.
These maps induce a surjective $R$-algebra homomorphism
$q\colon A(B,R)\to A(B_\alpha,R)$.
Deletion of zero-weight components leaves
\eqref{eq:weighted-matrix-rank} unchanged.
Hence, for every rectangular matrix $X$,
\[
 \rk_{\alpha|_{B_\alpha}}(q(X))=\rk_\alpha(X).
\]
The surjection $q$ therefore induces an isometric isomorphism between the
quotients by the rank-zero ideals.
This isomorphism and its inverse extend continuously to the completions and
preserve matrix ranks, proving \eqref{eq:support-completion}.

\medskip
Let $\beta\in\mathcal H(B_\alpha)$ and extend it by zero outside the support.
By (\ref{item:support-harmonic}), the harmonic equations at support
vertices are preserved.  Outside the support, both sides vanish
because no edge leads into the support.  Normalization is unchanged.
Restriction is its inverse on $F_\alpha$.
Both maps are affine and continuous in the product topology.

The coordinate conditions defining $F_\alpha$ are closed.
For $0<t<1$ and $\beta,\gamma\in\mathcal H(B)$, nonnegativity gives
\[
 t\beta_n(v)+(1-t)\gamma_n(v)=0
 \ \Longrightarrow\ \beta_n(v)=\gamma_n(v)=0.
\]
Thus $t\beta+(1-t)\gamma\in F_\alpha$ implies
$\beta,\gamma\in F_\alpha$, so $F_\alpha$ is a closed face.
Since $\alpha\in F_\alpha$, every convex decomposition of $\alpha$ in
$\mathcal H(B)$ lies in this face.  Therefore $\alpha$ is extreme in $\mathcal H(B)$
if and only if it is extreme in $F_\alpha$.  The affine bijection with
$\mathcal H(B_\alpha)$ proves (\ref{item:support-face}).

\medskip
Finally, the sums in \eqref{eq:alpha-aperiodic} are unchanged after zero-weight
terms are deleted, and the relevant path counts agree by (\ref{item:support-harmonic}), proving (\ref{item:support-aperiodic}).
\end{proof}

\begin{proposition}\label{prop:weighted-rank-faithful}
Assume that the Sylvester matrix rank function $\rk$ on $R$ is faithful.
Then the induced rank $\rk_\alpha$ on $A(B,R)$ is faithful if and only if
$\alpha$ is strictly positive.
In particular, it is faithful whenever $B$ is simple.
\end{proposition}

\begin{proof}
Suppose that $\alpha$ is strictly positive and that $0\neq x=(x_v)_{v\in
V_n}\in A_n(B,R)$.
There exists $v$ with $x_v\neq0$, and therefore
\[
  \rk_{\alpha,n}(x)
  \geq \frac{\alpha_n(v)}{p_n(v)}\rk(x_v)>0.
\]
Injectivity of the connecting maps and the coordinate inequalities
from \Cref{subsec:sylvester-ranks} imply faithfulness on all matrices
over the direct limit.
Conversely, if $\alpha_n(v)=0$, then
\[
 \phi_{\infty,n}(e_{n,v})\neq0,\qquad
 \rk_\alpha(\phi_{\infty,n}(e_{n,v}))=\alpha_n(v)=0.
\]
The final assertion follows from \Cref{prop:simple-positive}.
\end{proof}

\begin{example}\label{ex:positive-not-compact}
Let the root have one edge to each of two vertices $u_1,w_1$, and, for every
$n\geq1$, let there be two edges from $u_n$ to $u_{n+1}$ and two edges from
$w_n$ to $w_{n+1}$, with no edges between the two branches.
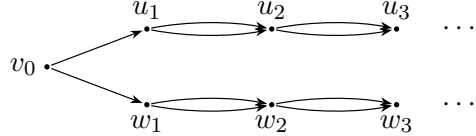
\begin{center}
  \begin{tikzpicture}[x=1.65cm]
    \fill (0.2,0) circle [radius=1pt] node[left] {\(v_{0}\)} coordinate (r);

    \foreach \n in {1,2,3} {
      \fill (\n,.5) circle [radius=1pt] node[above] {\(u_{\n}\)} coordinate (u\n);
      \fill (\n,-.5) circle [radius=1pt] node[below] {\(w_{\n}\)} coordinate (w\n);
    }

    \draw[bratteli edge] (r) -- (u1);
    \draw[bratteli edge] (r) -- (w1);

    \foreach \n/\m in {1/2,2/3} {
      \foreach \b in {u,w} {
        \foreach \h in {-3,3} {
          \draw[bratteli edge] (\b\n)
            .. controls ([xshift=15pt,yshift=\h pt]\b\n)
            and ([xshift=-15pt,yshift=\h pt]\b\m) .. (\b\m);
        }
      }
    }

    \node at (3.5,.5) {\(\cdots\)};
    \node at (3.5,-.5) {\(\cdots\)};
\end{tikzpicture}
\end{center}
Then
\[
  p_n(u_n)=p_n(w_n)=2^{n-1},
\]
so the diagram is aperiodic.
The harmonic equations reduce to
\[
 \alpha_n(u_n)=\alpha_{n+1}(u_{n+1}),\qquad
 \alpha_n(w_n)=\alpha_{n+1}(w_{n+1})
 \qquad(n\geq1).
\]
Together with normalization, these equations show that every harmonic
function is uniquely determined by $t=\alpha_1(u_1)\in[0,1]$:
\[
  \alpha_n^t(u_n)=t,
  \qquad
  \alpha_n^t(w_n)=1-t.
\]
Conversely, these weights define a harmonic function for every $t\in[0,1]$.
The maps $t\mapsto\alpha^t$ and $\alpha\mapsto\alpha_1(u_1)$ are inverse
continuous affine maps, hence $\mathcal H(B)\cong[0,1]$.
The extreme harmonic functions are $\alpha^0$ and $\alpha^1$.
Both pairs $(B,\alpha^0)$ and $(B,\alpha^1)$ satisfy the corresponding
aperiodicity condition, but neither harmonic function is strictly positive.

For $0<t<1$, the support subdiagram is $B$ itself.  At $t=0$ it
consists of the $w$-branch, and at $t=1$ of the $u$-branch.
Thus support reduction deletes precisely the branch of weight zero.
\end{example}

\subsection{Factor sequences}
\label{subsec:factor-sequence-models}

\begin{definition}\label{def:factor-sequence}
An unbounded sequence $\mu=(n_i)_{i\geq1}$ of positive integers is a
\emph{factor sequence} if $n_i\mid n_{i+1}$ for every $i\geq1$.
\end{definition}

For positive integers $q,r$, write
\[
 \jmath_{q,r}\colon\Mat_q(R)\longrightarrow\Mat_{qr}(R),
 \qquad y\longmapsto1_r\otimes y.
\]
For every rectangular matrix $Y$ over $\Mat_q(R)$, block additivity
and permutation invariance give
\[
 \rk_{qr}(\jmath_{q,r}(Y))
 =\frac{r\rk(Y)}{qr}=\rk_q(Y).
\]
Given a factor sequence $\mu$, set $n_0=1$ and
$d_i=n_i/n_{i-1}$ for $i\geq1$.  The maps
\[
 \phi_i:=\jmath_{n_{i-1},d_i}\colon
 \Mat_{n_{i-1}}(R)\longrightarrow\Mat_{n_i}(R)
 \qquad(i\geq1)
\]
define a direct system starting with $\Mat_{n_0}(R)=R$.
Write
\[
 R_\mu:=\varinjlim_{i\geq0}(\Mat_{n_i}(R),\phi_{i+1}).
\]
The normalized matrix ranks $\rk_{n_i}$ are compatible with these maps
and induce a Sylvester matrix rank $\rk_\mu$ on $R_\mu$.
Adjoining the initial stage $R$ does not change the ranked direct limit,
since the stages indexed by $i\geq1$ form a cofinal subsystem.

Let $B_\mu$ have one vertex $v_i$ at each level $i\geq0$ and $d_i$ edges
from $v_{i-1}$ to $v_i$:
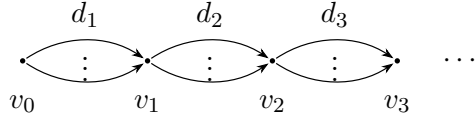
\begin{center}
\begin{tikzpicture}[x=1.65cm]
  \foreach \n in {0,1,2,3} {
    \fill (\n,0) circle [radius=1pt] node[below=9pt] {\(v_{\n}\)} coordinate (v\n);
  }
  \foreach \n/\m in {0/1,1/2,2/3} {
    \foreach \h in {-10,10} {
      \draw[bratteli edge] (v\n)
        .. controls ([xshift=15pt,yshift=\h pt]v\n)
        and ([xshift=-15pt,yshift=\h pt]v\m) .. (v\m);
    }
    \node[inner sep=0pt] at ($(v\n)!0.5!(v\m)$) {\(\vdots\)};
    \node[above=9pt] at ($(v\n)!0.5!(v\m)$) {\(d_{\m}\)};
  }
  \node at (3.5,0) {\(\cdots\)};
\end{tikzpicture}
\end{center}
The bundle labelled $d_i$ represents $d_i$ edges.  Then
\[
 p_i(v_i)=\prod_{j=1}^i d_j=n_i,
 \qquad A_i(B_\mu,R)=\Mat_{n_i}(R).
\]
The Bratteli connecting maps are exactly the maps $\phi_i$ above.
Normalization determines the unique harmonic function
$\alpha_i(v_i)=1$, and its weighted ranks are $\rk_{n_i}$.
Thus
\[
 (A(B_\mu,R),\rk_\alpha)\cong(R_\mu,\rk_\mu).
\]
For each path $p$ of length $i$, the cylinder measure is
$\mu_\alpha(C(p))=1/n_i$.
Since $n_i\mid n_{i+1}$ and $(n_i)$ is unbounded, $n_i\to\infty$.
Hence $\alpha$ is extreme and $(B_\mu,\alpha)$ is $\alpha$-aperiodic.

For the factor sequence $\mu=(2^k)_{k\geq1}$, the connecting maps are
\begin{equation}\label{eq:dyadic-system}
 \jmath_{2^k,2}\colon\Mat_{2^k}(R)\longrightarrow\Mat_{2^{k+1}}(R),
 \qquad x\longmapsto1_2\otimes x=\diag(x,x)
 \qquad(k\geq0).
\end{equation}
We denote the rank completion of $R_\mu$ with respect to
$\rk_\mu$ by
\[
 \mathcal M_{R,\rk}:=\overline{R_\mu}.
\]

\section{The uniqueness theorem}\label{sec:uniqueness}

Fix a Bratteli diagram $B$ and $\alpha\in\mathcal H(B)$.
We suppress $\alpha$ from the notation for the auxiliary maps.
The proof of \Cref{thm:main} uses homomorphisms of the
following form, with $1\leq m<n$ and $q,r\geq1$:
\begin{center}
\begin{tikzpicture}[x=2.6cm,>=Stealth]
  \node (A) at (0,0) {$A_m$};
  \node (B) at (1,0) {$\Mat_q(R)$};
  \node (C) at (2,0) {$A_n$};
  \node (D) at (3,0) {$\Mat_{qr}(R)$};
  \draw[->] (A) -- node[above] {$\rho$} (B);
  \draw[->] (B) -- node[above] {$\sigma$} (C);
  \draw[->] (C) -- node[above] {$\rho'$} (D);
  \draw[->] (A) to[bend left=32] node[above] {$\phi_{n,m}$} (C);
  \draw[->] (B) to[bend right=32] node[below] {$\jmath_{q,r}$} (D);
\end{tikzpicture}
\end{center}
Here $\jmath_{q,r}$ is the connecting map from
\Cref{subsec:factor-sequence-models}.
Extremality and $\alpha$-aperiodicity provide the estimates needed
to choose these maps with small rank distortions and small errors in
$\sigma\rho-\phi_{n,m}$ and $\rho'\sigma-\jmath_{q,r}$
(\Cref{lem:homogeneous,lem:tall,lem:finite-approximation}).
Iterating along cofinal sequences, with matrix sizes chosen to be
powers of $2$ and errors summable on each fixed matrix space,
gives a rank-preserving $R$-algebra isomorphism
\[
 \overline A_\alpha(B,R)\cong\mathcal M_{R,\rk}
\]
by \Cref{prop:summable}.

\subsection{Path distributions and aperiodicity}\label{subsec:asymptotic-homogeneity}

For $1\leq m<n$, $v\in V_m$, and $w\in V_n$, define
\begin{equation}\label{eq:beta}
 \beta_{w,v}^{(n,m)}
 :=\frac{p_m(v)a_{w,v}^{(n,m)}}{p_n(w)},
 \qquad
 \beta_w^{(n,m)}
 :=\bigl(\beta_{w,v}^{(n,m)}\bigr)_{v\in V_m}
 \in\Delta(V_m).
\end{equation}
The vector $\beta_w^{(n,m)}$ is the distribution of the level-$m$ vertex
of a uniformly chosen path from the root to $w$.
Equation \eqref{eq:path-count} gives
\[
 \sum_{v\in V_m}\beta_{w,v}^{(n,m)}=1.
\]

Equation \eqref{eq:harmonic-long} becomes
\begin{equation}\label{eq:barycenter}
 \alpha_m=\sum_{w\in V_n}\alpha_n(w)\beta_w^{(n,m)}.
\end{equation}
The coefficients also satisfy the composition rule, for \(1\leq k<r<n\)
\begin{equation}\label{eq:beta-composition}
 \beta_{w,u}^{(n,k)}
 =\sum_{z\in V_r}
 \beta_{w,z}^{(n,r)}\beta_{z,u}^{(r,k)}.
\end{equation}
Using the definition of the coefficients and path concatenation gives
\[
 \begin{aligned}
 \sum_{z\in V_r}\beta_{w,z}^{(n,r)}\beta_{z,u}^{(r,k)}
 &=\sum_{z\in V_r}
   \frac{p_r(z)a_{w,z}^{(n,r)}}{p_n(w)}
   \frac{p_k(u)a_{z,u}^{(r,k)}}{p_r(z)}\\
 &=\frac{p_k(u)}{p_n(w)}
   \sum_{z\in V_r}a_{w,z}^{(n,r)}a_{z,u}^{(r,k)}\\
 &=\frac{p_k(u)a_{w,u}^{(n,k)}}{p_n(w)}
 =\beta_{w,u}^{(n,k)}.
 \end{aligned}
\]
For a vector $z$ on the finite set $V_m$, we use $\|z\|_1=\sum_{v\in
V_m}|z(v)|$.

Extremality implies that these path distributions converge in weighted
average to $\alpha_m$ for each fixed $m$; compare
Vershik \cite[Section~2.5]{Vershik2014}.

\begin{lemma}\label{lem:homogeneous}
Let $B=(V,E)$ be a Bratteli diagram and let $\alpha\in\mathcal H(B)$ be extreme.
Then, for every fixed $m\geq1$,
\begin{equation}\label{eq:homogeneous}
 \lim_{n\to\infty}
 \sum_{w\in V_n}\alpha_n(w)
 \bigl\|\beta_w^{(n,m)}-\alpha_m\bigr\|_1=0.
\end{equation}
\end{lemma}

\begin{proof}
Fix $m\geq1$, put $N=|V_m|$, and write
\[
 d_n:=\sum_{w\in V_n}\alpha_n(w)
       \|\beta_w^{(n,m)}-\alpha_m\|_1.
\]
Suppose that $d_n$ does not converge to zero.
There exist $\eta>0$ and a strictly increasing sequence $(n_j)_j$ such that
$d_{n_j}\geq\eta$ for every $j$.

Two applications of Cauchy--Schwarz, using
$\sum_w\alpha_{n_j}(w)=1$ and $|V_m|=N$, give
\[
 \begin{aligned}
 \eta^2\leq d_{n_j}^2
 &\leq\sum_{w\in V_{n_j}}\alpha_{n_j}(w)
       \|\beta_w^{(n_j,m)}-\alpha_m\|_1^2\\
 &\leq N\sum_{v\in V_m}\sum_{w\in V_{n_j}}\alpha_{n_j}(w)
       (\beta_{w,v}^{(n_j,m)}-\alpha_m(v))^2.
 \end{aligned}
\]
By averaging over $V_m$, for each $j$ there exists
$v_j\in V_m$ such that
\[
 \sum_{w\in V_{n_j}}\alpha_{n_j}(w)
       (\beta_{w,v_j}^{(n_j,m)}-\alpha_m(v_j))^2
 \geq\frac{\eta^2}{N^2}.
\]
Finiteness of $V_m$ yields a subsequence on which $v_j=v_*$ is constant;
retain the notation $n_j$ for this subsequence.
Set $c_w^{(j)}:=\beta_{w,v_*}^{(n_j,m)}-\alpha_m(v_*)$.  Then
\begin{equation}\label{eq:positive-barycentric-variance}
 \sum_{w\in V_{n_j}}\alpha_{n_j}(w)(c_w^{(j)})^2
 \geq\frac{\eta^2}{N^2}.
\end{equation}
Both $\beta_{w,v_*}^{(n_j,m)}$ and $\alpha_m(v_*)$ lie in $[0,1]$, so
$|c_w^{(j)}|\leq1$.  Moreover,
\begin{equation}\label{eq:centered-barycentric-coefficients}
 \sum_{w\in V_{n_j}}\alpha_{n_j}(w)c_w^{(j)}
 =\sum_w\alpha_{n_j}(w)\beta_{w,v_*}^{(n_j,m)}
   -\alpha_m(v_*)\sum_w\alpha_{n_j}(w)
 =\alpha_m(v_*)-\alpha_m(v_*)=0.
\end{equation}
The last equality uses \eqref{eq:barycenter} and normalization of the weights.

For $1\leq k<n_j$ and $u\in V_k$, define
\begin{equation}\label{eq:finite-coherent-perturbation}
 h_k^{(j)}(u):=\sum_{w\in V_{n_j}}
       \alpha_{n_j}(w)c_w^{(j)}\beta_{w,u}^{(n_j,k)}.
\end{equation}

For $1\leq k<r$, define the linear map
\[
 K_{k,r}\colon\mathbb R^{V_r}\longrightarrow\mathbb R^{V_k},
 \qquad
 (K_{k,r}x)(u)
 :=\sum_{z\in V_r}x(z)\beta_{z,u}^{(r,k)}.
\]

For $k<r<n_j$ and $u\in V_k$, interchanging finite sums gives
\[
 \begin{aligned}
 (K_{k,r}h_r^{(j)})(u)
 &=\sum_{z\in V_r}\sum_{w\in V_{n_j}}
   \alpha_{n_j}(w)c_w^{(j)}\beta_{w,z}^{(n_j,r)}
   \beta_{z,u}^{(r,k)}\\
 &=\sum_{w\in V_{n_j}}\alpha_{n_j}(w)c_w^{(j)}
   \left(\sum_{z\in V_r}\beta_{w,z}^{(n_j,r)}\beta_{z,u}^{(r,k)}\right)\\
 &=\sum_{w\in V_{n_j}}\alpha_{n_j}(w)c_w^{(j)}\beta_{w,u}^{(n_j,k)}
 =h_k^{(j)}(u).
 \end{aligned}
\]
The third equality is \eqref{eq:beta-composition}.
Thus, for \(k<r<n_j\)
\begin{equation}\label{eq:finite-coherence-h}
 h_k^{(j)}=K_{k,r}h_r^{(j)}.
\end{equation}
Also, each $\beta_w^{(n_j,k)}$ has total mass one, so
\[
 \sum_{u\in V_k}h_k^{(j)}(u)
 =\sum_{w\in V_{n_j}}\alpha_{n_j}(w)c_w^{(j)}
       \sum_{u\in V_k}\beta_{w,u}^{(n_j,k)}
 =\sum_{w\in V_{n_j}}\alpha_{n_j}(w)c_w^{(j)}=0
\]
by \eqref{eq:centered-barycentric-coefficients}.
Finally, the triangle inequality, $|c_w^{(j)}|\leq1$, and nonnegativity of
the coefficients give
$|h_k^{(j)}(u)|\leq\sum_w\alpha_{n_j}(w)|c_w^{(j)}|\beta_{w,u}^{(n_j,k)}$.
Using \eqref{eq:barycenter} at level $k$ yields
\begin{equation}\label{eq:finite-domination-h}
 |h_k^{(j)}(u)|\leq\sum_{w\in V_{n_j}}
       \alpha_{n_j}(w)\beta_{w,u}^{(n_j,k)}=\alpha_k(u).
\end{equation}

For fixed $k$ and $n_j>k$, \eqref{eq:finite-domination-h} gives
\[
 h_k^{(j)}\in\prod_{u\in V_k}[-\alpha_k(u),\alpha_k(u)],
\]
a compact subset of $\mathbb R^{V_k}$.
Successive subsequence extraction gives nested infinite index sets
$J_1\supseteq J_2\supseteq\cdots$ such that $h_k^{(j)}$ converges along
$J_k$.  Choose $j_\ell\in J_\ell$ with $j_\ell>j_{\ell-1}$ and
$n_{j_\ell}>\ell$.  For each fixed $k$, all terms with $\ell\geq k$
belong to $J_k$; hence $h_k^{(j_\ell)}\to h_k$.
Retain the index $j$ for this diagonal subsequence.

For fixed $k<r$, the equality \eqref{eq:finite-coherence-h} holds for all
sufficiently large $j$, since $n_j\to\infty$.
Continuity of the finite-dimensional linear map $K_{k,r}$,
together with the zero-sum identity and the coordinate bounds, yields
\begin{equation}\label{eq:coherent-tangent-vector}
 h_k=K_{k,r}h_r,\qquad
 \sum_{u\in V_k}h_k(u)=0,\qquad |h_k(u)|\leq\alpha_k(u).
\end{equation}
The limit is nonzero: by \eqref{eq:centered-barycentric-coefficients},
\[
 h_m^{(j)}(v_*)
 =\sum_w\alpha_{n_j}(w)c_w^{(j)}
       (c_w^{(j)}+\alpha_m(v_*))
 =\sum_w\alpha_{n_j}(w)(c_w^{(j)})^2
 \geq\frac{\eta^2}{N^2}.
\]
Passing to the limit gives
\[
 h_m(v_*)\geq\frac{\eta^2}{N^2}>0.
\]
Define $\gamma_k=\alpha_k+\tfrac12h_k$ and
$\delta_k=\alpha_k-\tfrac12h_k$.
For every $u\in V_k$, the coordinate bound gives
\[
 \gamma_k(u),\delta_k(u)
 \geq\alpha_k(u)-\tfrac12|h_k(u)|\geq\tfrac12\alpha_k(u)\geq0.
\]
Their total masses are
$\sum_u\alpha_k(u)\pm\tfrac12\sum_u h_k(u)=1$.
The harmonicity of $\alpha$ and the coherence of $h$ give, for $k<r$,
\[
 K_{k,r}(\alpha_r\pm\tfrac12h_r)
 =\alpha_k\pm\tfrac12h_k.
\]
In particular the one-step harmonic equations hold for both families,
so $\gamma,\delta\in\mathcal H(B)$.
Since $\gamma_m(v_*)-\delta_m(v_*)=h_m(v_*)>0$, the decomposition
$\alpha=\tfrac12(\gamma+\delta)$ contradicts extremality.
Hence $d_n\to0$.
\end{proof}

\begin{lemma}\label{lem:tall}
Assume $(B,\alpha)$ is $\alpha$-aperiodic.
For every fixed integer $q\geq1$,
\[
 \lim_{n\to\infty}
 \sum_{w\in V_n}\alpha_n(w)
 \frac{p_n(w)-q\lfloor p_n(w)/q\rfloor}{p_n(w)}=0.
\]
\end{lemma}

\begin{proof}
The remainder $s_n(w)=p_n(w)-q\lfloor p_n(w)/q\rfloor$ satisfies
$0\leq s_n(w)<q$ and $s_n(w)\leq p_n(w)$.
For fixed $L\geq1$, the ratio $s_n(w)/p_n(w)$ is bounded by $1$
when $p_n(w)<L$ and by $q/L$ otherwise.  Therefore
\[
 \begin{aligned}
 0\leq\sum_w\alpha_n(w)\frac{s_n(w)}{p_n(w)}
 &\leq\sum_{p_n(w)<L}\alpha_n(w)
   +\frac qL\sum_{p_n(w)\geq L}\alpha_n(w)\\
 &\leq\sum_{p_n(w)<L}\alpha_n(w)+\frac qL.
 \end{aligned}
\]
Given $\varepsilon>0$, choose $L$ with $q/L<\varepsilon/2$.
For this $L$, \eqref{eq:alpha-aperiodic} bounds the first sum by
$\varepsilon/2$ for all sufficiently large $n$, proving convergence.
\end{proof}

\subsection{Construction of the auxiliary homomorphisms}
\label{subsec:finite-matrix-models}

A possibly nonunital ring homomorphism $f\colon S\to T$ between
unital $R$-algebras is called \emph{$R$-compatible} if
\[
 f\bigl(\iota_S(r)x\iota_S(s)\bigr)
 =\iota_T(r)f(x)\iota_T(s),
 \qquad \forall r,s\in R,\ \forall x\in S.
\]
A unital ring homomorphism is $R$-compatible if and only if it is
an $R$-algebra homomorphism.

Copying matrix blocks and inserting zero blocks preserve the left and
right coefficient actions.  Conjugation by a permutation matrix is also
$R$-compatible, since
\[
 P(1_k\otimes r)=(1_k\otimes r)P,
 \qquad \forall r\in R
\]
for every permutation matrix $P\in\Mat_k(R)$.
The auxiliary homomorphisms below are compositions of these
operations and are therefore $R$-compatible.

\medskip
\paragraph{\textbf{The map $\rho$.}}
Fix a level $m\geq1$ and a positive integer $q$.
For $v\in V_m$, put
\[
 c_{m,q}(v)
 :=\left\lfloor\frac{q\alpha_m(v)}{p_m(v)}\right\rfloor,
\]
and set
\begin{equation}\label{eq:floor-c}
 \lambda_{m,q}(v)
 :=\frac{c_{m,q}(v)p_m(v)}q,
 \qquad
 \eta_{m,q}
 :=1-\sum_{v\in V_m}\lambda_{m,q}(v).
\end{equation}
Since
\[
 q\eta_{m,q}
 =q-\sum_{v\in V_m}c_{m,q}(v)p_m(v)
\]
is a nonnegative integer, the formula
\[
 \rho_{m,q}\colon A_m\longrightarrow \Mat_q(R),
 \qquad
 \rho_{m,q}((x_v)_v)
 :=\diag\bigl((1_{c_{m,q}(v)}\otimes x_v)_{v\in V_m},0_{q\eta_{m,q}}\bigr)
\]
defines a ring homomorphism.
The summands follow the vertex order fixed in the preliminaries; terms
with $c_{m,q}(v)=0$ are omitted.  Changing this order conjugates
$\rho_{m,q}$ by a permutation matrix.
For a permutation matrix $P\in\Mat_q(R)$, put
\[
 \rho_{m,q}^{P}:=\Ad(P)\rho_{m,q}.
\]

\medskip
\paragraph{\textbf{The map $\sigma$.}}
Fix $1\leq m<n$ and $q\geq1$.
For $w\in V_n$, put
\[
 t_w:=\left\lfloor\frac{p_n(w)}q\right\rfloor,
 \qquad
 s_w:=p_n(w)-t_wq,
 \qquad
 \theta_w:=\frac{t_wq}{p_n(w)}=1-\frac{s_w}{p_n(w)},
\]
and define
\[
 \sigma_{n,q,w}^0(y)
 :=\diag(1_{t_w}\otimes y,0_{s_w})
 \in\Mat_{p_n(w)}(R).
\]
Here $1_0\otimes y$ and $0_0$ are interpreted as empty blocks.

For $w\in V_n$ and $v\in V_m$, let $J_{w,v,a}$,
$1\leq a\leq a_{w,v}^{(n,m)}$, be the coordinate inclusion
matrices of the copies of $x_v$ in \eqref{eq:bratteli-composite-map},
and let $K_{w,v,b}$, $1\leq b\leq t_wc_{m,q}(v)$, be those in
$\sigma_{n,q,w}^0\rho_{m,q}(x)$.
Thus
\begin{equation}\label{eq:alignment-coordinate-copies}
 \begin{aligned}
 (\phi_{n,m}(x))_w
 &=\sum_{v\in V_m}\sum_{a=1}^{a_{w,v}^{(n,m)}}
       J_{w,v,a}x_vJ_{w,v,a}^{T},\\
 \sigma_{n,q,w}^0\rho_{m,q}(x)
 &=\sum_{v\in V_m}\sum_{b=1}^{t_wc_{m,q}(v)}
       K_{w,v,b}x_vK_{w,v,b}^{T}.
 \end{aligned}
\end{equation}
The matrices $J_{w,v,a}$ use the path order in
\eqref{eq:bratteli-composite-map}, induced by the orders on
$E_{m+1},\ldots,E_n$.  The matrices $K_{w,v,b}$ use the coordinate
positions in $\sigma_{n,q,w}^0\rho_{m,q}$.
Enumerate the copies of each $x_v$ in their respective coordinate order.
An \emph{$(m,n,q)$-alignment} is a family of permutation matrices
$U=(U_w)_{w\in V_n}$ satisfying
\begin{equation}\label{eq:alignment-coordinate-match}
 U_wK_{w,v,a}=J_{w,v,a},
 \qquad \forall v\in V_m,\ \forall 1\leq a\leq \min\{a_{w,v}^{(n,m)},t_wc_{m,q}(v)\}.
\end{equation}
Such matrices exist: the prescribed columns define a bijection between
two coordinate subsets of equal cardinality, and any bijection of their
complements completes it.  In particular, $U$ is independent of $x$ and
does not change the prescribed connecting map.
For an alignment $U$, define
\[
 \sigma_{m,n,q}^{U}(y)
 :=\bigl(\Ad(U_w)\sigma_{n,q,w}^0(y)\bigr)_{w\in V_n}
 \in A_n.
\]
Finally, when $P\in\Mat_q(R)$ is a permutation matrix, put
\begin{equation}\label{eq:aligned-packing-P}
 \sigma_{m,n,q}^{U,P}
 :=\sigma_{m,n,q}^{U}\Ad(P^{-1}).
\end{equation}
Then
\[
 \sigma_{m,n,q}^{U,P}\rho_{m,q}^{P}
 =\sigma_{m,n,q}^{U}\rho_{m,q}.
\]
The composition is independent of $P$.  The map $\rho'$ constructed
below has the form $\rho_{n,qr}^{P'}$ and is used as the next map
in the recursive construction.

\medskip
\paragraph{\textbf{The map $\rho'$.}}
Fix an integer $r\geq1$ and put
\begin{equation}\label{eq:return-multiplicity}
 D_{n,q,r}:=\sum_{w\in V_n}c_{n,qr}(w)t_w.
\end{equation}
The multiplicity satisfies
\[
 \frac{D_{n,q,r}}r
 =\sum_{w\in V_n}\lambda_{n,qr}(w)\theta_w
 \leq\sum_{w\in V_n}\lambda_{n,qr}(w)
 =1-\eta_{n,qr}\leq1.
\]
Consider $\rho_{n,qr}\colon A_n\to\Mat_{qr}(R)$.
The alignment $U_w$, chosen at the level of the $x_v$-blocks, need not preserve
the $q\times q$ block decomposition of $\sigma_{n,q,w}^0(y)$.
Define the permutation matrix
\[
 Q_1=\diag\bigl((1_{c_{n,qr}(w)}\otimes U_w^{-1})_{w\in V_n},
                  1_{qr\eta_{n,qr}}\bigr).
\]
Conjugation by $Q_1$ transforms
$\rho_{n,qr}\sigma_{m,n,q}^{U}(y)$ into a block diagonal matrix with
$c_{n,qr}(w)t_w$ copies of $y$ for each $w$.
Choose a permutation matrix $Q_2$, independent of $y$, that places these
$D_{n,q,r}$ copies in the first $D_{n,q,r}$ blocks of size $q$,
preserving the internal order of each copy.  The remaining coordinates are zero,
since
\[
 qr-\sum_{w\in V_n}c_{n,qr}(w)p_n(w)
 +\sum_{w\in V_n}c_{n,qr}(w)s_w
 =(r-D_{n,q,r})q.
\]
Define
\[
 \widetilde\rho_{m,n,q,r}^{U}
 :=\Ad(Q_2Q_1)\rho_{n,qr}
 \colon A_n\longrightarrow\Mat_{qr}(R).
\]
For the permutation $P$ used in $\rho_{m,q}^{P}$, define
\begin{equation}\label{eq:return-map}
 \rho_{m,n,q,r}^{U,P}
 :=\Ad(1_r\otimes P)\widetilde\rho_{m,n,q,r}^{U}.
\end{equation}
More explicitly, put $P'=(1_r\otimes P)Q_2Q_1$.
It is a product of permutation matrices, and the composition rule
$\Ad(A)\Ad(B)=\Ad(AB)$ gives
$\rho_{m,n,q,r}^{U,P}=\rho_{n,qr}^{P'}$.
For $z=P^{-1}yP$, the definitions of $Q_1,Q_2$ give
\[
 \widetilde\rho_{m,n,q,r}^{U}\sigma_{m,n,q}^{U}(z)
 =\diag(1_{D_{n,q,r}}\otimes z,0_{(r-D_{n,q,r})q}).
\]
Conjugation by $1_r\otimes P$ therefore yields
\begin{equation}\label{eq:return-map-composition}
 \rho_{m,n,q,r}^{U,P}\sigma_{m,n,q}^{U,P}(y)
 =\diag(1_{D_{n,q,r}}\otimes y,0_{(r-D_{n,q,r})q}),
 \qquad \forall y\in\Mat_q(R).
\end{equation}

\subsection{Rank estimates}

For a possibly nonunital homomorphism $f\colon S\to T$
between rings equipped with Sylvester matrix rank functions, and for $d\geq1$, set
\begin{equation}\label{eq:rank-distortion}
 \operatorname{dist}_d(f)
 :=\sup_{X\in \Mat_d(S)}
 \bigl|\rk_T(f(X))-\rk_S(X)\bigr|.
\end{equation}
Write $\operatorname{dist}(f):=\operatorname{dist}_1(f)$.
Since $f(1)$ is idempotent, \eqref{eq:idempotent-complement-rank} gives
\begin{equation}\label{eq:identity-controlled-by-distortion}
 \rk_T(1-f(1))
 =1-\rk_T(f(1))
 \leq\operatorname{dist}(f).
\end{equation}

For later use, set
\begin{equation}\label{eq:finite-error-parameters}
 \xi_{m,n}
 :=\sum_{w\in V_n}\alpha_n(w)
   \|\beta_w^{(n,m)}-\alpha_m\|_1,
 \qquad
 \tau_{n,q}
 :=\sum_{w\in V_n}\alpha_n(w)(1-\theta_w).
\end{equation}

The maps above are constructed by copying matrix blocks, inserting zero blocks,
and conjugating by permutation matrices.
By block additivity and permutation invariance, their ranks are
computed from the block multiplicities.
The composition errors are
estimated by counting the copies remaining after matched copies
cancel, with the appropriate normalization and vertex weights.
We have the following rank estimates.

\begin{lemma}
\label{lem:finite-model-estimates}
Fix $1\leq m<n$, $q,r\geq1$, a permutation matrix $P\in\Mat_q(R)$, and an
$(m,n,q)$-alignment $U$.
Then
\begin{equation}\label{eq:eta}
 0\leq\eta_{m,q}\leq\frac1q\sum_{v\in V_m}p_m(v).
\end{equation}
For every $d\geq1$,
\begin{equation}\label{eq:amplified-eta}
 \operatorname{dist}_d(\rho_{m,q}^{P})=d\eta_{m,q},
 \qquad
 \operatorname{dist}_d(\sigma_{m,n,q}^{U,P})=d\tau_{n,q},
 \qquad
 \operatorname{dist}_d(\rho_{m,n,q,r}^{U,P})=d\eta_{n,qr}.
\end{equation}
For every $X\in\Mat_d(A_m)$,
\begin{equation}\label{eq:parameter-first-square}
 \rk_{\alpha,n}\bigl(
   \phi_{n,m}(X)-
   \sigma_{m,n,q}^{U,P}\rho_{m,q}^{P}(X)
 \bigr)
 \leq d\bigl(\xi_{m,n}+\eta_{m,q}+\tau_{n,q}\bigr).
\end{equation}
For every $Y\in\Mat_d(\Mat_q(R))$,
\begin{equation}
   \label{eq:parameter-second-square}
  \rk_{qr}\bigl(
   \jmath_{q,r}(Y)-
   \rho_{m,n,q,r}^{U,P}\sigma_{m,n,q}^{U,P}(Y)
 \bigr)
 \leq d\bigl(\eta_{n,qr}+\tau_{n,q}\bigr).
\end{equation}
\end{lemma}

\begin{proof}
The floor inequality gives, for every $v\in V_m$,
  \[
\frac{c_{m,q}(v)p_m(v)}q\leq\alpha_m(v)<\frac{p_m(v)}q + \frac{c_{m,q}(v)p_m(v)}q,
\]
Equivalently,
\[
0\leq\alpha_m(v) -  \lambda_{m,q}(v)<\frac{p_m(v)}q.
\]
Moreover,
\[
  \eta_{m,q}
  = 1- \sum_{v\in V_m}\lambda_{m,q}(v)
 =\sum_{v\in V_m}\bigl(\alpha_m(v)-\lambda_{m,q}(v)\bigr)\leq\frac1q\sum_{v\in V_m}p_m(v).
\]

For $X\in\Mat_d(A_m)$, by definition we have
\[
  \rho_{m,q}^{P}(X)
  =(1_d\otimes P)\rho_{m,q}(X)(1_d\otimes P^{-1}).
\]
Block additivity and permutation invariance give
\[
 \begin{aligned}
 \rk_q(\rho_{m,q}^{P}(X))
   & =  \rk_q (\rho_{m,q}(X))\\
   & =\frac1q\sum_{v\in V_m}c_{m,q}(v)\rk(X_v)\\
 &=\sum_{v\in V_m}\frac{c_{m,q}(v)p_m(v)}q
        \rk_{p_m(v)}(X_v)\\
 &=\sum_{v\in V_m}\lambda_{m,q}(v)\rk_{p_m(v)}(X_v).
 \end{aligned}
\]
Since
\[
\rk_{\alpha,m}(X)-\rk_q(\rho_{m,q}^{P}(X))=\sum_v(\alpha_m(v)-\lambda_{m,q}(v))\rk_{p_m(v)}(X_v),
\]
we have
\[
0\leq \rk_{\alpha,m}(X)-\rk_q(\rho_{m,q}^{P}(X))\leq d\eta_{m,q}.
\]
The first inequality follows from
$0\leq\lambda_{m,q}(v)\leq\alpha_m(v)$, and the second from
$0\leq\rk_{p_m(v)}(X_v)\leq d$.
At $X=1_d$, every component satisfies $\rk_{p_m(v)}(X_v)=d$,
so the bound is attained.  Thus
\[
 \operatorname{dist}_d(\rho_{m,q}^{P})=d\eta_{m,q}.
\]
Since $\rho_{m,n,q,r}^{U,P}=\rho_{n,qr}^{P'}$, the same calculation
with $(n,qr,P')$ gives, for $Z\in\Mat_d(A_n)$,
\[
  \sum_{w\in V_n}(\alpha_n(w)-\lambda_{n,qr}(w))
  \rk_{p_n(w)}(Z_w)\leq d\eta_{n,qr}.
\]
Again $Z=1_d$ attains equality, so
\[
\operatorname{dist}_d(\rho_{m,n,q,r}^{U,P})=d\eta_{n,qr}.
\]

For $Y\in\Mat_d(\Mat_q(R))$, precomposition by $\Ad(P^{-1})$
preserves its rank.
After flattening and conjugation by a permutation matrix independent of $Y$,
the $w$-component of $\sigma_{m,n,q}^{U,P}(Y)$ is block diagonal with
$t_w$ copies of $Y$ and a zero block.  Hence
\[
 \rk_{\alpha,n}(\sigma_{m,n,q}^{U,P}(Y))
 =\sum_w\frac{\alpha_n(w)}{p_n(w)}t_w\rk(Y)
 =\rk_q(Y)\sum_w\alpha_n(w)\frac{t_wq}{p_n(w)}.
\]
Thus,
\[
\rk_q(Y) - \rk_{\alpha,n}(\sigma_{m,n,q}^{U,P}(Y)) =
\rk_q(Y)\left(1-\sum_w\alpha_n(w)\frac{t_wq}{p_n(w)}\right) =
\rk_q(Y)\sum_w\alpha_n(w) (1-\theta_w),
\]
the last equality using $\sum_w\alpha_n(w)=1$.
Since $0\leq\rk_q(Y)\leq d$, this difference lies in $[0,d\tau_{n,q}]$, and the bound is
attained at $Y=1_d\in\Mat_d(\Mat_q(R))$.
This proves
\[
\operatorname{dist}_d(\sigma_{m,n,q}^{U,P})=d\tau_{n,q}.
\]

For \eqref{eq:parameter-first-square}, cancellation of
$\Ad(P^{-1})\Ad(P)$ in \eqref{eq:aligned-packing-P} reduces the composition to
$\sigma_{m,n,q}^{U}\rho_{m,q}$.
Put $\widehat J_{v,a}=1_d\otimes J_{w,v,a}$ and
$\widehat K_{v,b}=1_d\otimes(U_wK_{w,v,b})$ for each fixed $w$.
By \eqref{eq:alignment-coordinate-match},
$\widehat J_{v,a}=\widehat K_{v,a}$ for
$1\leq a\leq\min\{a_{w,v}^{(n,m)},t_wc_{m,q}(v)\}$.
Hence \eqref{eq:alignment-coordinate-copies}, applied entrywise, yields
\[
 \begin{aligned}
 &(\phi_{n,m}(X)-\sigma_{m,n,q}^{U}\rho_{m,q}(X))_w\\
 &=\sum_v\sum_{t_wc_{m,q}(v)<a\leq a_{w,v}^{(n,m)}}
       \widehat J_{v,a}X_v\widehat J_{v,a}^{T}\\
 &\quad-\sum_v\sum_{a_{w,v}^{(n,m)}<b\leq t_wc_{m,q}(v)}
       \widehat K_{v,b}X_v\widehat K_{v,b}^{T}.
 \end{aligned}
\]
For each $v$, the two sums contain altogether
$|a_{w,v}^{(n,m)}-t_wc_{m,q}(v)|$ terms.
Therefore,
\[
  \rk\bigl((\phi_{n,m}(X)-\sigma_{m,n,q}^{U}\rho_{m,q}(X))_w\bigr)\leq\sum_v|a_{w,v}^{(n,m)}-t_wc_{m,q}(v)|\rk(X_v).
\]
The flattened matrix $X_v$ has size $dp_m(v)\times dp_m(v)$, so
$\rk(X_v)\leq dp_m(v)$.  Therefore
\[
 \begin{aligned}
 \frac1{p_n(w)}\sum_v
 |a_{w,v}^{(n,m)}-t_wc_{m,q}(v)|\rk(X_v)
 &\leq d\sum_v
 \left|\frac{p_m(v)a_{w,v}^{(n,m)}}{p_n(w)}
       -\frac{p_m(v)t_wc_{m,q}(v)}{p_n(w)}\right|\\
 &=d\sum_v
       |\beta_{w,v}^{(n,m)}-\theta_w\lambda_{m,q}(v)|.
 \end{aligned}
\]
The last equality uses the definitions of $\beta$, $\theta_w$, and
$\lambda_{m,q}(v)$.

The triangle inequality gives
\[
 |\beta_{w,v}^{(n,m)}-\theta_w\lambda_{m,q}(v)|
 \leq|\beta_{w,v}^{(n,m)}-\alpha_m(v)|
       +|\alpha_m(v)-\lambda_{m,q}(v)|+(1-\theta_w)\lambda_{m,q}(v).
\]
The last term is nonnegative since $0\leq\theta_w\leq1$ and
$\lambda_{m,q}(v)\geq0$.

Summation over $v$ yields
\[
 \begin{aligned}
 \sum_v|\beta_{w,v}^{(n,m)}-\theta_w\lambda_{m,q}(v)|
 &\leq\|\beta_w^{(n,m)}-\alpha_m\|_1
       +\eta_{m,q}+(1-\theta_w)(1-\eta_{m,q})\\
 &\leq\|\beta_w^{(n,m)}-\alpha_m\|_1
       +\eta_{m,q}+(1-\theta_w).
 \end{aligned}
\]
The first inequality uses
$\sum_v|\alpha_m(v)-\lambda_{m,q}(v)|=\eta_{m,q}$ and
$\sum_v\lambda_{m,q}(v)=1-\eta_{m,q}$; the second uses
$0\leq\eta_{m,q}\leq1$.
Multiplying by $\alpha_n(w)$ and summing over $w\in V_n$ proves
\eqref{eq:parameter-first-square}.

Finally, apply \eqref{eq:return-map-composition} to
$Y=(y_{ij})\in\Mat_d(\Mat_q(R))$.
In the $(i,j)$-entry, the difference from $\jmath_{q,r}(Y)$ has its first $D_{n,q,r}$ copies equal to zero and
its remaining $r-D_{n,q,r}$ copies equal to $y_{ij}$.
The same coordinate permutation on rows and columns, ordering coordinates
first by the copy index and then by the matrix index, expresses the
difference as a block diagonal matrix with $r-D_{n,q,r}$ copies of $Y$ and a zero block.
Thus
\[
\rk_{qr}\bigl(
   \jmath_{q,r}(Y)-
   \rho_{m,n,q,r}^{U,P}\sigma_{m,n,q}^{U,P}(Y)
 \bigr)
 =\frac{r-D_{n,q,r}}r\rk_q(Y).
\]
Moreover,
\[
1-\frac{D_{n,q,r}}r
 =1-\sum_{w\in V_n}\lambda_{n,qr}(w)\theta_w
 =\eta_{n,qr}
   +\sum_{w\in V_n}\lambda_{n,qr}(w)(1-\theta_w)
 \leq\eta_{n,qr}+\tau_{n,q},
\]
because $\lambda_{n,qr}(w)\leq\alpha_n(w)$.
Since $\rk_q(Y)\leq d$, this proves
\eqref{eq:parameter-second-square}.
\end{proof}

The preceding estimates allow the parameters to be chosen so that
the rank distortions and composition errors are arbitrarily small.

\begin{lemma}\label{lem:finite-approximation}
Fix positive integers $m,q$ and positive numbers $\varepsilon,\kappa$.
Suppose that $\alpha$ is extreme and $(B,\alpha)$ is $\alpha$-aperiodic.
Assume that
\begin{equation}\label{eq:input-floor-error}
 \eta_{m,q}<\varepsilon.
\end{equation}
Let $P\in\Mat_q(R)$ be a permutation matrix and let
$\rho=\rho_{m,q}^{P}\colon A_m\to\Mat_q(R)$.
Then there are $n$ with $m<n$, an integer $r\geq1$, and homomorphisms
\[
  \sigma\colon \Mat_q(R)\longrightarrow A_n,
  \qquad
  \rho'\colon A_n\longrightarrow \Mat_{qr}(R),
\]
with $\rho'=\rho_{n,qr}^{P'}$ for a permutation matrix
$P'\in\Mat_{qr}(R)$.
For every $d\geq1$, these maps satisfy
\begin{align}
 \rk_{\alpha,n}\bigl(\phi_{n,m}(X)-\sigma\rho(X)\bigr)
 &<3d\varepsilon,\qquad \forall X\in\Mat_d(A_m),
 \label{eq:first-square}\\
 \rk_{qr}\bigl(\jmath_{q,r}(Y)-\rho'\sigma(Y)\bigr)
 &<2d\varepsilon,\qquad \forall Y\in\Mat_d(\Mat_q(R)).
 \label{eq:second-square}
\end{align}
Moreover,
\begin{equation}\label{eq:finite-approximation-distortions}
 \operatorname{dist}_d(\sigma)=d\tau_{n,q}<d\varepsilon,
 \qquad
 \operatorname{dist}_d(\rho')=d\eta_{n,qr}
 <d\min\{\varepsilon,\kappa\}.
\end{equation}
All three homomorphisms $\rho$, $\sigma$, and $\rho'$ are $R$-compatible.
The indices $n$ and $r$ may be required to exceed any given bounds.
If $q$ is a power of $2$, then $r$ may also be chosen to be a power of $2$.
\end{lemma}

\begin{proof}
By \Cref{lem:homogeneous,lem:tall}, choose $n$ with $m<n$ such that
\begin{align}
 \xi_{m,n}&<\varepsilon,
 \label{eq:choose-hom}\\
 \tau_{n,q}&<\varepsilon.
 \label{eq:choose-tall}
\end{align}
Both estimates hold for all \(n\) sufficiently large.
Thus, for any given  $N_0$, one can choose $n>\max\{m,N_0\}$.
Choose an $(m,n,q)$-alignment $U$, and put \(\sigma:=\sigma_{m,n,q}^{U,P}\).
For these fixed $n,q$, \eqref{eq:eta} gives the explicit bound
\[
 \eta_{n,qr}\leq\frac1{qr}\sum_{w\in V_n}p_n(w).
\]
Choose $r$ larger than any prescribed lower bound and such that
\[
 r>\frac{\sum_{w\in V_n}p_n(w)}{q\min\{\varepsilon,\kappa\}}.
\]
We may choose $r$ to be a power of $2$, since these are unbounded.
Then
\begin{equation}\label{eq:output-floor-error}
 \eta_{n,qr}<\min\{\varepsilon,\kappa\}.
\end{equation}

Set
\[
 \rho':=\rho_{m,n,q,r}^{U,P}.
\]

The choices of $n,U,r$ are independent of the matrix size $d$.
For each $d\geq1$, \Cref{lem:finite-model-estimates} applies as follows.
By \eqref{eq:input-floor-error}, \eqref{eq:choose-hom}, and
\eqref{eq:choose-tall},
\[
 \rk_{\alpha,n}\bigl(\phi_{n,m}(X)-\sigma\rho(X)\bigr)
 \leq d\bigl(\xi_{m,n}+\eta_{m,q}+\tau_{n,q}\bigr)
 <3d\varepsilon,
\]
which proves \eqref{eq:first-square}.
Likewise, \eqref{eq:choose-tall} and \eqref{eq:output-floor-error} give
\[
 \rk_{qr}\bigl(\jmath_{q,r}(Y)-\rho'\sigma(Y)\bigr)
 \leq d\bigl(\eta_{n,qr}+\tau_{n,q}\bigr)
 <2d\varepsilon,
\]
which proves \eqref{eq:second-square}.
Finally, \eqref{eq:amplified-eta} gives
\[
 \operatorname{dist}_d(\sigma)=d\tau_{n,q}<d\varepsilon,
 \qquad
 \operatorname{dist}_d(\rho')=d\eta_{n,qr}
 <d\min\{\varepsilon,\kappa\}.
\]
Each map is constructed by copying matrix blocks, inserting zero
blocks, and conjugating by permutation matrices.  These operations
are $R$-compatible.
\end{proof}

\subsection{Approximate intertwining}
The preceding lemma constructs maps between selected Bratteli levels
and full matrix algebras whose compositions approximate the connecting
maps.  These maps need not be exactly compatible with the direct systems,
so they do not directly induce maps on the algebraic direct limits.
The following proposition shows that summable composition errors and rank
distortions yield an exact rank-preserving isomorphism after completion.
This is a rank-metric version of the approximate intertwining argument.
A closely related Cauchy-limit construction in the continuous-factor setting
appears in \cite[Lemma~2.4]{AraClaramunt2018}.
We give the proof here for general coefficient rings, with explicit control
of the rank distortions at every matrix size.

\begin{proposition}\label{prop:summable}
Let $(S_i,\phi_i)$ and $(T_i,\psi_i)$ be direct systems of unital $R$-algebras
with Sylvester matrix rank functions and rank-preserving $R$-algebra
connecting maps $\phi_{i+1}\colon S_i\to S_{i+1}$ and
$\psi_{i+1}\colon T_i\to T_{i+1}$.  Suppose there are
$R$-compatible homomorphisms
$\rho_i\colon S_i\to T_i$ and $\sigma_i\colon T_i\to S_{i+1}$ and positive numbers
$\delta_i$ with $\sum_i\delta_i<\infty$ such that, for every $i,d\geq1$, $X\in\Mat_d(S_i)$, and
$Y\in\Mat_d(T_i)$,
\begin{align}
 \rk_{S_{i+1}}\bigl(\phi_{i+1}(X)-\sigma_i\rho_i(X)\bigr)&\leq d\delta_i,\label{eq:sum-square1}\\
 \rk_{T_{i+1}}\bigl(\psi_{i+1}(Y)-\rho_{i+1}\sigma_i(Y)\bigr)&\leq d\delta_i,\label{eq:sum-square2}
\end{align}
and
\[
 \operatorname{dist}_d(\rho_i),\operatorname{dist}_d(\sigma_i)\leq d\delta_i.
\]
Then the rank completions of $\varinjlim S_i$ and $\varinjlim T_i$ are
isomorphic as unital $R$-algebras, and the isomorphism preserves the Sylvester
matrix rank function.
\end{proposition}

The composition estimates \eqref{eq:sum-square1} and
\eqref{eq:sum-square2} correspond to the triangles in
\[
\begin{tikzpicture}[>=Stealth,baseline=(current bounding box.center)]
  \node (s0) at (0,1.6) {$S_i$};
  \node (s1) at (3.3,1.6) {$S_{i+1}$};
  \node (s2) at (6.6,1.6) {$S_{i+2}$};
  \node (t0) at (0,0) {$T_i$};
  \node (t1) at (3.3,0) {$T_{i+1}$};
  \node (t2) at (6.6,0) {$T_{i+2}$};
  \draw[->] (s0) -- node[above] {$\phi_{i+1}$} (s1);
  \draw[->] (s1) -- node[above] {$\phi_{i+2}$} (s2);
  \draw[->] (t0) -- node[below] {$\psi_{i+1}$} (t1);
  \draw[->] (t1) -- node[below] {$\psi_{i+2}$} (t2);
  \draw[->] (s0) -- node[left] {$\rho_i$} (t0);
  \draw[->] (s1) -- node[right] {$\rho_{i+1}$} (t1);
  \draw[->] (s2) -- node[right] {$\rho_{i+2}$} (t2);
  \draw[->] (t0) -- node[above,sloped] {$\sigma_i$} (s1);
  \draw[->] (t1) -- node[above,sloped] {$\sigma_{i+1}$} (s2);
\end{tikzpicture}
\]
Each triangle commutes up to the rank error $d\delta_i$
(or $d\delta_{i+1}$) on $d\times d$ matrices.

\begin{proof}[Proof of \Cref{prop:summable}]
Write $S_\infty=\varinjlim S_i$ and $T_\infty=\varinjlim T_i$, and denote
their rank completions by $\overline{S_\infty}$ and $\overline{T_\infty}$.

We compare elements from different stages in the respective rank completions.
For $x\in S_i$ and $j\geq i$, set
\[
  u_j(x)=\rho_j(\phi_{j,i}(x))\in T_j.
\]
Writing $z=\phi_{j,i}(x)$ and comparing consecutive terms in $T_{j+1}$ gives
\begin{align}
 &\rk_{T_{j+1}}\bigl(\psi_{j+1}u_j(x)-u_{j+1}(x)\bigr)\notag\\*
 &\quad \leq
 \rk_{T_{j+1}}\bigl(\psi_{j+1}\rho_j(z)-\rho_{j+1}\sigma_j\rho_j(z)\bigr) +
 \rk_{T_{j+1}}\bigl(\rho_{j+1}(\sigma_j\rho_j(z)-\phi_{j+1}(z))\bigr)\notag\\
 &\quad \leq \delta_j+(\delta_j+\delta_{j+1})\notag\\
 &\quad =2\delta_j+\delta_{j+1}.
 \label{eq:F-consecutive}
\end{align}
The first term is at most $\delta_j$ by \eqref{eq:sum-square2}.
The second is at most
$\delta_j+\operatorname{dist}(\rho_{j+1})
\leq\delta_j+\delta_{j+1}$
by \eqref{eq:sum-square1} and the definition of rank distortion.

For $l>j\geq i$, the triangle inequality gives
\[
 d_{\overline{T_\infty}}(u_l(x),u_j(x))
 \leq\sum_{k=j}^{l-1}(2\delta_k+\delta_{k+1}).
\]
Since $(\delta_k)$ is summable, $(u_j(x))_j$ is Cauchy.  Set
$F_i(x)=\lim_{j\to\infty}u_j(x)$.
Each $\rho_j\phi_{j,i}$ is a ring homomorphism, so continuity of addition
and multiplication implies that $F_i$ is a ring homomorphism.
The identity $\phi_{j,i+1}\phi_{i+1}=\phi_{j,i}$ gives
$F_{i+1}\phi_{i+1}=F_i$.
Thus the maps $F_i$ induce a ring homomorphism
$F\colon S_\infty\to\overline{T_\infty}$.

The connecting maps preserve rank, and hence
\[
 \left|\rk_{T_j}(u_j(x))-\rk_{S_i}(x)\right|
 \leq\operatorname{dist}(\rho_j)\leq\delta_j.
\]
Since $\delta_j\to0$, rank continuity gives $\rk(F(x))=\rk(x)$.
Additivity now implies that $F$ induces an isometry on the rank-zero
quotient, and therefore extends uniquely to an isometric ring homomorphism
$F\colon\overline{S_\infty}\to\overline{T_\infty}$.
Moreover, since $u_j(1)$ is idempotent,
equation~\eqref{eq:idempotent-complement-rank} gives
\[
 \rk(1-u_j(1))=1-\rk(u_j(1))
 \leq\operatorname{dist}(\rho_j)\leq\delta_j.
\]
Passing to the limit yields $F(1)=1$.
For $a,b\in R$ and $x\in S_i$, $R$-compatibility gives
\[
 u_j(\iota_{S_i}(a)x\iota_{S_i}(b))
 =\iota_{T_j}(a)u_j(x)\iota_{T_j}(b).
\]
The coefficient elements have fixed images in the completion.
Taking limits, and then using density, proves $R$-compatibility of $F$.
Together with unitality, this makes $F$ an $R$-algebra homomorphism.

For $X\in\Mat_d(S_i)$, applying the same calculation at matrix
size $d$, using the assumed composition bounds and
$\operatorname{dist}_d(\rho_{j+1})\leq d\delta_{j+1}$,
gives successive errors bounded by $d(2\delta_j+\delta_{j+1})$.
By \eqref{eq:rank-continuity-bounds}, its limit in the matrix rank metric
is the entrywise amplification $F(X)$.
Rank continuity and the distortion bound give
\[
 \bigl|\rk(F(X))-\rk_{S_i}(X)\bigr|
 =\lim_{j\to\infty}
 \bigl|\rk_{T_j}(\rho_j\phi_{j,i}(X))-\rk_{S_i}(X)\bigr|
 \leq\lim_{j\to\infty}d\delta_j=0.
\]
Density extends this identity to matrices over $\overline{S_\infty}$.
For rectangular matrices, apply square rank preservation to
$\begin{pmatrix}0&X\\0&0\end{pmatrix}$, whose rank equals that of $X$
by permutation invariance and block additivity.
Thus $F$ preserves the Sylvester matrix rank function.

In the reverse direction, for $y\in T_i$ and $j\geq i$, set
\[
 v_j(y)=\sigma_j(\psi_{j,i}(y))\in S_{j+1}.
\]
The same estimates, with the index shift in the target of $\sigma_j$, give
\begin{equation}\label{eq:G-consecutive}
 \rk_{S_{j+2}}\bigl(\phi_{j+2}v_j(y)-v_{j+1}(y)\bigr)
 \leq\delta_j+2\delta_{j+1}.
\end{equation}
The preceding argument therefore defines a unital $R$-algebra homomorphism
$G\colon\overline{T_\infty}\to\overline{S_\infty}$ by
$G(y)=\lim_{j\to\infty}v_j(y)$, preserving all matrix ranks.

It remains to prove that $F$ and $G$ are inverse to each other.
For $x\in S_i$, put $z_j=\phi_{j,i}(x)$.
Summing \eqref{eq:G-consecutive} for the element $\rho_j(z_j)\in T_j$
and using \eqref{eq:sum-square1} gives
\[
 \begin{aligned}
 d\bigl(G(\rho_j(z_j)),x\bigr)
 &\leq d\bigl(G(\rho_j(z_j)),\sigma_j\rho_j(z_j)\bigr)
       +d\bigl(\sigma_j\rho_j(z_j),\phi_{j+1}(z_j)\bigr)\\
 &\leq\sum_{k=j}^{\infty}(\delta_k+2\delta_{k+1})+\delta_j
 \longrightarrow0.
 \end{aligned}
\]
Here $\phi_{j+1}(z_j)$ represents $x$ in the direct limit.
Since $\rho_j(z_j)\to F(x)$ and $G$ is continuous, $GF(x)=x$.
Similarly, \eqref{eq:F-consecutive} and \eqref{eq:sum-square2}
give $FG(y)=y$ for $y\in T_\infty$.
Density and continuity extend both identities to the completions,
proving the required rank-preserving unital $R$-algebra isomorphism.
\end{proof}

\subsection{Proof of the main theorem}

\begin{proof}[Proof of \Cref{thm:main}]
Choose positive numbers $\varepsilon_i$ with $\sum_i\varepsilon_i<\infty$;
for example, $\varepsilon_i=2^{-i}$.
We recursively select levels
\[
  n_1<n_2<\cdots,
\]
matrix sizes $q_i=2^{k_i}$ with $k_i\to\infty$, and homomorphisms
\[
  \rho_i\colon A_{n_i}\longrightarrow \Mat_{q_i}(R),
  \qquad
  \sigma_i\colon \Mat_{q_i}(R)\longrightarrow A_{n_{i+1}}.
\]
Each $\rho_i$ will have the form $\rho_{n_i,q_i}^{P_i}$ for a permutation
matrix $P_i\in\Mat_{q_i}(R)$, and the
induction also maintains
\begin{equation}\label{eq:recursive-floor-error}
 \eta_{n_i,q_i}<\varepsilon_i.
\end{equation}

Set $n_1=1$.  Since $V_1$ is finite, the number $\sum_{v\in V_1}p_1(v)$
is finite.  Choose
$q_1=2^{k_1}>\varepsilon_1^{-1}\sum_{v\in V_1}p_1(v)$.
Set $\rho_1=\rho_{n_1,q_1}$.  By \eqref{eq:eta},
$\eta_{n_1,q_1}<\varepsilon_1$, proving
\eqref{eq:recursive-floor-error} for $i=1$.
Given $n_i,q_i,\rho_i$ satisfying \eqref{eq:recursive-floor-error}, apply
\Cref{lem:finite-approximation} with $m=n_i$, $\varepsilon=\varepsilon_i$, and
$\kappa=\varepsilon_{i+1}$, requiring $n>i+1$ and choosing the resulting
$r=r_i$ to be a power of $2$ large enough that $q_ir_i=2^{k_{i+1}}$ with
$k_{i+1}>\max\{k_i,i+1\}$.
Set $q_{i+1}=q_ir_i$, $n_{i+1}=n$, $\sigma_i=\sigma$, and
$\rho_{i+1}=\rho'$.
Taking $d=1$ in \eqref{eq:finite-approximation-distortions} proves
\eqref{eq:recursive-floor-error} at stage
$i+1$, and $\rho_{i+1}$ has the required permutation form at
$(n_{i+1},q_{i+1})$.  Thus \Cref{lem:finite-approximation} can be applied at
the next stage.
The sequences $(n_i)$ and $(k_i)$ are strictly increasing and unbounded.
Thus the selected stages are cofinal in both systems, and
\Cref{prop:cofinal-direct-limit} applies.

The initial floor construction gives
$\operatorname{dist}_d(\rho_1)<d\varepsilon_1$.
At the recursive step, \eqref{eq:finite-approximation-distortions} and the choice
$\kappa=\varepsilon_{i+1}$ give
\[
 \operatorname{dist}_d(\sigma_i)<d\varepsilon_i,
 \qquad
 \operatorname{dist}_d(\rho_{i+1})<d\varepsilon_{i+1},
\]
and \eqref{eq:first-square} and \eqref{eq:second-square} give composition errors
bounded by
$3d\varepsilon_i$ and $2d\varepsilon_i$, respectively.
Apply \Cref{prop:summable} with
\[
 S_i=A_{n_i},\quad T_i=\Mat_{q_i}(R),\quad
 \phi_{i+1}=\phi_{n_{i+1},n_i},\quad
 \psi_{i+1}=\jmath_{q_i,r_i},
\]
and $\delta_i=3\varepsilon_i$.
The $S$-connecting maps preserve rank by \Cref{prop:rank-compatible};
the $T$-connecting maps preserve rank by
\Cref{subsec:factor-sequence-models}.
Both are unital $R$-algebra maps.
The maps $\rho_i,\sigma_i$ are $R$-compatible by \Cref{lem:finite-approximation}.
Their distortions are less than $d\varepsilon_i\leq d\delta_i$, and the
two composition errors are less than $3d\varepsilon_i=d\delta_i$ and
$2d\varepsilon_i\leq d\delta_i$.  Finally,
$\sum_i\delta_i=3\sum_i\varepsilon_i<\infty$.
All hypotheses of \Cref{prop:summable} are therefore satisfied.
That proposition produces a unital $R$-algebra isomorphism between the two rank
completions, preserving the Sylvester matrix rank function on every rectangular
matrix space.
Cofinality identifies the two completions with $\overline A_\alpha(B,R)$ and
$\mathcal M_{R,\rk}$, respectively.
\end{proof}

\section{Consequences and coefficient ranks}
\label{sec:consequences}

Throughout this section, fix a unital ring $R$ with a Sylvester
matrix rank $\rk$, and put $Q=\mathcal M_{R,\rk}$.

\subsection{Consequences of the main theorem}
\label{subsec:classical-comparison}

\begin{corollary}[Factor sequences]\label{cor:factor-sequence-uniqueness}
For every factor sequence $\mu$, there is a unital $R$-algebra isomorphism
\begin{equation}\label{eq:factor-sequence-completion}
 \overline{R_\mu}\cong\mathcal M_{R,\rk}
\end{equation}
preserving the specified ranks on all rectangular matrices.
Here $\overline{R_\mu}$ denotes the rank completion with respect to $\rk_\mu$.
\end{corollary}

\begin{proof}
By \Cref{subsec:factor-sequence-models}, $R_\mu$ is the ranked algebra
associated with $B_\mu$.  Its unique harmonic function $\alpha$ is
extreme, and $(B_\mu,\alpha)$ is $\alpha$-aperiodic.
The conclusion follows from \Cref{thm:main}.
\end{proof}

In particular, the isomorphism preserves the natural left and right
$R$-module structures considered in \cite{Halperin1968}.
More generally, \Cref{prop:simple-positive,thm:main} give
\[
 B\text{ simple and aperiodic},\quad
 \alpha\in\operatorname{ext}\mathcal H(B)
 \quad\Longrightarrow\quad
 \overline A_\alpha(B,R)\cong\mathcal M_{R,\rk}.
\]
Here simplicity gives strictly positive harmonic weights, and
aperiodicity implies $\alpha$-aperiodicity.

\begin{example}\label{ex:two-branch-rank-completions}
For the two-branch diagram in \Cref{ex:positive-not-compact},
\[
 A_n(B,R)=\Mat_{2^{n-1}}(R)\times\Mat_{2^{n-1}}(R),
\]
and the connecting maps repeat blocks independently on the two factors.
Thus, for $0<t<1$,
\begin{equation}\label{eq:two-branch-completion}
 \overline A_{\alpha^t}(B,R)\cong Q\times Q,\qquad
 \rk_{\alpha^t}(X,Y)=t\rk(X)+(1-t)\rk(Y).
\end{equation}
Indeed, a sequence is Cauchy for this weighted rank metric if and only
if both component sequences are Cauchy; the rank-zero quotient and
completion therefore give the product of two copies of $Q$.
At $t=0$ or $t=1$, the zero-weight branch is removed by
\Cref{prop:support-reduction}, giving
\[
 \overline A_{\alpha^0}(B,R)\cong Q
 \cong\overline A_{\alpha^1}(B,R).
\]

\end{example}

\begin{example}\label{ex:pascal-rank-completions}
For the Pascal diagram in \Cref{ex:pascal-harmonic}, \Cref{thm:main} gives
\[
 \overline A_{\alpha^t}(B,R)\cong Q\qquad(0<t<1).
\]
At $t=0$ and $t=1$, the support subdiagram has one vertex and one edge
at each level.  Its algebraic limit is $R$, so support reduction gives
\[
 \overline A_{\alpha^0}(B,R)
 \cong\overline{R_{\rk}}\cong\overline A_{\alpha^1}(B,R).
\]
\end{example}

\subsection{Coefficient rings and ranks}
\label{subsec:coefficient-ranks}

Following \cite[Section~2]{AraClaramunt2018}, a \emph{continuous factor}
is a simple, regular, left and right self-injective ring of type
$\mathrm{II}_f$.  Such a ring has a unique rank function, is complete
in its rank metric, and its scalar rank takes every value in $[0,1]$.

For a division ring $D$, \Cref{thm:main} is a consequence of
Ara and Claramunt \cite[Theorem~3.2]{AraClaramunt2018}.
Scalar pseudo-rank functions $N$ on $A(B,D)$ are affinely parametrized
by $\alpha\in\mathcal H(B)$ through
\[
 \alpha_n(v)=N(e_{n,v}),\qquad
 N((x_v)_v)=\sum_{v\in V_n}\frac{\alpha_n(v)}{p_n(v)}
                   \operatorname{rank}_D(x_v).
\]
Indeed, invertible row and column operations and orthogonal additivity
determine $N$ from its values on the block units, and compatibility
between stages is precisely the harmonic equation.
Conversely, the weighted Sylvester matrix rank restricts to this
pseudo-rank by \Cref{rem:pseudo-rank-dimension}.
Thus extreme harmonic functions correspond to extremal pseudo-ranks.
If $(B,\alpha)$ is $\alpha$-aperiodic, positive-weight blocks have
unbounded sizes; their minimal idempotents have ranks
$\alpha_n(v)/p_n(v)>0$ tending to zero along a suitable sequence.
Hence the corresponding pseudo-rank is non-discrete.

For extreme $\alpha$, the completion is regular and simple
\cite[Theorems~19.6 and 19.14]{Goodearl1991}, and non-discreteness
makes it a continuous factor \cite[Lemma~3(i)]{Halperin1968}.
Ara and Claramunt's theorem therefore gives the required $D$-ring
isomorphism.  It preserves the unique scalar rank and hence all matrix
ranks by \Cref{rem:pseudo-rank-dimension}.

Every division ring contains its center as a field.  More generally,
the field case extends to any coefficient ring containing a central
field, with no regularity assumption on the coefficient ring.
The following proof establishes this extension without using
\Cref{sec:uniqueness}.

\begin{proposition}
\label{prop:central-field-transfer}
Suppose that $K\subseteq Z(R)$ is a field with the same identity as $R$.
Let $\alpha$ be an extreme harmonic function on a Bratteli diagram $B$
such that $(B,\alpha)$ is $\alpha$-aperiodic.  Then there is a unital
$R$-algebra isomorphism
\[
 \overline A_\alpha(B,R)\cong\mathcal M_{R,\rk}
\]
preserving the specified ranks on all rectangular matrices.
\end{proposition}

\begin{proof}
Put
\[
 C=A(B,K),\qquad
 L_K=\varinjlim_k\Mat_{2^k}(K).
\]
By \eqref{eq:bratteli-connecting-map},
$A(B,R)=R\otimes_K C$ canonically.
Since $K$ is a field, $\rk$ restricts to its unique Sylvester matrix rank.
Thus $C\to \overline A_\alpha(B,R)$ preserves the weighted matrix ranks and extends to an
embedding $j:\overline C_\alpha\to \overline A_\alpha(B,R)$.
Its image commutes with $\iota_R(R)$ by centrality of $K$ and continuity.
The preceding discussion shows that $\overline C_\alpha$ is a continuous
factor and that the pseudo-rank on $C$ is extremal.
The image of $C$ in $\overline C_\alpha$ is a rank-dense
$K$-subalgebra of at most countable dimension.  Moreover, $C$ is
ultramatricial and $\overline C_\alpha$ is its completion with respect
to the extremal pseudo-rank.  Thus condition~(ii) of
\cite[Theorem~2.2]{AraClaramunt2018} gives a
rank-preserving $K$-algebra isomorphism
\[
 \psi:\mathcal M_K\longrightarrow\overline C_\alpha.
\]
For $q=2^k$, let $(e_{ij})$ be the matrix units of the corresponding
stage of $L_K$.  The formula
\[
 \theta_q:\Mat_q(R)\longrightarrow \overline A_\alpha(B,R),\qquad
 (r_{ij})\longmapsto\sum_{i,j}\iota_R(r_{ij})j\psi(e_{ij})
\]
defines a unital homomorphism, since these matrix units commute with
$\iota_R(R)$.  Moreover, $\theta_q(r1_q)=\iota_R(r)$.
The pullback of the matrix rank of $\overline A_\alpha(B,R)$ therefore restricts along the
diagonal embedding of $R$ to the specified rank, on all rectangular
matrices.  The correspondence with ranks on $\Mat_q(R)$
\cite[Proposition~1.4]{JaikinLopez2020} identifies this pullback with
$\rk_q$.  Hence
\[
 \rk(\theta_q(X))=\rk_q(X)
 \qquad\bigl(X\in\Mat_{a\times b}(\Mat_q(R))\bigr).
\]
For $x\in\Mat_q(R)$, we have
$\theta_{2q}(\diag(x,x))=\theta_q(x)$.
Thus the maps $\theta_q$ induce a rank-preserving map
$\varinjlim_k\Mat_{2^k}(R)\to \overline A_\alpha(B,R)$.

To prove density of the image, write an element $x\in A(B,R)$ as
$x=\sum_{i=1}^h\iota_R(r_i)b_i$, with $b_i$ the images of elements of $C$.
Since $\psi(L_K)$ is dense in $\overline C_\alpha$, choose $c_i\in L_K$
so that $j\psi(c_i)$ approximates $b_i$.  Taking all $c_i$ in a common
matrix stage shows that $\sum_i\iota_R(r_i)j\psi(c_i)$ lies in the image.
Moreover,
\[
 \rk\left(x-\sum_{i=1}^h\iota_R(r_i)j\psi(c_i)\right)
 \leq\sum_{i=1}^h\rk(b_i-j\psi(c_i)),
\]
which is arbitrarily small.  The map extends isometrically to
$\mathcal M_{R,\rk}$; its image is closed by completeness and dense by the
preceding estimate, hence equals $\overline A_\alpha(B,R)$.  It preserves the coefficient
map and all matrix ranks.
\end{proof}

An $R$-module $M$ has finite length if it admits a composition series
\[
 0=M_0\subsetneq M_1\subsetneq\cdots\subsetneq M_n=M
\]
with simple successive quotients (nonzero modules with no nonzero
proper submodules).  The number $n$ is independent of the chosen
series and is called the composition length $\ell_R(M)$.
Composition length is additive in short exact sequences; see
\cite[Chapter~6, Propositions~6.7--6.9]{AtiyahMacdonald1969}.
In the following two examples, $R$ is commutative with $\ell_R(R)=2$,
and normalized image length defines a Sylvester matrix rank function.

\begin{example}\label{ex:dual-number-rank}
Let $R=K[t]/(t^2)$ for a field $K$, and write $t$ also for its
residue class in $R$.
Right multiplication relative to the row basis $(1,t)$ gives
the unital embedding
\[
 \lambda:R\longrightarrow\Mat_2(K),\qquad
 \lambda(a+bt)=\begin{pmatrix}a&b\\0&a\end{pmatrix}.
\]
Pullback of normalized matrix rank defines the faithful Sylvester
matrix rank
\[
 \rk_\ell(X)=\tfrac12\operatorname{rank}_K(\lambda(X))
 =\tfrac12\dim_K(R^{1\times r}X),
 \qquad X\in\Mat_{r\times s}(R).
\]
The second equality identifies the same row map in the chosen basis.
The chain $0\subsetneq tR\subsetneq R$ is a composition series,
with both factors isomorphic to $R/(t)\cong K$.
Thus $\ell_R(R)=2$ and $\ell_R(tR)=1$.
Since every simple $R$-module is isomorphic to $K$,
composition length agrees with $K$-dimension for finite-dimensional
$R$-modules.  Hence
\[
 \rk_\ell(X)=\tfrac12\ell_R(R^{1\times r}X).
\]
In particular,
\[
 t^2=0,\qquad \rk_\ell(t)=\tfrac12.
\]
The ring $R$ is not von Neumann regular: since $t$ is central and
$t^2=0$, one has $tyt=0\ne t$ for every $y\in R$.
Nevertheless, $\rk_\ell$ is induced from the regular ring $\Mat_2(K)$.
Since $K\subseteq Z(R)$, \Cref{prop:central-field-transfer}
shows that, for every extreme harmonic function $\alpha$ on a
Bratteli diagram $B$ such that $(B,\alpha)$ is $\alpha$-aperiodic,
\[
 \overline A_\alpha(B,R)\cong\mathcal M_{R,\rk_\ell}
\]
as unital $R$-algebras, with all matrix ranks preserved.
\end{example}

\begin{example}
\label{ex:nonregular-rank}
Let $R=\mathbb Z/4\mathbb Z$.
Since $R$ has characteristic $4$, it contains no unital subfield.
The composition series $0\subsetneq 2R\subsetneq R$ has both
factors isomorphic to $R/2R$, so $\ell_R(R)=2$.
The formula
\[
 \rk_\ell(X):=\frac12\ell_R\bigl(\operatorname{im}(X:R^{s\times1}
        \to R^{r\times1})\bigr),
 \qquad X\in\Mat_{r\times s}(R),
\]
defines a faithful Sylvester matrix rank.  Indeed, length is additive on
direct sums, and $\operatorname{im}(XY)$ is both an image of
$\operatorname{im}Y$ and a submodule of $\operatorname{im}X$, giving
the product inequality.  For $Z=\left(\begin{smallmatrix}X&C\\0&Y\end{smallmatrix}\right)$,
projection onto the second coordinate maps $\operatorname{im}Z$ onto
$\operatorname{im}Y$, with kernel containing $\operatorname{im}X\oplus0$.
Length additivity in this exact sequence gives the block triangular
inequality.  Normalization and faithfulness follow from
$\ell_R(R)=2$ and positivity of the length of a nonzero image.
In particular,
\[
 \rk_\ell(2)=\tfrac12\ell_R(2R)=\tfrac12.
\]
This rank is not induced from a von Neumann regular ring
\cite[Example~2.1.13]{LopezAlvarez2021}; see also
\cite[Remark~4.6]{HungLi2023}.
Indeed, let $f:R\to S$ be a unital homomorphism with $S$ regular.  Then
$u=f(2)=2\cdot1_S$ is central and $u^2=0$.
For $b\in S$ with $u=ubu$, centrality gives $u=bu^2=0$.
Thus $f$ cannot preserve $\rk_\ell(2)=1/2$.
\end{example}

\begin{remark}\label{ex:nonregular-completion}
Both examples yield completions outside the class of continuous factors.
Consider either coefficient rank ring in
\Cref{ex:dual-number-rank,ex:nonregular-rank}, and put $c=t$
or $c=2$, respectively.  In both cases,
\[
 c\in Z(R),\qquad c^2=0,\qquad\rk_\ell(c)=\tfrac12.
\]
Identify $R$ with its coefficient image in $\mathcal M_{R,\rk_\ell}$,
which is injective by faithfulness of $\rk_\ell$.
The element $c$ commutes with every finite-stage matrix.  By continuity
of multiplication, it is central in the completion, where
$c^2=0$ and $\rk_\ell(c)=1/2$ still hold.
Regularity would give $c=cyc=yc^2=0$ for some $y$, a contradiction.
The nonzero ideal $c\mathcal M_{R,\rk_\ell}$ has square zero, so the completion
is also non-simple.  By \Cref{thm:main}, the same conclusions hold for
every Bratteli completion satisfying the extremality and
$\alpha$-aperiodicity hypotheses over either coefficient pair.
\end{remark}

\begin{example}\label{ex:coefficient-rank-dependence}
Fix $R=\mathbb Z$ and take the system \eqref{eq:dyadic-system}.
Equivalently, let $B$ have one vertex at each level and two edges
between successive levels.  Then
\[
 A(B,\mathbb Z)=\varinjlim_n\bigl(\Mat_{2^n}(\mathbb Z),
             \,X\mapsto\diag(X,X)\bigr).
\]
Its unique harmonic function $\alpha$ is extreme and
$(B,\alpha)$ is $\alpha$-aperiodic.  The ring $\mathbb Z$ contains no
unital subfield and admits faithful Sylvester matrix ranks giving
nonisomorphic completions, as follows.

Let $p_1,p_2,\ldots$ enumerate the odd primes.  For an integer matrix
$X$, set
\[
 \rk(X)=\operatorname{rank}_{\mathbb Q}(X),\qquad
 \rk'(X)=\sum_{j=1}^{\infty}2^{-j}
       \operatorname{rank}_{\mathbb F_{p_j}}(X\bmod p_j).
\]
Here $\mathbb F_{p_j}=\mathbb Z/p_j\mathbb Z$, and
$X\bmod p_j$ denotes entrywise reduction.
Both are Sylvester matrix ranks.  For the second, ranks of a fixed
$r\times s$ matrix are bounded by $\min\{r,s\}$, so the series
converges and the rank axioms pass to the positive weighted sum;
normalization follows from $\sum_j2^{-j}=1$.
Countable convex combinations of Sylvester ranks are also discussed
in \cite[Section~3]{JaikinLopez2020}.
Both ranks are faithful: if an entry of $X$ is a nonzero integer,
then $X$ has positive rank over $\mathbb Q$, and it remains nonzero
modulo every prime not dividing that entry.
Normalize each rank by $2^{-n}$ at stage $n$.  The resulting
completions are $\mathcal M_{\mathbb Z,\rk}$ and $\mathcal M_{\mathbb Z,\rk'}$.
We use $\rk$ and $\rk'$ also for the induced ranks on the direct
limit and its completions.  We distinguish the completions by
invertibility of $2\cdot1$.

For $Y\in\Mat_{2^n}(\mathbb Z)$,
$2Y-1_{2^n}\equiv-1_{2^n}\pmod2$, so its determinant is odd and
in particular nonzero.  The normalized rank therefore satisfies
\[
 2^{-n}\operatorname{rank}_{\mathbb Q}(2Y-1_{2^n})=1.
\]
Consequently $d_{\rk}(2b,1)=1$ for every $b\in A(B,\mathbb Z)$.
If $x\in \mathcal M_{\mathbb Z,\rk}$ satisfied $2x=1$, choose $b_k\in A(B,\mathbb Z)$
converging to $x$.  The product inequality would give
\[
 1=d_{\rk}(2b_k,1)
  =d_{\rk}(2b_k,2x)
  \leq d_{\rk}(b_k,x)\longrightarrow0,
\]
a contradiction.  Thus $2\cdot1$ is not invertible in $\mathcal M_{\mathbb Z,\rk}$.

For each $N$, the Chinese remainder theorem gives $b_N\in\mathbb Z$
with $2b_N\equiv1\pmod{p_j}$ for $1\leq j\leq N$.
For $M\geq N$, the residues of $b_M$ and $b_N$ agree at these primes,
and hence
\[
 \rk'(b_M-b_N)\leq\sum_{j>N}2^{-j}=2^{-N},
 \qquad
 \rk'(2b_N-1)\leq2^{-N}.
\]
The coefficient images of $b_N$ thus converge in $\mathcal M_{\mathbb Z,\rk'}$
to an element $b$ satisfying $2b=b2=1$.  Since any ring isomorphism
between unital rings preserves $2\cdot1$ and its invertibility,
\[
 \mathcal M_{\mathbb Z,\rk}\not\cong \mathcal M_{\mathbb Z,\rk'}.
\]
\end{example}

\subsection{Corner isomorphisms}
\label{subsec:prescribed-rank-corners}

For a rank ring $(S,\rk_S)$ and an idempotent $e\in S$ with
$\rk_S(e)>0$, the function
\[
 X\longmapsto\frac{\rk_S(X)}{\rk_S(e)}
 \qquad\bigl(X\in\Mat_{a\times b}(eSe)\bigr)
\]
is a Sylvester matrix rank on the unital ring $eSe$, whose identity is $e$.
Indeed, normalization holds at $e$, and the other rank axioms are
inherited from $S$ and preserved by positive rescaling.

For a unital regular rank ring $R$ whose rank completion has no
nontrivial central idempotents,
Halperin \cite[Theorem~2]{Halperin1968} gives a ring isomorphism
$Q\cong eQe$ for every nonzero idempotent $e\in Q$.
The proof first constructs corners of prescribed rank using
\cite[Lemma~1]{Halperin1968}, and then compares arbitrary idempotents
of equal rank by the dimension theory of continuous regular factors
\cite[Lemma~4(ii)]{Halperin1968}.

For general $(R,\rk)$, a nonzero corner need not be isomorphic to $Q$;
see \Cref{ex:product-corner-classification}.
For a factor sequence $\mu$, we also denote by $\rk_\mu$ the continuous
extension of its rank to the rank completion $\overline{R_\mu}$.

\begin{proposition}\label{prop:scaled-corner-embedding}
Let $\mu=(n_i)$ and $\nu=(m_i)$ be factor sequences such that
\[
 c_i:=\frac{n_i}{m_i}\searrow\theta>0,\qquad c_1\leq1.
\]
There is an idempotent $e\in\overline{R_\nu}$ commuting with the
coefficient image of $R$, with $\rk_\nu(e)=\theta$, and an isomorphism
\[
 \Phi:(\overline{R_\mu},\rk_\mu)
 \longrightarrow(e\overline{R_\nu}e,\theta^{-1}\rk_\nu)
\]
of ranked unital $R$-algebras, where the coefficient map of the corner
is $r\mapsto e\iota_\nu(r)$.
\end{proposition}

\begin{proof}
We use the construction of Halperin \cite[Lemma~1]{Halperin1968}.  Its rank estimates require only the
Sylvester rank axioms and apply to rectangular matrices as follows.
Put $k_i=n_{i+1}/n_i$ and $l_i=m_{i+1}/m_i$, so $k_i\leq l_i$, and set
\[
 J_1=\begin{pmatrix}1_{n_1}\\0\end{pmatrix},\qquad
 J_{i+1}=\begin{pmatrix}1_{k_i}\otimes J_i\\0\end{pmatrix},\qquad
 f_i(x)=J_ixJ_i^T,\qquad e_i=J_iJ_i^T.
\]
Here $J_i\in\Mat_{m_i\times n_i}(R)$ and $J_i^TJ_i=1_{n_i}$, so
$f_i$ is an isomorphism onto $e_i\Mat_{m_i}(R)e_i$.
Coordinate insertion and block additivity give, for every rectangular
matrix $X$ over $\Mat_{n_i}(R)$,
\[
 \begin{aligned}
 \rk_{m_i}(f_i(X))&=c_i\rk_{n_i}(X),\\
 \rk_{m_{i+1}}\bigl(
 \jmath_{m_i,l_i}f_i(X)-f_{i+1}\jmath_{n_i,k_i}(X)\bigr)
 &=(c_i-c_{i+1})\rk_{n_i}(X).
 \end{aligned}
\]
Here all maps act entrywise.  For an element $x$, the difference
consists of $l_i-k_i$ copies of $f_i(x)$ and zero blocks; the same
calculation applies to rectangular matrices after flattening and
coordinate permutations.

For $x\in\Mat_{n_i}(R)$, the second identity makes
\[
 \Phi([x]):=\lim_{j\to\infty}
 [f_j(1_{n_j/n_i}\otimes x)]
\]
well defined: the rank distance between the terms indexed by $h$ and
$j>h\geq i$ is at most $(c_h-c_j)\rk_{n_i}(x)$.
Changing a stage representative leaves the tail of this sequence unchanged.
Continuity of addition and multiplication therefore gives a ring
homomorphism on the algebraic limit.  The first identity gives
\[
 \rk_\nu(\Phi(X))=\theta\rk_\mu(X).
\]
Thus $\Phi$ descends through the rank-zero quotient and extends to the
rank completion, with closed image; continuity preserves this
identity on every fixed rectangular matrix space.

The diagonal idempotents $[e_i]$ decrease to $e=\Phi(1)$, with
$\rk_\nu(e)=\theta$ and $\rk_\nu([e_i]-e)=c_i-\theta$.
Since $\Phi$ is multiplicative, every element of its image satisfies
$\Phi(x)=e\Phi(x)e$, so the image lies in the corner
$e\overline{R_\nu}e$.
Each $e_i$ has entries in $\{0,1_R\}$.  Coordinate insertion and
passage to the limit give
\[
 e\iota_\nu(r)=\iota_\nu(r)e,\qquad
 \Phi(\iota_\mu(r))=e\iota_\nu(r)\quad(r\in R).
\]
This is the coefficient compatibility noted in
\cite[Remark~1]{Halperin1968}.
Finally, for $z=eze$, choose $a_i\in\Mat_{m_i}(R)$ with
$\rk_\nu(z-[a_i])\to0$ and put $b_i=J_i^Ta_iJ_i$.
Since $f_i(b_i)=e_ia_ie_i$ and $e[e_i]=[e_i]e=e$, compression gives
\[
 \rk_\nu(z-[e_ia_ie_i])
 \leq\rk_\nu(z-[a_i])+2\rk_\nu([e_i]-e).
\]
The bound for the tail of the defining sequence, applied to $b_i$, gives
\[
 \rk_\nu([f_i(b_i)]-\Phi([b_i]))
 \leq(c_i-\theta)\rk_{n_i}(b_i)\leq c_i-\theta.
\]
Combining these inequalities yields
\[
 \rk_\nu(z-\Phi([b_i]))
 \leq\rk_\nu(z-[a_i])+3(c_i-\theta)\longrightarrow0.
\]
Hence the image is dense in the corner; closedness proves surjectivity.
\end{proof}

\Needspace{12\baselineskip}
\begin{corollary}\label{cor:prescribed-rank-corners}
For every $0<\theta\leq1$, there is an
idempotent $e_\theta\in Q$ commuting with the coefficient image of $R$
such that
\[
 \rk(e_\theta)=\theta,\qquad
 (Q,\rk)\cong(e_\theta Qe_\theta,\theta^{-1}\rk)
\]
as ranked unital $R$-algebras, with corner coefficient map
$r\mapsto e_\theta\iota_Q(r)$.  In particular,
\[
 \Set*{\rk(e) \given e\in Q,\ e^2=e}=[0,1].
\]
The same conclusions hold for every completion in \Cref{thm:main}.
\end{corollary}

\begin{proof}
Choose positive integers $a_i,b_i$ with $a_i\leq b_i$ and
$a_i/b_i\searrow\theta$.  Starting with $n_0=m_0=1$, put
\[
 L_i=2n_{i-1}m_{i-1},\qquad n_i=a_iL_i,\qquad m_i=b_iL_i.
\]
Both sequences are factor sequences, since
\[
 \frac{n_i}{n_{i-1}}=2a_im_{i-1}\geq2,\qquad
 \frac{m_i}{m_{i-1}}=2b_in_{i-1}\geq2,
 \qquad \frac{n_i}{m_i}=\frac{a_i}{b_i}\searrow\theta.
\]
Apply \Cref{prop:scaled-corner-embedding} and identify both factor-sequence
completions with $Q$ by \eqref{eq:factor-sequence-completion}.  This proves the corner assertion;
$e=0$ supplies rank zero.  The final assertion follows by transporting
these idempotents and isomorphisms along \Cref{thm:main}.
\end{proof}

Let
\[
 C=\Set*{a\in Q\given
 a\iota_Q(r)=\iota_Q(r)a\text{ for every }r\in R}.
\]
For $e^2=e\in C$, the corner $eQe$ has coefficient map
$r\mapsto e\iota_Q(r)$.

\begin{proposition}\label{prop:corner-isomorphism-criteria}
Let $0\ne e=e^2\in C$ and put $\theta=\rk(e)$.
\begin{enumerate}[label=(\alph*), ref=\alph*]
\item\label{item:corner-necessary}
If there is an isomorphism
\begin{equation}\label{eq:normalized-corner-isomorphism}
 (Q,\rk)\cong(eQe,\theta^{-1}\rk)
\end{equation}
of ranked unital $R$-algebras, then
\begin{equation}\label{eq:corner-coefficient-ranks}
 \rk\bigl((e\iota_Q(x_{ij}))_{i,j}\bigr)
 =\theta\rk_R(X)
 \qquad\bigl(X=(x_{ij})\in\Mat_{a\times b}(R),\ a,b\geq1\bigr),
\end{equation}
where $\rk_R$ is the specified rank on $R$.
\item\label{item:corner-sufficient}
Let $e_\theta$ be an idempotent supplied by
\Cref{cor:prescribed-rank-corners}.
If there exist
\[
 u\in eCe_\theta,\qquad v\in e_\theta Ce,
 \qquad uv=e,\quad vu=e_\theta,
\]
then \eqref{eq:normalized-corner-isomorphism} holds.
\end{enumerate}
\end{proposition}

\begin{proof}
For (\ref{item:corner-necessary}), let $\Phi$ be such an isomorphism.
Since $\Phi$ preserves the coefficient map,
$\Phi(\iota_Q(r))=e\iota_Q(r)$ for every $r\in R$.
The coefficient map $\iota_Q$ preserves matrix ranks, so
\[
 \rk\bigl((e\iota_Q(x_{ij}))_{i,j}\bigr)
 =\rk(\Phi(\iota_Q(X)))
 =\theta\rk(\iota_Q(X))
 =\theta\rk_R(X).
\]

For (\ref{item:corner-sufficient}), the maps
\[
 \Psi:e_\theta Qe_\theta\longrightarrow eQe,\quad a\longmapsto uav,
 \qquad
 \Psi^{-1}:eQe\longrightarrow e_\theta Qe_\theta,\quad b\longmapsto vbu
\]
are inverse ring homomorphisms because $vu=e_\theta$ and $uv=e$.
Since $u,v\in C$,
\[
 \Psi(e_\theta\iota_Q(r))=u\iota_Q(r)v=e\iota_Q(r).
\]
Thus $\Psi$ is an $R$-algebra isomorphism.
For a rectangular matrix $A$ over $e_\theta Qe_\theta$, multiplication
by diagonal matrices with entries $u$ and $v$, followed by the same
argument for $\Psi^{-1}$, gives
\[
 \rk(\Psi(A))\leq\rk(A)
 =\rk(\Psi^{-1}\Psi(A))\leq\rk(\Psi(A)).
\]
Hence $\Psi$ preserves every matrix rank.  Composing it with the
isomorphism from \Cref{cor:prescribed-rank-corners} proves the claim.
\end{proof}

\begin{example}\label{ex:product-corner-classification}
Let $R=\mathbb F_2\times\mathbb F_3$ with matrix rank
\[
 \rk_R(X_2,X_3)
 =\tfrac12\bigl(\operatorname{rank}_{\mathbb F_2}(X_2)+\operatorname{rank}_{\mathbb F_3}(X_3)\bigr).
\]
Put $Q_p=\mathcal M_{\mathbb F_p}$ and write $\rk_{Q_p}$ for its matrix rank
in this example.  The finite-stage product decompositions commute
with the connecting maps, and a sequence is rank Cauchy if and only if
both component sequences are rank Cauchy.  Therefore
\[
 Q=Q_2\times Q_3,\qquad
 \rk(X_2,X_3)=\tfrac12\bigl(\rk_{Q_2}(X_2)+\rk_{Q_3}(X_3)\bigr).
\]
Here $C$ denotes the commutant of the coefficient image defined
above.  Since that image is central, $C=Q$.
For $0\ne e=(e_2,e_3)=e^2$, set
\[
 t_p=\rk_{Q_p}(e_p)\quad(p=2,3),\qquad
 \theta=\rk(e)=\tfrac12(t_2+t_3).
\]
Then
\begin{equation}\label{eq:product-corner-criterion}
 (Q,\rk)\cong(eQe,\theta^{-1}\rk)
 \text{ as ranked unital }R\text{-algebras}
 \quad\Longleftrightarrow\quad t_2=t_3>0.
\end{equation}
Indeed, applying \eqref{eq:corner-coefficient-ranks} to $(1,0)$ and
$(0,1)$ gives $t_2/2=\theta/2=t_3/2$, proving necessity.
For sufficiency, suppose that $t_2=t_3=\theta$ and choose
$e_\theta=(h_2,h_3)$ as in \Cref{cor:prescribed-rank-corners}.
The same necessary condition gives $\rk_{Q_p}(h_p)=\theta$ for $p=2,3$.
In each continuous factor $Q_p$, idempotents of equal rank are
equivalent: there exist
\[
 u_p\in e_pQ_ph_p,\qquad v_p\in h_pQ_pe_p,
 \qquad u_pv_p=e_p,\quad v_pu_p=h_p;
\]
see the proof of \cite[Lemma~4(ii)]{Halperin1968}.
The elements $u=(u_2,u_3)$ and $v=(v_2,v_3)$ satisfy
\Cref{prop:corner-isomorphism-criteria}(\ref{item:corner-sufficient}),
which proves sufficiency.

If only a ring isomorphism is required, then
\[
 Q\cong eQe\quad\Longleftrightarrow\quad e_2\ne0\text{ and }e_3\ne0.
\]
When both components are nonzero, apply
\cite[Theorem~2]{Halperin1968} to each factor.
If one component is zero, $eQe$ has characteristic $2$ or $3$,
whereas $Q$ has characteristic $6$.
In particular, $e=(1_{Q_2},0)$ and the idempotent $e_{1/2}$ from
\Cref{cor:prescribed-rank-corners} both have rank $1/2$, but only the
latter has a corner isomorphic to $Q$.
\end{example}

\Needspace{10\baselineskip}
\section{Comparison of completions}
\label{sec:af-completions}

In this section, we specialize to complex coefficients and use
constructions and results from operator algebra theory to clarify
the relationships among the various completions and closures of
$A(B,\mathbb C)$.

\subsection{\texorpdfstring{$C^*$}{C*}-completion and classification}
\label{subsec:af-dimension-traces}

Fix a Bratteli diagram $B$ and write
\[
 A_n=A_n(B,\mathbb C)
 =\bigoplus_{v\in V_n}\Mat_{p_n(v)}(\mathbb C),\qquad
 A(B,\mathbb C)=\varinjlim_n A_n.
\]
Each $A_n$ has its usual involution and norm
$\|(x_v)_v\|=\max_v\|x_v\|$.  The connecting maps preserve
the involution and norm, giving the $C^*$-completion
\[
 \mathcal A_B=\overline{A(B,\mathbb C)}^{\,\|\cdot\|}.
\]
It is a unital separable AF $C^*$-algebra \cite{Bratteli1972}.

Elliott \cite{Elliott1976} classified unital separable AF $C^*$-algebras
by their ordered $K_0$-groups with order unit; see also \cite{Effros1981}.  For $\mathcal A_B$,
\begin{equation}\label{eq:af-dimension-group}
 \bigl(K_0(\mathcal A_B),K_0(\mathcal A_B)^+,[1]\bigr)
 \cong\varinjlim_n
 \bigl(\mathbb Z^{V_n},\mathbb Z_+^{V_n},p^{(n)};F_n\bigr).
\end{equation}

\begin{example}\label{ex:uhf-classification}
Let $B_d$ be the one-vertex diagram with $d\geq2$ edges between
successive levels.  By \Cref{subsec:factor-sequence-models},
$p_n(v_n)=d^n$ and the unique harmonic function is $\alpha_n(v_n)=1$.
Its $C^*$-completion is the UHF algebra of type $d^\infty$
\cite{Glimm1960}; for $d=2$, this is the CAR algebra.  Moreover,
\[
 \bigl(K_0(\mathcal A_{B_d}),K_0(\mathcal A_{B_d})^+,[1]\bigr)
 \cong(\mathbb Z[1/d],\mathbb Z[1/d]_+,1).
\]
For integers $d,e\geq2$, the $C^*$-algebras $\mathcal A_{B_d}$
and $\mathcal A_{B_e}$ are isomorphic if and only if $d$ and $e$
have the same prime divisors.
\end{example}

\subsection{Traces and von Neumann algebras}
\label{subsec:trace-ranks}

Write $T(\mathcal A_B)$ for the space of tracial states on
$\mathcal A_B$, equipped with the weak-* topology.

\begin{proposition}\label{prop:trace-harmonic}
For the block units $e_{n,v}$, the assignment
$\alpha_n^\tau(v)=\tau(e_{n,v})$ gives affine homeomorphisms
\[
 T(\mathcal A_B)\cong\mathcal H(B)
 \cong\operatorname{Prob}_{\mathrm{cent}}(X_B).
\]
Extreme traces correspond to extreme harmonic functions and tail-ergodic
central measures.
\end{proposition}

\begin{proof}
At each stage, a trace is determined by its block weights:
\begin{equation}\label{eq:finite-trace-weights}
 \tau((x_v)_v)=\sum_{v\in V_n}\alpha_n^\tau(v)
                    \operatorname{tr}_{p_n(v)}(x_v).
\end{equation}
The block multiplicities give
\[
 \alpha_n^\tau(v)=\sum_{w\in V_{n+1}}
       \frac{p_n(v)a_{w,v}}{p_{n+1}(w)}\alpha_{n+1}^\tau(w).
\]
Conversely, harmonic weights give compatible tracial states by the same
formula; the bound $|\tau(x)|\leq\|x\|$ extends them to $\mathcal A_B$.
The constructions are inverse and affine.  Evaluation on the block units
is weak-* continuous, so compactness gives a homeomorphism.
The assertions about central measures and extremality follow from
\Cref{prop:harmonic-central-measure}.
\end{proof}

Apply the tracial GNS construction from
\Cref{subsec:cstar-preliminaries} to $\mathcal A_B$ and
$\tau\in T(\mathcal A_B)$, using the notation
$M_\tau$ and $\widetilde\tau$ from \eqref{eq:gns-normal-trace}.

\begin{proposition}\label{prop:tracial-gns-factor}
For $\tau\in T(\mathcal A_B)$, the algebra $M_\tau$ is hyperfinite and
\[
 \tau\in\operatorname{ext}T(\mathcal A_B)
 \quad\Longleftrightarrow\quad M_\tau\text{ is a factor}.
\]
If $\tau$ is extreme and $(B,\alpha^\tau)$ is
$\alpha^\tau$-aperiodic, then $M_\tau\cong\mathcal R$, where
$\mathcal R$ is the hyperfinite $\mathrm{II}_1$ factor.
\end{proposition}

\begin{proof}
The finite-dimensional algebras $\pi_\tau(A_n)$ have weakly dense
union, so $M_\tau$ is hyperfinite.  Write $\alpha=\alpha^\tau$.
A central projection $z$ with $0<\widetilde\tau(z)<1$
decomposes $\tau$ into the distinct normalized traces obtained from
$z$ and $1-z$; distinctness follows from weak density of
$\pi_\tau(A(B,\mathbb C))$.
Conversely, if $\tau=t\tau_1+(1-t)\tau_2$ with $0<t<1$, set
\[
 h_n=\sum_{\alpha_n(v)>0}
 \frac{\alpha_n^{\tau_1}(v)}{\alpha_n(v)}\pi_\tau(e_{n,v}),
 \qquad 0\leq h_n\leq t^{-1}1.
\]
For $a\in A_m$ and $n\geq m$, harmonicity and the finite trace formula give
\[
 [h_n,\pi_\tau(a)]=0,\qquad
 \widetilde\tau(h_n\pi_\tau(a))=\tau_1(a).
\]
An ultraweak cluster point $h$ therefore lies in $Z(M_\tau)$ and
satisfies the same trace identity.  If $M_\tau$ is a factor, then
$h=\widetilde\tau(h)1=1$, so $\tau_1=\tau$ by norm density.

If $\tau$ is extreme, $M_\tau$ is a finite factor with separable
predual.  Suppose also that $(B,\alpha^\tau)$ is
$\alpha^\tau$-aperiodic.
For every $L$, aperiodicity gives a positive-weight block with
$p_n(v)\geq L$.  A minimal projection $e$ in that block has
\[
 0<\widetilde\tau(\pi_\tau(e))=\frac{\alpha_n(v)}{p_n(v)}\leq L^{-1}.
\]
These arbitrarily small positive projection traces exclude matrix
factors.  Separability and uniqueness of the hyperfinite $\mathrm{II}_1$ factor give $M_\tau\cong\mathcal R$
\cite{MurrayVonNeumann1943}.
\end{proof}

\begin{example}\label{ex:two-branch-gns}
For the two-branch diagram in \Cref{ex:positive-not-compact},
$\mathcal A_B$ is the direct sum of two CAR algebras.
Its traces are $\tau_t=t\tau_{\mathrm{CAR}}\oplus(1-t)\tau_{\mathrm{CAR}}$,
$0\leq t\leq1$, where $\tau_{\mathrm{CAR}}$ is the unique trace of the CAR algebra,
as follows from \Cref{prop:trace-harmonic}.
For $0<t<1$, the GNS representation acts on both summands, so
\[
 M_{\tau_t}\cong\mathcal R\oplus\mathcal R,\qquad
 \widetilde\tau_t(x,y)=t\tau_{\mathcal R}(x)+(1-t)\tau_{\mathcal R}(y),
\]
where $\tau_{\mathcal R}$ is the normalized trace on $\mathcal R$.
At $t=0$ or $t=1$, the zero-weight summand is removed, leaving
$M_{\tau_t}\cong\mathcal R$.
Although the original algebra has two nonzero proper closed ideals,
either endpoint trace annihilates one summand and gives a factor
as its GNS closure.
\end{example}

\begin{example}\label{ex:pascal-gns}
For the Pascal diagram in \Cref{ex:pascal-harmonic}, the traces
$\tau_t$ corresponding to $\alpha^t$ satisfy
\[
 \tau_t(e_{1,1})=t,\qquad M_{\tau_t}\cong\mathcal R
 \qquad(0<t<1)
\]
by \Cref{prop:tracial-gns-factor}.  Hence distinct traces on the same
$C^*$-algebra give abstractly isomorphic von Neumann algebras.
At $t=0,1$, the GNS representation is supported on the single path
of weight one, so its image and weak operator closure are $\mathbb C$.
\end{example}

\subsection{Rank completions and their comparison}
\label{subsec:completion-mechanism}

Let $M$ be a finite von Neumann algebra with faithful normal tracial
state $\tau$.
A closed densely defined operator $x$ is \emph{affiliated} with $M$
if $uxu^*=x$ for every unitary $u\in M'$.
These operators form a ring $\mathcal U(M)$, with sums and products
obtained by closing the corresponding operator sums and products.
Its matrix rank is
\[
 \rk_\tau(X)=(\operatorname{Tr}_r\otimes\tau)(p_X),
 \qquad X\in\Mat_{r\times s}(\mathcal U(M)),
\]
where $p_X\in\Mat_r(M)$ is the range projection.  For a positive
affiliated operator $a$, write $s(a)=\mathbf 1_{(0,\infty)}(a)$; thus
$p_X=s(XX^*)$.

The ring $\mathcal U(M)$ is complete in the rank metric
\cite[Lemma~2.2]{Thom2008}.  Spectral truncation gives
\[
 x\in\mathcal U(M),\quad e_k=\mathbf 1_{[0,k]}(|x|)
 \quad\Longrightarrow\quad
 xe_k\in M,\qquad \rk_\tau(x-xe_k)\leq\tau(1-e_k)\longrightarrow0.
\]
Indeed, $\|xe_k\|\leq k$, the right support of $x-xe_k$ is dominated
by $1-e_k$, and normality gives $\tau(1-e_k)\to0$.
Thus $\mathcal U(M)$ is the rank completion of $M$.

Ore localization provides an algebraic description of the same
ring $\mathcal U(M)$ by adjoining inverses to the non-zero-divisors
of $M$.

\begin{remark}[Ore localization]\label{rem:ore-localization}
Let
\[
 \Sigma_M=\Set*{a\in M\given s(a^*a)=s(aa^*)=1}
\]
be the set of elements that are neither left nor right zero divisors.
It satisfies the left and right Ore conditions, and
\begin{equation}\label{eq:affiliated-ore-localization}
 \mathcal U(M)\cong M\Sigma_M^{-1}
 \cong\Sigma_M^{-1}M
\end{equation}
\cite[Proposition~2.2]{Vas2005}.
Every $x\in\mathcal U(M)$ has the fraction representation
\[
 b=(1+|x|)^{-1}\in\Sigma_M,\qquad
 a=xb\in M,\qquad x=ab^{-1}.
\]
\end{remark}

\begin{proposition}
\label{prop:cstar-rank-affiliated}
Let $\mathcal A$ be a unital $C^*$-algebra with tracial state $\tau$,
and use the GNS notation from \eqref{eq:gns-normal-trace}.
The representation $\pi_\tau$ induces a rank-preserving unital
$*$-isomorphism
\begin{equation}\label{eq:cstar-rank-affiliated}
 \overline{\mathcal A}^{\,\rk_\tau}\cong\mathcal U(M_\tau).
\end{equation}
\end{proposition}

\begin{proof}
For $X\in\Mat_{r\times s}(\mathcal A)$, spectral bounded convergence
in \eqref{eq:trace-root-rank} and polar decomposition give
\begin{equation}\label{eq:trace-support-rank}
 \rk_\tau(X)=\rk_{\widetilde\tau}(\pi_\tau(X))
 =(\operatorname{Tr}_r\otimes\widetilde\tau)(p_{\pi_\tau(X)}).
\end{equation}
Indeed, $|\pi_\tau(X)|^{1/k}$ converges strongly to its support,
and the initial and final projections in the polar decomposition of
$\pi_\tau(X)$ have equal traces.
Faithfulness of $\widetilde\tau$ implies
$\ker(\rk_\tau)=\ker\pi_\tau$.
Thus the rank-zero quotient of $\mathcal A$ identifies isometrically
with $\pi_\tau(\mathcal A)$.
For $x\in M_\tau$ and $\varepsilon>0$, the noncommutative Lusin
theorem of Sait\^o \cite[Theorem~2]{Saito1967}, applied with normal functional
$\widetilde\tau$ and projection $1$, gives $a\in\pi_\tau(\mathcal A)$
and a projection $e\in M_\tau$ such that
\[
 (x-a)e=0,\qquad \widetilde\tau(1-e)<\varepsilon.
\]
The right support of $x-a=(x-a)(1-e)$ is dominated by $1-e$, whence
\begin{equation}\label{eq:lusin-rank-estimate}
 \rk_{\widetilde\tau}(x-a)\leq\widetilde\tau(1-e)<\varepsilon.
\end{equation}
Together with spectral truncation, this proves rank density in
$\mathcal U(M_\tau)$.  The isometric embedding therefore extends onto
this complete ring, giving \eqref{eq:cstar-rank-affiliated}.
Rectangular ranks extend by continuity, and the involution extends
isometrically because $\rk_{\widetilde\tau}(x^*)=\rk_{\widetilde\tau}(x)$.
\end{proof}

We now apply the preceding proposition to $\mathcal A_B$ with
the trace $\tau=\tau_\alpha$ associated with $\alpha\in\mathcal H(B)$.
Write $Q_\alpha=\overline A_\alpha(B,\mathbb C)$ for the rank
completion of the algebraic direct limit.

\begin{corollary}\label{prop:trace-induced-rank}
The trace-induced rank restricts to the weighted rank:
\begin{equation}\label{eq:trace-rank-finite-stage}
 \rk_\tau(X)=\sum_{v\in V_n}\frac{\alpha_n(v)}{p_n(v)}
                  \operatorname{rank}_{\mathbb C}(X_v)
 =\rk_\alpha(X),\qquad X\in\Mat_{r\times s}(A_n).
\end{equation}
The restriction of $\pi_\tau$ to $A(B,\mathbb C)$ extends to a
rank-preserving unital $*$-embedding
\begin{equation}\label{eq:core-affiliated-embedding}
 Q_\alpha\lhook\joinrel\longrightarrow\mathcal U(M_\tau),
\end{equation}
whose image is the rank closure of $\pi_\tau(A(B,\mathbb C))$.
\end{corollary}

\begin{proof}
Apply \eqref{eq:trace-root-rank} blockwise in
\eqref{eq:finite-trace-weights}.  Since
\[
 \lim_{k\to\infty}\operatorname{Tr}(|X_v|^{1/k})
 =\operatorname{rank}_{\mathbb C}(X_v),
\]
we obtain \eqref{eq:trace-rank-finite-stage}.
Together with \Cref{prop:cstar-rank-affiliated}, this shows that
$\pi_\tau$ preserves the weighted matrix rank.
It therefore induces an isometry
on the rank-zero quotient, which extends to its completion.
The image of this extension is the stated rank closure.
\end{proof}

The constructions give the commuting diagram
\begin{equation}\label{eq:two-completion-chains}
 \begin{tikzpicture}[baseline=(current bounding box.center),
   >=Stealth, every node/.style={inner sep=3pt}]
  \node (alg) at (0,1.9) {$A(B,\mathbb C)$};
  \node (af) at (3.8,1.9) {$\mathcal A_B$};
  \node (vn) at (7.8,1.9) {$M_\tau$};
  \node (rank) at (0,0) {$Q_\alpha$};
  \node (aff) at (7.8,0) {$\mathcal U(M_\tau)$};
  \draw[->] (alg) -- node[above] {$\|\cdot\|$} (af);
  \draw[->] (af) -- node[above] {$\pi_\tau,\ \mathrm{WOT}$} (vn);
  \draw[->] (alg) -- node[left] {$\rk_\alpha$} (rank);
  \draw[->] (af) -- node[above,sloped] {$\rk_\tau$} (aff);
  \draw[->] (vn) -- node[right] {$\rk_{\widetilde\tau}$} (aff);
  \draw[->] (rank) -- node[below] {$\mathrm{inclusion}$} (aff);
 \end{tikzpicture}
\end{equation}
The arrows are the canonical homomorphisms.  The norm and rank labels
specify the corresponding completions, while $\pi_\tau(\mathcal A_B)$
is weak operator dense in $M_\tau$.
Identifying $Q_\alpha$ with its image under
\eqref{eq:core-affiliated-embedding}, the rank closures inside
$\mathcal U(M_\tau)$ satisfy
\begin{equation}\label{eq:rank-closure-comparison}
 \begin{aligned}
 \overline{\pi_\tau(A(B,\mathbb C))}^{\,\rk_{\widetilde\tau}}
 &=Q_\alpha\subseteq\mathcal U(M_\tau),\\
 \overline{\pi_\tau(\mathcal A_B)}^{\,\rk_{\widetilde\tau}}
 &=\overline{M_\tau}^{\,\rk_{\widetilde\tau}}
 =\mathcal U(M_\tau).
 \end{aligned}
\end{equation}
\begin{example}
\label{ex:proper-core-rank-closure}
The inclusion in \eqref{eq:rank-closure-comparison} can be proper.
Let $B=B_2$ be the diagram from \Cref{ex:uhf-classification},
with unique harmonic function $\alpha$.  Use the tensor-product presentation
\[
 A_n=\bigotimes_{k=1}^n\Mat_2(\mathbb C),\qquad
 A(B,\mathbb C)=\bigcup_n A_n,\qquad
 \mathcal A_B=\bigotimes_{k\geq1}\Mat_2(\mathbb C).
\]
The maps $a\mapsto a\otimes1_2$ and $a\mapsto1_2\otimes a$ are
conjugate by permutations.
The recursive conjugacy in
\Cref{subsec:associated-algebra} identifies their direct systems,
preserving the normalized ranks and traces.
Let $\tau$ be the faithful product trace and identify $\mathcal A_B$
with its GNS image in $M_\tau\cong\mathcal R$.

Let $e_k$ be $\diag(0,1)$ in the $k$th tensor factor and the identity
elsewhere.  Put
\[
 x=\sum_{k\geq1}2^{-k}e_k\in\mathcal A_B,\qquad
 x_n=\sum_{k=1}^n2^{-k}e_k\in A_n,\qquad
 y_n=x-x_n=\sum_{k>n}2^{-k}e_k.
\]
The series converges in norm, with $\|x-x_n\|\leq2^{-n}$.
For a self-adjoint element $z\in M_\tau$, its \emph{spectral
 distribution} with respect to $\widetilde\tau$ is the probability
measure
\[
 \mu_z(E)=\widetilde\tau\bigl(\mathbf 1_E(z)\bigr)
 \qquad(E\subseteq\mathbb R\text{ Borel}),
\]
where $\mathbf 1_E(z)$ is the spectral projection of $z$ associated
with $E$.
The element $y_n$ belongs to the remaining tensor factors and commutes
with $A_n$.  Its spectral distribution $\mu_{y_n}$ is nonatomic:
the binary coordinates are independent with equal probabilities,
each sequence has measure zero, and each real number has at most
two binary expansions.
For $a\in A_n$, the tensor decomposition and the spectral theorem for
$y_n$ represent $x-a$ by the matrix-valued function
\[
 t\longmapsto(x_n-a)+t1_{2^n}.
\]
The trace on this matrix-valued representation is
$\operatorname{tr}_{2^n}\otimes\mu_{y_n}$.  Consequently,
\[
 \rk_\tau(x-a)=\int_{\mathbb R}
 2^{-n}\operatorname{rank}_{\mathbb C}
 \bigl((x_n-a)+t1_{2^n}\bigr)\,d\mu_{y_n}(t).
\]
The determinant is a monic polynomial of degree $2^n$, so it vanishes
at only finitely many points.  Since $\mu_{y_n}$ is nonatomic,
the matrix is invertible almost everywhere.  The integrand is
therefore $1$ almost everywhere, giving $\rk_\tau(x-a)=1$.
This argument applies to every $a\in A_n$, without a self-adjointness
assumption.
As every element of $A(B,\mathbb C)$ lies in some $A_n$,
\begin{equation}\label{eq:proper-core-distance}
 \inf_{a\in A(B,\mathbb C)}\rk_\tau(x-a)=1.
\end{equation}
Although $x_n\to x$ in norm, the element
$x\in\mathcal A_B\subseteq\mathcal U(M_\tau)$ does not belong
to $Q_\alpha$, proving that the canonical embedding in
\Cref{prop:trace-induced-rank} is not surjective for this diagram.

The same binary-coordinate argument shows that $\mu_x$ has no atom
at zero.  Thus the kernel projection of $x$ is zero, and spectral
calculus gives an inverse $x^{-1}$ in $\mathcal U(M_\tau)$.
On the event that its first $m$ binary coordinates vanish,
$0<x\leq2^{-m}$ almost everywhere, and this event has trace $2^{-m}$.
Hence $x^{-1}$ is unbounded.  A bounded $y$ with $xyx=x$ would satisfy
$y=x^{-1}$, a contradiction.  Thus $M_\tau$ and $\mathcal A_B$ are
not regular as rings.
\end{example}

Returning to the general setting, $\mathcal U(M_\tau)$ is
$*$-regular: if $x=u|x|$, then
\[
 x^\dagger=g(|x|)u^*,\qquad
 g(0)=0,\quad g(t)=t^{-1}\ (t>0),\qquad xx^\dagger x=x.
\]
Here $*$-regularity means von Neumann regularity with a proper
involution: $x^*x=0$ implies $x=0$.
The element $x^\dagger$ is the \emph{relative inverse} of $x$,
uniquely characterized by
\[
 xx^\dagger x=x,\qquad x^\dagger xx^\dagger=x^\dagger,
 \qquad (xx^\dagger)^*=xx^\dagger,\qquad
 (x^\dagger x)^*=x^\dagger x.
\]
For $X\in\Mat_{r\times s}(\mathcal U(M_\tau))$, the same construction
gives $XX^\dagger=p_X$, and hence
$X\mathcal U(M_\tau)^{s\times1}=p_X\mathcal U(M_\tau)^{r\times1}$.
Hence the projective-module dimension from
\Cref{rem:pseudo-rank-dimension} satisfies
\begin{equation}\label{eq:affiliated-rank-dimension}
 \dim_{\widetilde\tau}\bigl(X\mathcal U(M_\tau)^{s\times1}\bigr)
 =(\operatorname{Tr}_r\otimes\widetilde\tau)(p_X)=\rk_{\widetilde\tau}(X).
\end{equation}

\begin{proposition}\label{prop:regular-simple-comparison}
The ring $Q_\alpha$ is $*$-regular, and
\[
 \begin{aligned}
 \alpha\text{ extreme}
 &\iff\tau\text{ extreme}\iff M_\tau\text{ a factor}\\
 &\iff Q_\alpha\text{ simple}
 \iff\mathcal U(M_\tau)\text{ simple}.
 \end{aligned}
\]
Here simplicity concerns all two-sided ideals.
\end{proposition}

\begin{proof}
The algebra $A(B,\mathbb C)$ is regular by blockwise generalized
inverses.  Its rank completion remains regular \cite[Theorem~19.6]{Goodearl1991}.
The embedding \eqref{eq:core-affiliated-embedding} gives $Q_\alpha$
a proper involution, so $Q_\alpha$ is $*$-regular.
The equivalence between extremality of $\tau$ and factoriality of
$M_\tau$ is \Cref{prop:tracial-gns-factor}.
By \Cref{prop:trace-harmonic} and the division-ring discussion in
\Cref{subsec:coefficient-ranks}, extremality of
$\tau$, of $\alpha$, and of the scalar pseudo-rank on $A(B,\mathbb C)$ are
equivalent.  The last condition is equivalent to simplicity of its
completion \cite[Theorem~19.14]{Goodearl1991}; see also
\cite[Section~2]{AraClaramunt2018}.

A nontrivial central projection of $M_\tau$ generates a proper ideal
of $\mathcal U(M_\tau)$.  If $M_\tau$ is a factor, a nonzero ideal in
$\mathcal U(M_\tau)$ contains a nonzero range projection $p$, obtained
from an element and its relative inverse.  Trace comparison in a finite
factor gives a finite partition $1=\sum_iq_i$, with each $q_i$ equivalent
to a subprojection of $p$ \cite{MurrayVonNeumann1943}.
Thus every $q_i$ belongs to the ideal, which therefore contains $1$.
This proves the last equivalence.
\end{proof}

\begin{corollary}
\label{cor:parallel-completions}
If $\tau$ is extreme and $(B,\alpha^\tau)$ is $\alpha^\tau$-aperiodic, then
\begin{equation}\label{eq:parallel-completions}
 M_\tau\cong\mathcal R,\qquad Q_{\alpha^\tau}\cong\mathcal M_{\mathbb C},\qquad
 \overline{\mathcal A_B}^{\,\rk_\tau}\cong\mathcal U(\mathcal R).
\end{equation}
The second isomorphism preserves the unital $\mathbb C$-algebra structure
and all rectangular ranks.  Composing its inverse with
\eqref{eq:core-affiliated-embedding} gives a rank-preserving embedding
\begin{equation}\label{eq:continuous-ring-affiliated-embedding}
 \mathcal M_{\mathbb C}\lhook\joinrel\longrightarrow\mathcal U(\mathcal R).
\end{equation}
\end{corollary}

\begin{proof}
By \Cref{prop:tracial-gns-factor}, $M_\tau\cong\mathcal R$.
\Cref{thm:main} gives $Q_{\alpha^\tau}\cong\mathcal M_{\mathbb C}$.
The remaining assertions follow from
\Cref{prop:cstar-rank-affiliated,prop:trace-induced-rank}, since the
trace-preserving isomorphism $M_\tau\cong\mathcal R$ extends to
affiliated operators through their polar and spectral decompositions.
\end{proof}

\begin{example}\label{ex:uhf-completions}
For the diagrams $B_d$ in \Cref{ex:uhf-classification},
\Cref{prop:trace-harmonic} gives a unique trace $\tau$.
The unique harmonic function $\alpha$ is extreme, and
$(B_d,\alpha)$ is $\alpha$-aperiodic since $p_n(v_n)=d^n\to\infty$.  Thus
\Cref{cor:parallel-completions} gives
\[
 \overline A_\alpha(B_d,\mathbb C)\cong\mathcal M_{\mathbb C},\qquad
 M_\tau\cong\mathcal R,\qquad
 \overline{\mathcal A_{B_d}}^{\,\rk_\tau}\cong\mathcal U(\mathcal R).
\]
These isomorphism types are independent of $d$, whereas the
$C^*$-isomorphism type of $\mathcal A_{B_d}$ is determined by the
set of prime divisors of $d$; see \Cref{ex:uhf-classification}.
\end{example}

For the two-branch diagram in \Cref{ex:two-branch-gns},
\Cref{prop:cstar-rank-affiliated} gives
\[
 \overline{\mathcal A_B}^{\,\rk_{\tau_t}}
 \cong\mathcal U(\mathcal R)\times\mathcal U(\mathcal R)
 \qquad(0<t<1),
\]
with rank weights $t$ and $1-t$; at either endpoint it gives
$\mathcal U(\mathcal R)$.
For the Pascal diagram in \Cref{ex:pascal-gns}, the same proposition gives
\[
 \overline{\mathcal A_B}^{\,\rk_{\tau_t}}
 \cong
 \begin{cases}
  \mathcal U(\mathcal R),&0<t<1,\\
  \mathbb C,&t=0,1.
 \end{cases}
\]

\subsection{Measure and algebraic closures}
\label{subsec:other-closures}

Put $M=M_\tau$, with its faithful normal tracial state
$\widetilde\tau$, and $U=\mathcal U(M)$.
In this finite-trace setting every affiliated operator is
$\widetilde\tau$-measurable, so $U$ is also the algebra
$L^0(M,\widetilde\tau)$ of measurable operators used in
noncommutative integration \cite{Nelson1974,Nayak2019}.
For $\varepsilon,\delta>0$, set
\[
 V(\varepsilon,\delta)=
 \Set*{x\in U\given
 \begin{gathered}
 \text{there is a projection }e\in M\text{ with }
 \widetilde\tau(1-e)<\delta,\\
 xe\in M\text{ and }\|xe\|<\varepsilon
 \end{gathered}}.
\]
These sets form a neighbourhood basis of zero for the
\emph{measure topology} on $U$.
Thus $x_n\to x$ in measure if, for every $\varepsilon,\delta>0$,
$x_n-x\in V(\varepsilon,\delta)$ for all sufficiently large $n$.
The projection $e$ allows one to discard a part of trace less than
$\delta$ and require norm error less than $\varepsilon$ on the
remaining part.  For $M=L^\infty(X,\mu)$ on a probability space,
this is the usual convergence in measure on $U=L^0(X,\mu)$.

The measure topology makes $U$ a complete metrizable topological
$*$-algebra, and $M$ is dense in $U$
\cite{Nelson1974,Nayak2019}.
For a subalgebra $T\subseteq U$, its \emph{measure completion}
means its completion for the induced measure topology; it identifies
canonically with its measure closure in $U$.
In the present setting,
\begin{equation}\label{eq:common-measure-completion}
 \overline{\pi_\tau(A(B,\mathbb C))}^{\,\mathrm{measure}}
 =\overline{\pi_\tau(\mathcal A_B)}^{\,\mathrm{measure}}
 =\overline M^{\,\mathrm{measure}}=U.
\end{equation}
To see this, put $S=\pi_\tau(A(B,\mathbb C))$ and
$\|y\|_2=\widetilde\tau(y^*y)^{1/2}$ for $y\in M$.
Cyclicity of the GNS vector and norm density of
$A(B,\mathbb C)$ in $\mathcal A_B$ imply that $S$ is dense in
$M$ for $\|\cdot\|_2$.
For $y\in M$, the spectral estimate
\[
 \widetilde\tau\bigl(\mathbf 1_{(\varepsilon/2,\infty)}(|y|)\bigr)
 \leq4\varepsilon^{-2}\|y\|_2^2
\]
shows that $2$-norm convergence implies convergence in measure:
take $e=\mathbf 1_{[0,\varepsilon/2]}(|y|)$, so that
$\|ye\|\leq\varepsilon/2<\varepsilon$.
Finally, for $x\in U$, the spectral truncations
$x\mathbf 1_{[0,n]}(|x|)$ lie in $M$ and converge to $x$ in measure.
This proves \eqref{eq:common-measure-completion}.

Rank convergence implies convergence in measure: for $y\in U$,
the kernel projection $e$ satisfies $ye=0$ and
$\widetilde\tau(1-e)=\rk_{\widetilde\tau}(y)$.
The converse fails, since $n^{-1}1\to0$ in measure while
$\rk_{\widetilde\tau}(n^{-1}1)=1$.
In particular, the measure closure of $S$ is always $U$, whereas
its rank closure $Q_\alpha$ can be a proper subring of $U$.

Let $S$ be a unital $*$-subring of $U=\mathcal U(M_\tau)$.
Its \emph{rational closure} $\operatorname{Rat}_U(S)$ is the
smallest subring $T$ with $S\subseteq T\subseteq U$ such that,
whenever $X\in\Mat_n(T)$ is invertible in $\Mat_n(U)$,
all entries of $X^{-1}$ belong to $T$.
Its \emph{$*$-regular closure} $\operatorname{Reg}^{*}_U(S)$
is the smallest $*$-regular subring of $U$ containing $S$.
It is obtained by repeatedly adjoining relative inverses and
forming the generated $*$-subring
\cite[Proposition~6.2]{AraGoodearl2017}.

Every unital regular subring $T\subseteq U$ is rationally closed.
Indeed, if $X\in\Mat_n(T)$ is invertible over $U$, regularity of
$\Mat_n(T)$ gives $Y\in\Mat_n(T)$ with $XYX=X$.
Multiplication by $X^{-1}$ on both sides gives $Y=X^{-1}$.
Consequently,
\begin{equation}\label{eq:algebraic-closure-inclusions}
 S\subseteq\operatorname{Rat}_U(S)
 \subseteq\operatorname{Reg}^{*}_U(S)\subseteq U.
\end{equation}
For $S=\pi_\tau(A(B,\mathbb C))$, relative inverses can be
computed within the finite-dimensional stages, so $S$ is already
$*$-regular.  By \Cref{prop:regular-simple-comparison}, its rank
closure $Q_\alpha$ is also $*$-regular.  Hence
\begin{equation}\label{eq:core-algebraic-closures}
 \begin{aligned}
 \operatorname{Rat}_U(S)&=\operatorname{Reg}^{*}_U(S)=S,
 &\overline S^{\,\rk_{\widetilde\tau}}&=Q_\alpha,\\
 \operatorname{Rat}_U(Q_\alpha)&=
 \operatorname{Reg}^{*}_U(Q_\alpha)=Q_\alpha.
 \end{aligned}
\end{equation}
On the other hand, by \Cref{rem:ore-localization}, $U$ is the
classical ring of quotients of $M_\tau$: its elements are fractions
$ab^{-1}$ with $a,b\in M_\tau$
and $b$ a non-zero-divisor.  Thus adjoining inverses already gives
\[
 \operatorname{Rat}_U(M_\tau)=\operatorname{Reg}^{*}_U(M_\tau)=U.
\]


\begin{thebibliography}{99}

\bibitem{Anderson2017}
A.~Anderson,
\newblock The Fra\"iss\'e limit of matrix algebras with the rank metric,
\newblock preprint, arXiv:1712.04431 (2017).
\newblock \url{https://arxiv.org/abs/1712.04431}.

\bibitem{AraClaramunt2018}
P.~Ara and J.~Claramunt,
\newblock Uniqueness of the von Neumann continuous factor,
\newblock \emph{Canad. J. Math.} \textbf{70} (2018), 961--982.
\newblock \url{https://doi.org/10.4153/CJM-2018-010-3}.

\bibitem{AraGoodearl2017}
P.~Ara and K.~R.~Goodearl,
\newblock The realization problem for some wild monoids and the Atiyah problem,
\newblock \emph{Trans. Amer. Math. Soc.} \textbf{369} (2017), 5665--5710.
\newblock \url{https://doi.org/10.1090/tran/6889}.

\bibitem{AtiyahMacdonald1969}
M.~F.~Atiyah and I.~G.~Macdonald,
\newblock \emph{Introduction to Commutative Algebra},
\newblock Addison-Wesley, Reading, MA, 1969.

\bibitem{BezuglyiDudkoKarpel2026}
S.~Bezuglyi, A.~Dudko, and O.~Karpel,
\newblock Measures and dynamics on Pascal--Bratteli diagrams,
\newblock \emph{J. Math. Phys. Anal. Geom.} \textbf{22} (2026), no.~1,
  3--21.
\newblock \url{https://doi.org/10.15407/mag22.01.01}.

\bibitem{BezuglyiKarpel2016}
S.~Bezuglyi and O.~Karpel,
\newblock Bratteli diagrams: structure, measures, dynamics,
\newblock in \emph{Dynamics and Numbers}, Contemp. Math., vol.~669,
  Amer. Math. Soc., Providence, RI, 2016, pp.~1--36.
\newblock \url{https://doi.org/10.1090/conm/669/13421}.

\bibitem{BezuglyiKwiatkowskiMedynetsSolomyak2013}
S.~Bezuglyi, J.~Kwiatkowski, K.~Medynets, and B.~Solomyak,
\newblock Finite rank Bratteli diagrams: structure of invariant measures,
\newblock \emph{Trans. Amer. Math. Soc.} \textbf{365} (2013), no.~5,
  2637--2679.
\newblock \url{https://doi.org/10.1090/S0002-9947-2012-05744-8}.

\bibitem{Bratteli1972}
O.~Bratteli,
\newblock Inductive limits of finite dimensional $C^*$-algebras,
\newblock \emph{Trans. Amer. Math. Soc.} \textbf{171} (1972), 195--234.
\newblock \url{https://doi.org/10.1090/S0002-9947-1972-0312282-2}.

\bibitem{Effros1981}
E.~G.~Effros,
\newblock \emph{Dimensions and $C^*$-Algebras},
\newblock CBMS Regional Conference Series in Mathematics, vol.~46,
  Amer. Math. Soc., Providence, RI, 1981.
\newblock \url{https://doi.org/10.1090/cbms/046}.

\bibitem{Elek2013}
G.~Elek,
\newblock Connes embeddings and von Neumann regular closures of amenable group
  algebras,
\newblock \emph{Trans. Amer. Math. Soc.} \textbf{365} (2013), 3019--3039.
\newblock \url{https://doi.org/10.1090/S0002-9947-2012-05687-X}.

\bibitem{Elliott1976}
G.~A.~Elliott,
\newblock On the classification of inductive limits of sequences of
  semisimple finite-dimensional algebras,
\newblock \emph{J. Algebra} \textbf{38} (1976), 29--44.
\newblock \url{https://doi.org/10.1016/0021-8693(76)90242-8}.

\bibitem{Glimm1960}
J.~G.~Glimm,
\newblock On a certain class of operator algebras,
\newblock \emph{Trans. Amer. Math. Soc.} \textbf{95} (1960), 318--340.
\newblock \url{https://doi.org/10.1090/S0002-9947-1960-0112057-5}.

\bibitem{Goodearl1991}
K.~R.~Goodearl,
\newblock \emph{Von Neumann Regular Rings},
\newblock second ed., Krieger Publishing Co., Malabar, FL, 1991.
\newblock ISBN 978-0-89464-632-4.

\bibitem{Halperin1968}
I.~Halperin,
\newblock Von Neumann's manuscript on inductive limits of regular rings,
\newblock \emph{Canad. J. Math.} \textbf{20} (1968), 477--483.
\newblock \url{https://doi.org/10.4153/CJM-1968-045-0}.

\bibitem{HermanPutnamSkau1992}
R.~H.~Herman, I.~F.~Putnam, and C.~F.~Skau,
\newblock Ordered Bratteli diagrams, dimension groups and topological dynamics,
\newblock \emph{Internat. J. Math.} \textbf{3} (1992), no.~6, 827--864.
\newblock \url{https://doi.org/10.1142/S0129167X92000382}.

\bibitem{HungLi2023}
T.~F.~Hung and H.~Li,
\newblock Malcolmson semigroups,
\newblock \emph{J. Algebra} \textbf{623} (2023), 193--233.
\newblock \url{https://doi.org/10.1016/j.jalgebra.2023.01.031}.

\bibitem{JaikinLopez2020}
A.~Jaikin-Zapirain and D.~L\'opez-\'Alvarez,
\newblock On the space of Sylvester matrix rank functions,
\newblock preprint, arXiv:2012.15844 (2020).
\newblock \url{https://arxiv.org/abs/2012.15844}.

\bibitem{JaikinZapirain2019}
A.~Jaikin-Zapirain,
\newblock The base change in the Atiyah and the L\"uck approximation conjectures,
\newblock \emph{Geom. Funct. Anal.} \textbf{29} (2019), no.~2, 464--538.

\bibitem{JiangLi2021}
B.~Jiang and H.~Li,
\newblock Sylvester rank functions for amenable normal extensions,
\newblock \emph{J. Funct. Anal.} \textbf{280} (2021), no.~6, 108913.
\newblock \url{https://doi.org/10.1016/j.jfa.2020.108913}.

\bibitem{Li2021}
H.~Li,
\newblock Bivariant and extended Sylvester rank functions,
\newblock \emph{J. Lond. Math. Soc.} (2) \textbf{103} (2021), no.~1, 222--249.
\newblock \url{https://doi.org/10.1112/jlms.12372}.

\bibitem{LopezAlvarez2021}
D.~L\'opez \'Alvarez,
\newblock \emph{Sylvester rank functions, epic division rings and the strong
  Atiyah conjecture for locally indicable groups},
\newblock Ph.D. thesis, Universidad Aut\'onoma de Madrid, 2021.
\newblock \url{https://hdl.handle.net/10486/696201}.

\bibitem{Malcolmson1980}
P.~Malcolmson,
\newblock Determining homomorphisms to skew fields,
\newblock \emph{J. Algebra} \textbf{64} (1980), no.~2, 399--413.
\newblock \url{https://doi.org/10.1016/0021-8693(80)90153-2}.

\bibitem{MurrayVonNeumann1936}
F.~J.~Murray and J.~von Neumann,
\newblock On rings of operators,
\newblock \emph{Ann.\ of Math.} (2) \textbf{37} (1936), no.~1, 116--229.
\newblock \url{https://doi.org/10.2307/1968693}.

\bibitem{MurrayVonNeumann1943}
F.~J.~Murray and J.~von Neumann,
\newblock On rings of operators. IV,
\newblock \emph{Ann.\ of Math.} (2) \textbf{44} (1943), 716--808.
\newblock \url{https://doi.org/10.2307/1969107}.

\bibitem{Nayak2019}
S.~Nayak,
\newblock On Murray-von Neumann algebras---I: Topological, order-theoretic
  and analytical aspects,
\newblock \emph{Banach J. Math. Anal.} \textbf{15} (2021), article~45.
\newblock \url{https://doi.org/10.1007/s43037-021-00129-7}.
\newblock Revised version: \url{https://arxiv.org/abs/1911.01978v3}.

\bibitem{Nelson1974}
E.~Nelson,
\newblock Notes on non-commutative integration,
\newblock \emph{J. Funct. Anal.} \textbf{15} (1974), 103--116.
\newblock \url{https://doi.org/10.1016/0022-1236(74)90014-7}.

\bibitem{Putnam2018}
I.~F.~Putnam,
\newblock \emph{Cantor Minimal Systems},
\newblock University Lecture Series, vol.~70, American Mathematical Society,
  Providence, RI, 2018.
\newblock ISBN 978-1-4704-4115-9.

\bibitem{Saito1967}
K.~Sait\^o,
\newblock Non-commutative extension of Lusin's theorem,
\newblock \emph{T\^ohoku Math. J.} (2) \textbf{19} (1967), 332--340.
\newblock \url{https://doi.org/10.2748/tmj/1178243283}.

\bibitem{Schneider2024}
F.~M.~Schneider,
\newblock Group von Neumann algebras, inner amenability, and unit groups
  of continuous rings,
\newblock \emph{Int. Math. Res. Not. IMRN} (2024), 6422--6446.
\newblock \url{https://doi.org/10.1093/imrn/rnad181}.

\bibitem{Schneider2026}
F.~M.~Schneider,
\newblock Geometric properties of unit groups of von Neumann's continuous rings,
\newblock \emph{J. Algebra}, to appear; final version, arXiv:2509.01556v3 (2026).
\newblock \url{https://doi.org/10.1016/j.jalgebra.2026.08.005}.

\bibitem{SchneiderThom2026}
F.~M.~Schneider and A.~Thom,
\newblock Unitary representations and von Neumann's continuous geometries,
\newblock preprint, arXiv:2604.26104 (2026).
\newblock \url{https://arxiv.org/abs/2604.26104}.

\bibitem{Schofield1985}
A.~H.~Schofield,
\newblock \emph{Representations of Rings over Skew Fields},
\newblock London Math. Soc. Lecture Note Ser., vol.~92,
  Cambridge University Press, Cambridge, 1985.
\newblock \url{https://doi.org/10.1017/CBO9780511661914}.

\bibitem{Thom2008}
A.~Thom,
\newblock Sofic groups and diophantine approximation,
\newblock \emph{Comm. Pure Appl. Math.} \textbf{61} (2008), no.~8,
  1155--1171.
\newblock \url{https://doi.org/10.1002/cpa.20217}.

\bibitem{Vas2005}
L.~Va\v{s},
\newblock Torsion theories for finite von Neumann algebras,
\newblock \emph{Comm. Algebra} \textbf{33} (2005), no.~3, 663--688.
\newblock \url{https://doi.org/10.1081/AGB-200049871}.

\bibitem{Vershik2014}
A.~M.~Vershik,
\newblock The problem of describing central measures on the path spaces of
  graded graphs,
\newblock \emph{Funct. Anal. Appl.} \textbf{48} (2014), no.~4, 256--271.
\newblock \url{https://doi.org/10.1007/s10688-014-0069-5}.

\bibitem{vonNeumann1936Geometry}
J.~von Neumann,
\newblock Continuous geometry,
\newblock \emph{Proc. Natl. Acad. Sci. USA} \textbf{22} (1936), no.~2, 92--100.
\newblock \url{https://doi.org/10.1073/pnas.22.2.92}.

\end{thebibliography}
\end{document}